\documentclass[11pt,a4paper]{article}

\usepackage[left=2.45cm, top=2.45cm,bottom=2.45cm,right=2.45cm]{geometry}
\usepackage{mathtools,amssymb,amsthm,mathrsfs,calc,graphicx,xcolor,cleveref,dsfont,tikz,pgfplots,bm}
\usepackage[british]{babel}
\usepackage{amsfonts}              
\usepackage[T1]{fontenc}
\numberwithin{equation}{section}
\numberwithin{figure}{section}

\newtheorem {theorem}{Theorem}[section]
\newtheorem {proposition}[theorem]{Proposition}
\newtheorem {lemma}[theorem]{Lemma}
\newtheorem {corollary}[theorem]{Corollary}

{\theoremstyle{definition}
\newtheorem{definition}[theorem]{Definition}
}
{\theoremstyle{theorem}
\newtheorem {remark}[theorem]{Remark}
\newtheorem {example}[theorem]{Example}
}

\newcommand{\Cov}{\operatorname{Cov}}

\newcommand{\kappad}{\kappa}
\newcommand{\overkappad}{\overline{\kappa}}
\newcommand{\overnud}{\overline{\nu_{d-1}}}

\def\ba{\begin{array}}
\def\ea{\end{array}}
\def\bea{\begin{eqnarray} \label}
\def\eea{\end{eqnarray}}
\def\be{\begin{equation} \label}
\def\ee{\end{equation}}
\def\bit{\begin{itemize}}
\def\eit{\end{itemize}}
\def\ben{\begin{enumerate}}
\def\een{\end{enumerate}}

\def\lan{\langle}
\def\ran{\rangle}

\def\EE{\mathbb{E}}

\def\KK{\mathbb{K}}

\def\MM{\mathbb{M}}
\def\NN{\mathbb{N}}
\def\PP{\mathbb{P}}

\def\RR{\mathbb{R}}
\def\RRd1{\mathbb{R}^{d+1}}
\def\SS{\mathbb{S}}
\def\SSd{\mathbb{S}^d}
\def\TT{\mathbb{T}}

\def\b{\beta}

\def\s{\sigma}

\def\Sig{\Sigma}
\def\O{\Omega}

\def\sfM{{\sf M}}
\def\sfF{{\sf F}}

\def\bE{\mathbf{E}}
\def\bF{\mathbf{F}}

\def\bM{\mathbf{M}}
\def\bP{\mathbf{P}}
\def\bQ{\mathbf{Q}}

\def\blk{\langle}
\def\brk{\rangle}

\def\Var{\mathbf{V}\textup{ar}\,}
\def\Cov{\mathbf{C}\textup{ov}}

\def\cA{\mathcal{A}}
\def\cB{\mathcal{B}}
\def\cC{\mathcal{C}}
\def\cD{\mathcal{D}}

\def\cF{\mathcal{F}}
\def\cG{\mathcal{G}}

\def\cH{\mathcal{H}}

\def\cL{\mathcal{L}}
\def\cM{\mathcal{M}}

\def\cT{\mathcal{T}}

\def\sF{\mathscr{F}}

\def\sN{\mathscr{N}}

\def\dint{\textup{d}}

\def\SO{\textup{SO}}
\def\Grass{\textup{G}}
\def\filtration{\mathcal{Y}}
\def\sigmaalgebra{\mathcal{G}}
\def\Tess{\textup{Tess}}

\begin{document}

\title{\bfseries Splitting tessellations in spherical spaces}

\author{Daniel Hug\footnotemark[1]\;\; and Christoph Th\"ale\footnotemark[2]}

\date{}
\renewcommand{\thefootnote}{\fnsymbol{footnote}}
\footnotetext[1]{Karlsruhe Institute of Technology (KIT), D-76128 Karlsruhe, Germany. Email: daniel.hug@kit.edu}

\footnotetext[2]{
Ruhr University Bochum, D-44780 Bochum, Germany. Email: christoph.thaele@rub.de}

\maketitle

\begin{abstract}
The concept of splitting tessellations and splitting tessellation processes in spherical spaces of dimension $d\geq 2$ is introduced. Expectations, variances and covariances of spherical curvature measures induced by a splitting tessellation are studied using tools from spherical integral geometry. Also the spherical pair-correlation function of the $(d-1)$-dimensional Hausdorff measure is computed explicitly and compared to its analogue for Poisson great hypersphere tessellations. Finally, the typical cell distribution and the distribution of the typical spherical maximal face of any dimension $k\in\{1,\ldots,d-1\}$ are expressed as mixtures of the related distributions of Poisson great hypersphere tessellations. This in turn is used to determine the expected length and the precise birth time distribution of the typical  maximal spherical segment of a splitting tessellation.
\end{abstract}

\tableofcontents

\section{Introduction}

Random tessellations are among the central topics considered in stochastic geometry and constitute a key model for numerous applications, see \cite{CSKM,SW} as well as the references cited therein. While random tessellations in Euclidean spaces have been explored in great detail, their non-Euclidean counterparts are much less intensively studied in the literature. On the other hand, in recent years  non-Euclidean models in stochastic geometry have attracted increasing attention, in particular in random geometric structures on manifolds, see e.g.\ \cite{ArbeiterZaehleMosaics,BaranyHugReitznerSchneider,BrauchartEtAl,DeussHoerrmannThaele,HugReichenbacher,HugSchneider2016,
KabluchkoMarynychTemesThaele,MaeharaMartini17,MaeharaMartini18,MarinucciPeccati,MarinucciRossi,PenroseYukichMf,ReddyVadlamaniYog}. More specifically, the papers \cite{BrauchartEtAl,MaeharaMartini17,MaeharaMartini18} deal with spherical convex hulls, \cite{BaranyHugReitznerSchneider,KabluchkoMarynychTemesThaele} are concerned with spherical convex hulls of random points on half-spheres, the work \cite{PenroseYukichMf} considers central limit theorems for point process statistics of point processes on manifolds, \cite{ReddyVadlamaniYog} is devoted to the study of random systems on Cayley graphs, and \cite{MarinucciPeccati,MarinucciRossi} explores  geometric aspects of random fields on the sphere. Further results for tessellations of the $d$-dimensional unit sphere by great hyperspheres have been obtained in \cite{ArbeiterZaehleMosaics} and, more recently, in \cite{HugReichenbacher,HugSchneider2016}, which generalize at the same time some mean value computations for tessellations generated by great circles on the $2$-dimensional sphere $\SS^2$ in \cite{MilesSphere}. On the other hand, in \cite{Calkaetc}  an analysis of Voronoi tessellations in rather general Riemannian manifolds is initiated. Various other recent contributions combine spherical (or conical) integral geometry and probabilistic reasoning (see, for instance, \cite{ALMT,GNP,MT}) or contribute to point processes \cite{MNPR} and statistics in spherical space \cite{Baddeley,JMRu}.

\medskip

In the recent paper \cite{DeussHoerrmannThaele} a new tessellation model of the two-dimensional unit sphere has been introduced. It arises as the result of a recursive cell splitting scheme and can be regarded as the spherical analogue of the STIT-tessellation model that has  been studied intensively in Euclidean stochastic geometry in the last decade. While the work \cite{DeussHoerrmannThaele} focuses on the isotropic two-dimensional case and on metric and combinatorial parameters of individual cells, the present paper has a much broader scope and deals with so-called splitting tessellations in higher dimensional spherical spaces and with more general direction distributions. Our focus lies on first- and second-order properties of cumulative functionals that are induced by the family of spherical curvature measures. In addition, we shall describe precise distributions such as the one of the typical cell and the typical maximal spherical  segment.

\medskip

Let us briefly indicate the random recursive construction of the splitting tessellation $Y_t$ in the $d$-dimensio\-nal unit sphere $\SSd$. The random
construction starts with $\SSd$ as the unique cell. After an exponential waiting time with parameter 1 a random great hypersphere with  distribution $\kappad$ divides $\SSd$ into two cells. Now, this branching mechanism continues recursively and independently in both newly created cells. Let us describe this cell splitting scheme more formally:
\begin{description}
\item[1.\ Initiation] At time zero we put $Y_0:=\{\SSd\}$, $\tau_0:=0$, and we set a counter $n$ to be equal to $1$. (More generally, we may start with an arbitrary fixed tessellation $Y_0=\{T\}$ of $\SS^d$).
\item[2.\ Recursion] Suppose that the counter is $n\geq 1$ and that a random time $\tau_{n-1}$ and a random tessellation $Y_{\tau_{n-1}}$ have been realized. Generate a random time $\tau_n$ such that the holding time $\tau_n-\tau_{n-1}$ has the same distribution as $\min_{c\in Y_{\tau_{n-1}}}E_c$ with independent exponentially distributed random variables $E_c$ with parameter $\kappad(\SS_{d-1}[c])$, which in turn has the same distribution as an exponential random variable with parameter $\sum_{c\in Y_{\tau_{n-1}}}\kappad(\SS_{d-1}[c])$. Here we write $\SS_{d-1}[c]$ for the set of all great hyperspheres hitting the interior of $c$ and hence the parameter $\kappad(\SS_{d-1}[c])$ is equal to the probability that a  random great hypersphere with distribution $\kappad$ hits the interior of the cell $c$.

      If $\tau_n\leq t$, we
\begin{itemize}
\item[-] randomly pick  a cell $c_n\in Y_{\tau_{n-1}}$, where each cell $c\in Y_{\tau_{n-1}}$ available at time $\tau_{n-1}$  is selected with probability   $\kappad(\SS_{d-1}[c])/\sum_{{c}'\in Y_{\tau_{n-1}}}\kappa(\SS_{d-1}[{c}'])$,
\item[-] choose a great hypersphere $S_n$ with distribution $$\frac{\kappad(\SS_{d-1}[c_n]\cap \,\cdot\,)}{ \kappad(\SS_{d-1}[c_n])},$$
\item[-] put $$Y_{\tau_n}:=\oslash(c_n,S_n,Y_{\tau_{n-1}})\,,$$ that is, the cell $c_n$ is split by $S_n$ into two subcells, all other cells remain unchanged,
\item[-] increase the counter $n$ by one and repeat the recursion step.
\end{itemize}
If $\tau_n>t$, output the random tessellation $Y_{\tau_{n-1}}$.
\end{description}

\medskip

As a consequence of the probabilistic construction, the cells behave independently of each other, a cell $c$ has an exponentially distributed  lifetime with parameter $\kappad(\SS_{d-1}[c])$
and finally is split into two parts by a random great hypersphere with distribution $\kappad(\SS_{d-1}[c]\cap \,\cdot\,)/\kappad(\SS_{d-1}[c])$, as described above.
In Euclidean space,  a corresponding dynamic description, in a bounded observation window, has been the starting point
for various investigations of stationary iteration stable (STIT) or more general branching random tessellations. We refer to \cite{GST} for a detailed analysis of a more general model in Euclidean space (see, in particular, \cite[Example 2.9]{GST}). For the present contribution, the framework and the main foundational results of \cite{GST} can be adjusted to $\SSd$, in particular, the observation window in the Euclidean case can be replaced by the whole space $\SSd$ and the steering division kernel (cf.~\cite[Definition 2.6]{GST}) is chosen as the map which assigns to a tessellation $T$ of $\SSd$ and a cell $c\in T$ the restriction of $\kappad$ to $\SS_{d-1}[c]$. Moreover, we also take advantage of another description of the continuous time evolution of the random tessellation $Y_t$ for $t\ge 0$, in terms of martingale properties of the piecewise constant Markov jump process defined by $(Y_t)_{t\ge 0}$ and taking values in the space of tessellations of the $d$-dimensional unit sphere. The connection between the dynamic description given above and the characterization via its generator is studied in detail in \cite{GST} and will be summarized in Section \ref{Sec2.3}.

\medskip

Let us briefly present a rough overview of the content of this paper. We start in Sections \ref{subsec21:BasicSphericalGeometry} and \ref{subsec:22SphericalIntGeo} by recalling some background material from spherical geometry and spherical integral geometry. Tessellations on the sphere are introduced in Section \ref{Sec2.3}, where we also formally define the splitting tessellation process by its generator and by using the general theory of pure jump Markov processes. Our key technical devices are the content of Section \ref{subsec:24Martingales}. Here, we construct several classes of martingales connected to the splitting tessellation process. A first application of the theory developed there is the computation of the capacity functional of the random set arising as the union of the cell boundaries of a splitting tessellation $Y_t$ in Section \ref{sec:3CapacityFunctional}. The capacity functional is one of the most important characteristics associated with a random set. We also use martingale methods and tools from spherical integral geometry, most notably the valuation property and the spherical Crofton formula, to compute the expected sum over all (localized) spherical curvature measures in Section \ref{sec:4Expectations}. These first-order characteristics are used throughout the present work.

\medskip

Spherical integral-geometric transformation formulas are developed in Section \ref{subsec:IntegralGeometricTrafo} and combined in Section \ref{subsec:51VSA} with further martingale tools to determine second-order properties of splitting tessellations. In particular, the variance of the total $(d-1)$-dimensional Hausdorff measure of the union of all cell boundaries of cells of $Y_t$ is computed explicitly. Moreover, in Section \ref{subsec:52FurtherVariancesCovariances} we determine the covariance structure of all spherical curvature measures and provide fully explicit formulas for isotropic splitting tessellations on the $2$-dimensional unit sphere. One further second-order parameter that might be associated with a  random measure on $\SSd$ is its spherical pair-correlation function. We formally introduce this concept in Section \ref{sec:PCF} using Palm calculus for random measures on homogeneous spaces. A comparison of the second-order parameters of a splitting tessellation on $\SSd$ with the corresponding parameters of a Poisson great hypersphere tessellation is the content of Section \ref{sec:Comparison}. For the study of the pair-correlation function, we mainly focus on
the case of isotropic splitting tessellations, which arise precisely if $\kappad$ is rotation invariant, but we also provide general formulas. In particular, we compute the pair-correlation function of the $(d-1)$-dimensional Hausdorff measure of the union of the cell boundaries.

\medskip

The final Section \ref{sec:7TypicalCellsFacesAndDistributions} is devoted to distributional properties of the cells and the so-called maximal spherical  faces of a splitting tessellation. Again, no isotropy assumption is required. As a technical tool we introduce in Section \ref{sec:7.1} a continuous-time dynamic version of Poisson great hypersphere tessellations on $\SSd$, again by using the general theory of Markov jump processes. This point of view is used in Section \ref{subsec:72Relationships} to establish first a relationship between the cell intensity measure of a splitting tessellation and that of a Poisson great hypersphere tessellation. Furthermore, this result is used to express the intensity measure of the maximal spherical $k$-dimensional faces of a splitting tessellation at time $t\ge 0$ as a mixture of intensity measures of spherical $k$-dimensional faces of Poisson great hypersphere tessellations up to time $t$. This crucial representation is the key tool in Section \ref{subsec73:sphericalmaxfaces}, which leads to a representation of the distribution of the typical maximal  spherical $k$-dimensional  face of a splitting tessellation $Y_t$ as a mixture of the corresponding distributions of typical spherical $k$-dimensional  faces in  Poisson great hypersphere tessellations at time $s\in (0,t)$. Finally, this allows us to determine precisely the expected length as well as the birth time distribution of the typical maximal spherical segment of a splitting tessellation on $\SSd$.

\medskip

Whenever possible we compare the results we obtain on the sphere with those for STIT-tessel\-la\-tions in $\RR^d$ that are available in the literature. This makes transparent in which situations the results in the curved space $\SSd$ are similar to corresponding results in the flat case $\RR^d$ and allows us to highlight where significant differences can be observed. Most of our results can be stated for general (regular) great hypersphere distributions $\kappad$, but we always highlight the case of isotropic splitting tessellations, which arises precisely when $\kappad$ is rotation invariant, and we specialize to the case of $\SS^2$ to illustrate more general results.

\section{Preliminaries}

The study of random tessellations requires a number of results and tools from stochastic geometry, including
point processes of particles, random closed sets, and general methods from the theory of  stochastic processes. In this section, we recall or introduce the relevant concepts and explain what is needed from spherical convexity and integral geometry.

\subsection{Basic notions from spherical geometry}\label{subsec21:BasicSphericalGeometry}

We fix $d\geq 2$ and consider the $d$-dimensional unit sphere
$$
\SSd:=\{x\in\RR^{d+1}:\|x\|=1\}\subset\RR^{d+1}\,,
$$
where $\|\cdot\|$ stands for the usual Euclidean norm in $\RR^{d+1}$. The origin in $\RR^{d+1}$ is denoted by $o$. On $\SSd$ we use the spherical (geodesic) distance $\ell(x,y):=\arccos(\lan x,y\ran)$, for $x,y\in \SSd$, where we write $\lan x,y\ran$ for  the Euclidean scalar product of vectors $x$ and $y$. As usual, we denote the induced Borel $\s$-field by $\cB(\SSd)$. For $s\geq 0$ we write $\cH^s$ for the $s$-dimensional Hausdorff measure (normalized as in
\cite[p.~171]{Federer69})
and put
$$
\b_{d}:=\cH^d(\SSd)=\frac{2\pi^{\frac{d+1}{ 2}}}{\Gamma\big(\frac{d+1}{2}\big)}\,.
$$

\medskip

By a {\em spherically convex set} we understand the intersection of $\SSd$ with a non-empty Euclidean convex cone in $\RRd1$. Let $\KK^d$ denote the collection of all non-empty, compact, spherically convex sets of $\SSd$, whose elements are called spherically convex bodies. We equip $\KK^d$ with the (spherical) {\em Hausdorff distance} and write
$\cB(\KK^d)$ for the induced Borel $\s$-field. Note that the induced topology on $\KK^d$ is the same as the
subspace topology induced by the {\em Fell topology} (see \cite[Section 12.2]{SW}) on the space $\cF(\SSd)$ of closed subsets of $\SSd$. Thus $\KK^d$ is a compact topological space with countable base. These statements can be proved as in the Euclidean setting (cf. \cite[Sections 12.2 and 12.3]{SW}).  For $K\in\KK^d$ we define the dual (polar) set $K^\ast$ of $K$ by $K^\ast:=\{x\in\SSd:\ell(x,K)\geq\pi/2\},$ where  $\ell(x,K):=\min\{\ell(x,y):y\in K\}$ is the spherical distance from $x$ to $K$ (see \cite[Chapter 6.5]{SW} and \cite{HugSchneider2016} for background information and further references).

\medskip

Next we introduce a particular class of spherically convex sets. A {\em spherical polytope} is defined as the intersection of $\SSd$ with a Euclidean polyhedral cone in $\RRd1$. The latter is defined as the intersection of finitely many closed half-spaces in $\RRd1$ that contain the origin in their boundaries. It is convenient for our purposes to consider also the space $\RRd1$ as a (degenerate) Euclidean polyhedral cone arising from an empty intersection of half-spaces. This way $\SSd$, closed half-spheres, and also $k$-dimensional subspheres (see the next paragraph)  become (degenerate) spherical polytopes. By $\PP^d\subset\KK^d$ we denote the space of spherical polytopes and equip $\PP^d$ with the trace $\sigma$-field $\cB(\PP^d)$ of $\cB(\KK^d)$ on $\PP^d$. In what follows, we shall call the $d$-dimensional elements of $\PP^d$ cells. For a spherical polytope $c\in\PP^d$, which arises as the intersection of $\SSd$ with a Euclidean polyhedral cone $\hat c$, we say that a subset $F\subset c$ is a {\em $j$-face} of $c$ if $F$ is obtained as the intersection of $c$ with a $(j+1)$-dimensional Euclidean face of the polyhedral cone $\hat c$, $j\in\{0,\ldots,d\}$. By $\cF_j(c)$ we denote the collection of all $j$-faces of $c$.

\medskip

By a {\em subsphere} of $\SSd$ of dimension $k\in\{0,\ldots,d-1\}$ we understand the intersection of $\SSd$ with a $(k+1)$-dimensional linear subspace of $\RRd1$. We denote by $\SS_k$ the space of all $k$-dimensional subspheres of $\SSd$, which is equipped with the subspace topology and the natural Borel $\s$-field $\cB(\SS_k)$. We also put $\SS_d:=\{\SS^d\}$. Together with the action of the rotation group $\SO(d+1)$ on $\SS_k$, the space $\SS_k$ becomes a compact homogeneous space and as such it carries a unique Haar probability measure, which is denoted by $\nu_k$. For example, the measure $\nu_{d-1}$ has the representation
$$
\nu_{d-1}(\,\cdot\,)=\frac{1}{ \b_{d}}\int_{\SSd}{\bf 1}(u^\perp\cap\SSd\in\,\cdot\,)\,\cH^d(\dint u)\,.
$$
From now on we call elements of $\SS_{d-1}$ great hyperspheres  of $\SSd$.
For a set $B\subset\SSd$, let us denote by
$$
\SS_{d-1}[B]:=\{S\in\SS_{d-1}:S\cap{\rm int}(B)\neq\emptyset\}
$$
the set of all great hyperspheres  of $\SSd$ which have non-empty intersection with the interior ${\rm int}(B)$ of $B$ (here the interior refers to the topology of $\SSd$).

\begin{definition}
{\rm A Borel measure $\kappad$ on $\SS_{d-1}$ is called regular if $\kappad(\{S\in \SS_{d-1}: e\in S\})=0$ for $e\in \SS^d$.
}
\end{definition}

Clearly, $\nu_{d-1}$ and any measure
which is absolutely continuous with respect to $\nu_{d-1}$ is regular. The assertion of the following lemma will be used repeatedly. For $S\in\SS_{d-1}$ we denote by $S^+,S^-$ the two closed half-spheres bounded by $S$. Since we shall always work with statements symmetric in $S^+,S^-$, we do not have to specify this further.

\begin{lemma}\label{setzero}
If $\kappad$ is a regular Borel measure on $\SS_{d-1}$ and $c\in\PP^d$, then
$$
\kappa(\{S\in \SS_{d-1}:S\cap c\neq\emptyset, c\subset S^+ \text{ \rm  or }c\subset S^-\})=0\,.
$$
This remains true for an arbitrary compact set $C\subset\SS^d$ in place of $c$ if $\kappad$ is absolutely continuous with respect to $\nu_{d-1}$.
\end{lemma}

\begin{proof} First, suppose that $c\in\PP^d$.
For each face $F$ of $c$, we choose a point $c_F$ in the relative interior of $F$. If $S\in \SS_{d-1}$ satisfies $S\cap c\neq\emptyset$ and $ S\cap \text{int}(c)=\emptyset$, then there is a  face $F$ of $c$ (with $\text{dim}(F)\le d-1$)) such that $c_F\in S$. Since $F$ has only finitely many faces, $\kappa$ is subadditive and regular,
that is, $\kappad(\{S\in \SS_{d-1}: c_F\in S\})=0$,
the assertion follows. \medskip

 Now let $C\subset\SS^d$ be an arbitrary compact set. If $S\cap C\neq\emptyset$
and $C\subset S^+$ or $C\subset S^-$, then $\widehat{C}:= \text{conv}(\{o\}\cup C)$ (the Euclidean convex hull of $C$ and the origin)
 intersects the linear hull $\text{lin}(S)$ of $S$ in a non-degenerate segment  and $\text{lin}(S)$ supports  $\widehat{C}$. Let $U\subset\RR^{d+1}$ be an arbitrary fixed $d$-dimensional linear subspace. Then
\cite[Corollary 2.3.11]{Schneider} implies that
\begin{align*}
&\nu_{d-1}(\{S\in \SS_{d-1}:S\cap C\neq\emptyset, C\subset S^+ \text{ or }C\subset S^-\})\\
&\qquad \le \nu(\{\varrho\in \SO(d+1):\varrho U \text{ supports }\widehat{C} \text{ in a non-degenerate segment}\})=0\,.
\end{align*}
This also proves the assertion if  $\kappad$ is absolutely continuous with respect to $\nu_{d-1}$.
\end{proof}

For a regular Borel measure $\kappad$ on $\SS_{d-1}$, there is a uniquely determined symmetric Borel measure $\kappad^\circ$ on $\SSd$ such that
\begin{equation}\label{relme}
\kappad( \ \cdot \ )=\int_{\SSd}\mathbf{1}(u^\perp\cap \SSd\in\ \cdot\ )\, \kappad^\circ (\dint u)
\end{equation}
and $\kappad^\circ$ is regular in the sense that $\kappad^\circ(e^\perp\cap \SSd)=0$ for $e\in\SSd$. If $A\subset\SSd$ is a symmetric Borel set,
then $\kappad^\circ(A)=\kappad(\{S\in\SS_{d-1}:S^\perp\cap A\neq\emptyset\})$. Conversely, any regular Borel measure $\kappad^\circ$ on $\SSd$
leads to a regular Borel measure $\kappad$ on $\SS_{d-1}$ via \eqref{relme}.

\subsection{Elements of spherical integral geometry}\label{subsec:22SphericalIntGeo}

Let $K\in\KK^d$ be a non-empty spherically convex set and fix $0<r<\pi/2$. By $K_r$ we denote the spherical $r$-parallel set of $K$, that is,
$$
K_r=\{x\in\SSd:\ell(x,K)\leq r\}\,.
$$
The spherical version of Steiner's formula \cite[Theorem 6.5.1]{SW} says that $\cH^d(K_r\setminus K)$ can be expressed as
$$
\cH^d(K_r\setminus K)=\sum_{j=0}^{d-1}\b_{j}\b_{d-j-1}\,V_j(K)\int_0^r (\cos t)^j(\sin t)^{d-j-1}\,\dint t\,.
$$
The coefficients $V_0(K),\ldots,V_{d-1}(K)$ are the \textit{spherical intrinsic volumes} of $K$. In addition, we
put $V_i(\emptyset):=0$ for $i\in\{0,\ldots,d-1\}$.  It is often convenient to extend this list of functionals by putting
$V_d(K):=\cH^d(K)/\b_{d}$. The spherical intrinsic volumes are additive, rotation invariant, continuous with respect to the Hausdorff distance and bounded by $1$. For a spherical polytope $c\in\PP^d$, the intrinsic volumes have the explicit representation
$$
V_j(c)=\frac{1}{\b_{j}}\sum_{F\in\cF_j(c)}\gamma(F,c)\,\cH^j(F)\,,\qquad j\in\{0,\ldots,d-1\}\,,
$$
where $\gamma(F,c)$ is the external angle of $c$ at $F$. In terms of the polar body $c^\ast\in\PP^d$ of $c$
and the set $N(c,F):=\{y\in c^\ast:\lan x,y\ran=0\}$ with an arbitrary point $x$ in the relative interior of $F$,
$\gamma(F,c)$ can be written as
$$
\gamma(F,c)=\frac{1}{\b_{d-1-j}}\cH^{d-1-j}(N(c,F))\,.
$$
For later purposes, we need the following values. First, we have
$$
V_j(S_k)=\delta_{jk}\,,\qquad S_k\in\SS_{k}\,,\quad j,k\in\{0,\ldots,d\}\,,
$$
see \cite[(9)]{HugSchneider2016} (or \cite[(6.46)]{SW}). Here, $\delta_{jk}$ stands for the Kronecker symbol, which is equal to $1$ if $j=k$ and $0$ otherwise.
Moreover, if $K=\overline{xy}$ with $\ell(x,y)<\pi$ is the unique spherical segment connecting two points $x,y\in\SSd$, then we have  $V_0(K)=1/2$ and $V_1(K)=\ell(x,y)/(2\pi)$, while $V_2(K)=\ldots=V_{d-1}(K)=0$. In particular, it should be noted that $V_0(K)$ does not coincide with the Euler characteristic of $K$, which is in contrast to the Euclidean case. By continuity, these relations extend to the case that $\ell(x,y)=\pi$, but then the connecting geodesic is no longer uniquely determined. Moreover, the intrinsic volume of order $d-1$ always has the representation
\begin{equation*}
V_{d-1}(K)=\frac{\cH^{d-1}(\partial K)}{ 2\b_{d-1}}\,,
\end{equation*}
where $\partial K$ denotes the boundary of $K$, provided that $K$ has non-empty interior (for a $d$-dimen\-sio\-nal spherical polytope, the boundary is the union of all of its $(d-1)$-dimensional faces). If in addition $S\in\SS_{d-1}$, then
$$
V_{d-1}(K\cap S)=\frac{\cH^{d-1}(K\cap S)}{\b_{d-1}}\,.
$$

\medskip

We can now rephrase one of our crucial devices, namely \textit{Crofton's formula} for spherical intrinsic volumes. For $k\in\{0,\ldots,d\}$ and $j\in\{0,\ldots,k\}$, it states that
$$
\int_{\SS_k}V_j(K\cap S)\,\nu_k(\dint S)=V_{d-k+j}(K)\,,
$$
see \cite[p.~261]{SW} (note that the case $k=d$ is trivial).

\medskip

Finally, we recall from \cite[Equation (6.63)]{SW} that the invariant probability measure of all great hyperspheres hitting a spherically convex set $K\in\KK^d\setminus\bigcup_{k=0}^d\SS_k$ can be expressed as a sum of spherical intrinsic volumes. For this, we define
$$
\SS_{d-1}\blk B\brk:=\{S\in\SS_{d-1}:S\cap B\neq\emptyset\}
$$
for a set $B\subset\SSd$. Then we get
\begin{equation}\label{nonloc}
\nu_{d-1}(\SS_{d-1}\blk K\brk )=2\sum_{j=0}^{\left\lfloor\frac{d-1}{ 2}\right\rfloor}V_{2j+1}(K)\,.
\end{equation}
Especially for a spherical segment $\overline{xy}$ with length $\ell(x,y)\leq\pi$ we get
\begin{equation}\label{eq:InvMeasureLineSegment}
\nu_{d-1}(\SS_{d-1}\blk\overline{xy}\brk)=2V_1(\overline{xy})=\frac{1}{\pi}\ell(x,y)\,.
\end{equation}

In the following, we also consider local extensions of the spherical intrinsic volumes. To introduce these,
for $K\in \KK^d$ and $x\in \SS^d$ with $\ell(x,K)<\pi/2$, let $p(K,x)$ denote the unique point in $K$
closest to $x$. For a Borel set $A\subset\SS^d$ and $0\le r<\pi/2$, let
$$
K_r(A):=\{x\in K_r\setminus K:p(K,x)\in A\}
$$
be the local parallel set of $K$ determined by $A$ and $r$. Then \cite[Theorem 6.5.1]{SW} yields that
$$
\mathcal{H}^d(K_r(A))=
\sum_{j=0}^{d-1}\b_{j}\b_{d-j-1}\,\phi_j(K,A)\int_0^r (\cos t)^j(\sin t)^{d-j-1}\,\dint t\,,
$$
where
$$
\phi_j(K,A)=\frac{1}{\beta_j}\sum_{F\in \mathcal{F}_j(K)}\gamma(F,K)\mathcal{H}^{j}(F\cap A)
$$
if $K\in\PP^d$ is a spherical polytope. In particular, we have
\begin{equation}\label{eq:PhiD-1}
\phi_{d-1}(K,\,\cdot\,) = \frac{\cH^{d-1}\llcorner K}{\b_{d-1}}
\end{equation}
if $K\in\KK^d$ with $K\subset  S$ for some $S\in\SS_{d-1}$. Here and in what follows, $\cH^{d-1}\llcorner K:=\cH^{d-1}(\,\cdot\,\cap K)$ stands for the restriction of the measure $\cH^{d-1}$ to $K$. (The symbol $\llcorner$ is generally used to denote the restriction of a measure to a subset.)
Again we define $\phi_j(\emptyset,A):=0$ for $j\in\{0,\ldots,d-1\}$ and
$\phi_d(K,A):=\mathcal{H}^d(K\cap A)/\beta_d$. For each $K\in\KK^d$, $\phi_j(K,\cdot)$ is a finite Borel
measure on $\SS^d$,  the \textit{$j$th (spherical) curvature measure} of $K$. For any fixed Borel set $A\subset\SS^d$,
the map $K\mapsto \phi_j(K,A)$ is measurable and has the valuation property, that is,
\begin{equation}\label{eqvaluation}
\phi_j(K_1\cup K_2,A) = \phi_j(K_1,A)+\phi_j(K_2,A)-\phi_j(K_1\cap K_2,A)\,,
\end{equation}
whenever $ K_1,K_2,K_1\cup K_2\in\KK^d$.
The measure-valued map $K\mapsto \phi_j(K,\cdot)$ is weakly continuous, and $\phi_j$ is rotation covariant in the sense that $\phi_j(\vartheta K,\vartheta A)=\phi_j(K,A)$ for all $K\in\KK^d$, $A\in\mathcal{B}(\SS^d)$ and $\vartheta\in\SO(d+1)$, cf.\ \cite[Theorem 6.5.2]{SW}. Clearly, we have $\phi_j(K,\SS^d)=V_j(K)$ and
\begin{equation}\label{locbdmeasure}
\phi_{d-1}(K,A)=\frac{1}{2\beta_{d-1}}\mathcal{H}^{d-1}(\partial K\cap A)
\end{equation}
if $K\in\KK^d$ has non-empty interior. While \eqref{nonloc} does not have a local analogue, the spherical
Crofton formula extends to curvature measures in the form
\begin{equation}\label{eq:CroftonOnSphere2}
\int_{\SS_k}\phi_j(K\cap S,A\cap S)\, \nu_k(\dint S)=\phi_{d-k+j}(K,A)
\end{equation}
for all $K\in\KK^d$, Borel sets $A\subset \SS^d$, $k\in\{0,\ldots,d\}$ and $j\in\{0,\ldots,k\}$, see \cite[p.~261]{SW}.
Since $\phi_j(K\cap S,\cdot)$ is concentrated on $K\cap S$, the integrand in \eqref{eq:CroftonOnSphere2} can be replaced by $\phi_j(K\cap S,A)$ without changing the integral.

We point out that for all formulas (with the exception of \eqref{nonloc}) in this section, there exist Euclidean counterparts involving different normalizations and in the case of the Crofton formula a motion invariant measure which is infinite.

\subsection{Spherical tessellations and spherical splitting tessellations}\label{Sec2.3}

Spherical tessellations partition the unit sphere into finitely many non-overlapping $d$-dimen\-sio\-nal  spherically convex bodies. As in the Euclidean setting, these are  necessarily  spherical polytopes. For this reason, in the following definition  we can equivalently  consider finite collections of spherically convex bodies or spherical polytopes (which will be called cells in the sequel).

\begin{definition}{\rm
By a tessellation $T$ of $\SSd$ we understand a finite collection $T\subset\PP^d$ of $d$-dimensional spherical polytopes such that
\begin{itemize}
\item[{(i)}] $\bigcup\{c: c\in T\}= \SSd$,
\item[{(ii)}] any two elements of $T$ have disjoint interiors.
\end{itemize}
The set of all tessellations of $\SSd$ is denoted by $\TT^d$.}
\end{definition}

We now introduce a $\sigma$-field on $\TT^d$.
 Recall that $E=\KK^d$ is a compact Hausdorff space with countable base.
Let $\sN_s(E)$ denote the set of simple counting measures on $E$, and let $\sF_{lf}(E)$ denote the set of locally finite (hence finite) subsets of $E$. We can identify these spaces via the map $i_s:\sN_s(E)\to  \sF_{lf}(E)$, $\eta\mapsto i_s(\eta)=\text{supp}(\eta)$, where $\text{supp}(\mu)$ denotes the support of a measure $\mu$.  On $\sF_{lf}(E)$ we have the subspace topology
$\cT_{lf}$  of the Fell topology on $\cF(E)$, and on $\sN_s(E)$ we consider the vague topology $\cT_{vg}$, that is, the
coarsest topology such that all evaluation maps $\eta\mapsto \int_E g\, \dint \eta$ are continuous, whenever $g:E\to  \RR$ is a non-negative continuous function. Let us recall from \cite[p.~23]{DavisMarkovModels} that by a \textit{Borel space}  we understand a topological space which is homeomorphic to a Borel subset of a complete separable metric space.

\begin{lemma}\label{lemmess1}
Let $\cB_{lf}$ and $\cB_{vg}$ denote the $\s$-fields generated by $\cT_{lf}$ and $\cT_{vg}$,
respectively. Then $\cB_{vg}= i_s^{-1}(\cB_{lf})$ and $i_s(\cB_{vg})= \cB_{lf}$. In particular,
the $\sigma$-fields induced on $\TT^d$  by the vague topology and  by the Fell topology coincide, and
$\TT^d$ is a Borel space.
\end{lemma}

\begin{proof}
The Portmanteau theorem for vague convergence implies that $i_s$ is continuous. Hence, $i_s^{-1}(\cT_{lf})\subset \cT_{vg}$.
It is easy to see that this inclusion
is strict  (think, for example, of two sequences of distinct points $(x_n)_{n\in\NN}$ and $(y_n)_{n\in\NN}$ which converge to the same limit point $x\in E$ and the sets $\{x_n,y_n\}$, which converge to the set $\{x\}$ in the Fell topology, but for which the sequence of the associated counting measures does not converge in the vague topology). Then we deduce that $i_s^{-1}(\cB_{lf})\subset \cB_{vg}$, and the same is true for the induced
subspace $\s$-fields on $\TT^d$ and $\sN_s(\TT^d):=i_s^{-1}(\TT^d)$. On the other hand, if $g:E\to [0,\infty)$ is continuous and $O\subset\RR$ is open, then $\{\eta\in \sN_s(E): \int_Eg\, \dint \eta\in O\}
\in i_s^{-1}(\cB_{lf})$, which follows from \cite[Lemma 3.1.4]{SW} (this remains true if $g$ is measurable). Therefore, we also get the reverse inclusion  $\cB_{vg}\subset i_s^{-1}(\cB_{lf})$, whence the first assertion.
 This equality extends to the intersection with $\TT^d$ so that the second assertion also follows.

 The spherical analogue of \cite[Lemma 10.1.2]{SW} states that $\TT^d$ is a Borel subset of the Polish
 space $\cF(E)$ (with the Fell topology), and hence it is a Borel space.
\end{proof}

In the following, it will be sufficient to know that $\TT^d$ is a Borel space. Since $\TT^d$ is a subspace of the Polish space $\cF(E)$, but not apparently a countable intersection of open sets in $\cF(E)$, it is not clear whether it is even a Polish space.

\medskip

The $\sigma$-field on $\TT^d$ can now be used to define a random tessellation as a measurable map $Y:\Omega\to\TT^d$ from an underlying probability space $(\Omega,\sigmaalgebra,\bP)$ into the measurable space $(\TT^d,\cB_{lf})=(\TT^d,\cB_{vg})$. Here and in what follows we shall assume that the probability space $(\Omega,\sigmaalgebra,\bP)$ is rich enough to carry all the random objects we consider in this paper.

\begin{definition}{\rm
For $T\in\TT^d$, $c\in \PP^d$ and $S\in\SS_{d-1}$, we define $\oslash:\PP^d\times \SS_{d-1}\times \TT^d\to\TT^d$ by
$$
\oslash(c,S,T):= 
(T\setminus\{c\})\cup\{c\cap S^+,c\cap S^-\}\in \TT^d\,,
$$
if $c\in T$ and $S\in \SS_{d-1}[c]$,
where $S^\pm$ are the two closed half-spheres determined by $S$,
and otherwise we define $\oslash(c,S,T):=T$.}
\end{definition}

In other words, if $c\in T$ and $S\in \SS_{d-1}[c]$ then $\oslash(c,S,T)$ is the tessellation arising from $T$ when the cell $c$ is split by the great hypersphere $S$.

\begin{lemma}\label{lemmess2}
The map  $\oslash:\PP^d\times \SS_{d-1}\times \TT^d\to\TT^d$ is Borel measurable (with respect to the induced subspace topologies).
\end{lemma}

\begin{proof}
 To verify this in detail, the following observations are useful. First, the set $\{(c,T)\in \PP^d\times \TT^d:c\in T\}$ is measurable.
This follows from the preceding identification and from the general fact that $\{(x,\eta)\in E\times \sN_s(E):\eta(\{x\})>0\}$ is measurable. To see this,
we use $\eta(\{x\})=\int_E\mathbf{1}(x=y)\, \eta(\dint y)$ and that $(x,\eta)\mapsto \int_E f(x,y)\, \eta(\dint y)$ is measurable for  measurable functions  $f:E^2\to [0,\infty]$, since this is true for measurable functions of the form $f(x,y)=h(x)g(y)$. On this measurable set, the map $(c,T)\mapsto T\setminus\{c\}$ is measurable, since the corresponding map $(x,\eta)\mapsto \eta-\delta_c$ is measurable ($\delta_c$ denotes the point mass located at $c$). Next we observe that  the set
$\{(c,S)\in \KK^d\times \SS_{d-1}:\text{int}(c)\cap S\neq\emptyset\}$ is open (with respect to the corresponding subspace topologies) and hence measurable.
Then, on this set, the map $\KK^d\times \SS_{d-1}\to\sN_s(E)$, $(K,S)\mapsto \delta_{K\cap S^+}+\delta_{K\cap S^-}$, is measurable (here one can use \cite[Theorem 12.2.6 (a)]{SW}).
\end{proof}

The cell-splitting operation can be used to define a spherical splitting tessellation as a continuous time pure jump Markov process in the space
$ \TT^d$ via its generator.
This point of view had previously been adopted in \cite{GST,STSTITPlane, STSTITHigher, STSecondOrder, STBernoulli, STITLimits}, in a Euclidean space.
 For a general introduction to (time-homogeneous, pure) jump processes, we refer to \cite[Chapter 15]{Breiman}, \cite{DavisMarkovModels}, \cite[Chapter 12]{Kallenberg}, \cite[p.~19, Chapter 2.5]{LastBrandt}, or to the lecture notes \cite[Chapter 4]{Eberle}. In these texts, for instance, the existence (by explicit construction) \cite[Chapter 15, Section 6]{Breiman} and uniqueness \cite[Proposition 15.38]{Breiman} of pure jump processes with a given  generator
are established. Moreover, in \cite{GST} the explicit form of the generator and the distribution of the splitting tessellation up to time $t$ are derived from the algorithmic construction which is determined by a steering division kernel (see Remark \ref{rem:OtherInitialTessellation} below). These arguments transfer to the spherical setting without any changes (in fact, there is no need to use a bounded observation window). In our applications, we shall exclusively consider time-homogeneous and  non-explosive pure jump  processes.

\begin{figure}[t]
\begin{center}
\includegraphics[trim = 45mm 167mm 50mm 15mm, clip, width=0.6\columnwidth]{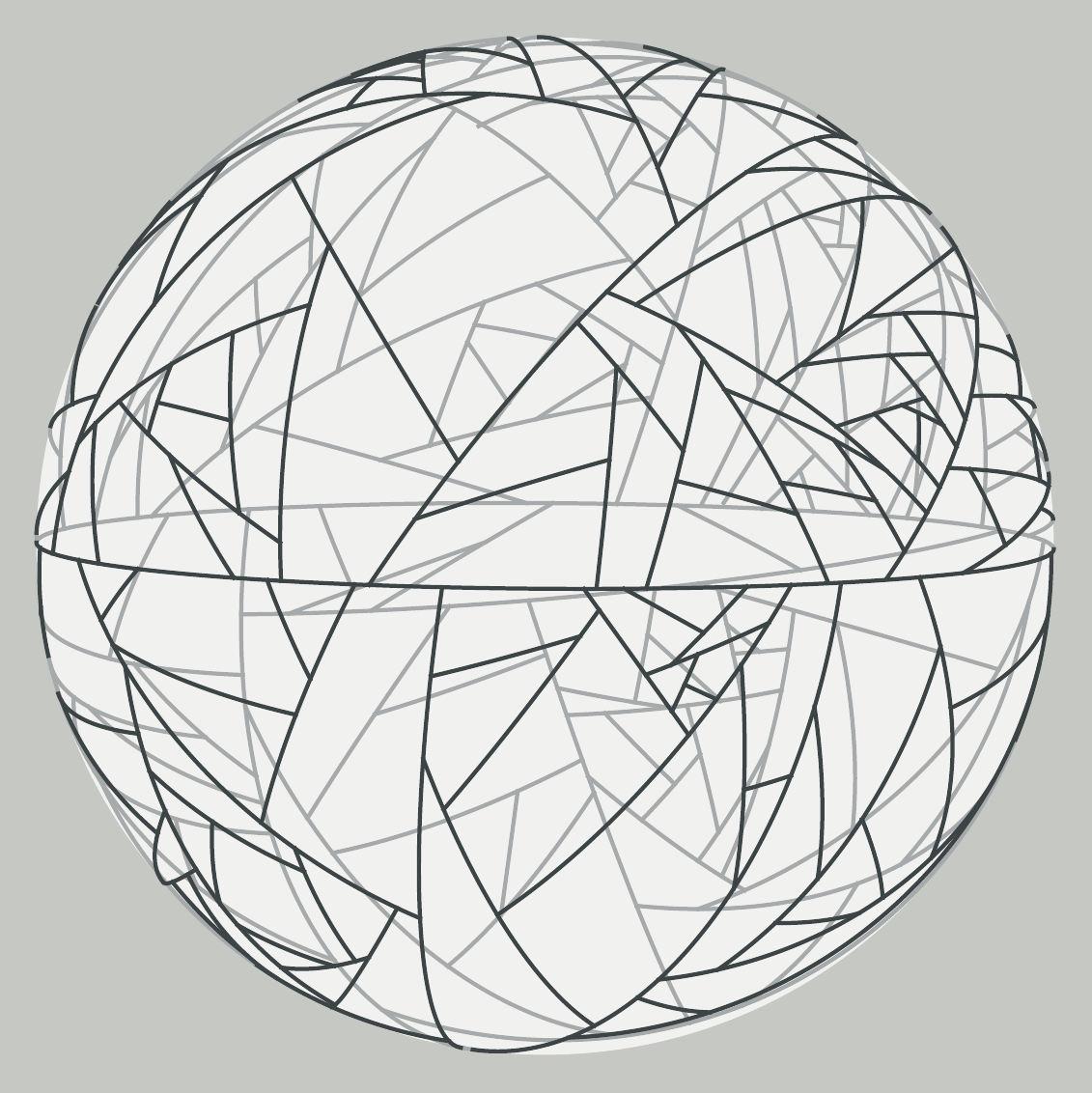}
\end{center}
\caption{Illustration of a splitting tessellation.}
\label{fig:ST}
\end{figure}

\begin{definition}\label{def:GeneratorA}{\rm
Let $\kappad$ be a regular probability measure on $\SS_{d-1}$.
By the splitting process $(Y_t)_{t\geq 0}$ with initial tessellation $Y_0:=\{\SSd\}$ and direction distribution $\kappad$ we understand the continuous time pure jump Markov process on $\TT^d$ whose generator $\cA$ is given by
\begin{align*}
(\cA f)(T) &:=\sum_{c\in T}\int_{\SS_{d-1}[c]}\big[f(\oslash(c,S,T))-f(T)\big]\,\kappad(\dint S)\,,\qquad T\in\TT^d\,,
\end{align*}
where $f:\TT^d\to\RR$ is bounded and measurable. For $t>0$ we call $Y_t$ a splitting tessellation with time parameter $t$.}
\end{definition}

Here and in the following, the underlying direction distribution $\kappad$ on $\SS_{d-1}$ will always be clear from the context. For this
reason we do not indicate the dependence of the splitting process or its generator on $\kappad$ in our notation.

\begin{remark}\label{rem:OtherInitialTessellation}
Similarly as above one can define a splitting process $(Y_t^{(T)})_{t\geq a}$ with $Y_a=T$ and $a\ge 0$, where $T\in\TT^d$ is an arbitrary (fixed) initial tessellation and the starting time is $a$. In most cases we shall work with $T=\{\SSd\}$ and $a=0$ in this paper, which was the motivation in Definition \ref{def:GeneratorA}. However, the generalization to arbitrary (fixed) initial tessellations is needed in Section \ref{sec:3CapacityFunctional}. There we shall
also use an explicit description of the distribution of a splitting tessellation process over a time interval $[a,b]$ starting with an initial tessellation $T$ at time $a$.

To describe this distribution,  we adjust the formalism introduced in \cite{GST} for so-called branching random tessellations and define, for $T\in\TT^d$, a measure $\phi(T;\,\cdot\,)$ on $\PP^d\times\SS_{d-1}$ by
\begin{align*}
\phi(T;\,\cdot\,) := \sum_{c\in T}\delta_c\otimes\kappa\llcorner\SS_{d-1}[c]
\end{align*}
whose total mass is denoted by $\phi(T):=\phi(T;\PP^d\times\SS_{d-1})$. For $T_a\in\TT^d$, $0\leq a\leq b$, $n\in\NN$, $a\leq s_1<\ldots<s_n\le b$, $c_1,\ldots,c_n\in\PP^d$, and $S_1\in\SS_{d-1}[c_1],\ldots,S_n\in\SS_{d-1}[c_n]$, we denote by $\cD(T_a;[a,b];(s_j,c_j,S_j)_{1\leq j\leq n})$ the space of c\`adl\`ag functions $(\Upsilon^{(T_a)}_u)_{u\in[a,b]}$ on $\TT^d$ (that is, $(\Upsilon^{(T_a)}_u)_{u\in[a,b]}$ is a $\TT^d$-valued function on $[a,b]$ that is right-continuous and with left limits) such that $\Upsilon_a^{(T_a)}=T_a$, $u\mapsto\Upsilon_u^{(T_a)}$ is piecewise constant  with jumps only at times $s_1,\ldots,s_n$ so that $\Upsilon_{s_{j}}^{(T_a)}=\oslash(c_{j-1},S_{j-1},\Upsilon_{s_{j-1}}^{(T_a)})$ for some $c_{j-1}\in\Upsilon_{s_{j-1}}^{(T_a)}$ and $S_{j-1}\in\SS_{d-1}[c_{j-1}]$, for $j=1,\ldots,n$,  where $s_0=a$. This notation allows us to translate especially \cite[Equation (4.2)]{GST} to the present (spherical) framework, which for the splitting process $(Y_u^{(T_a)})_{u\in[a,b]}$ with initial tessellation $T_a$ says that
\begin{align*}
\PP((Y_u^{(T_a)})_{u\in[a,b]}\in\,\cdot\,) &= \sum_{n=0}^\infty\hspace{0.7cm}\int\cdots\hspace{-0.8cm}\int\limits_{\hspace{-0.7cm}\{a=s_0<s_1<\ldots<s_n<b\}}\dint s_1\ldots\dint s_n\prod_{j=1}^n\int\phi(\Upsilon_{s_{j-1}}^{(T_a)};\dint(c_{j-1},S_{j-1}))\\
&\qquad\times e^{-\int_a^b\phi(\Upsilon_u^{(T_a)})\,\dint u}\,{\bf 1}((\Upsilon_u^{(T_a)})_{u\in[a,b]}\in\cD(T_a;[a,b];(s_j,c_j,S_j)_{1\leq j\leq n}))\\
&\qquad\times{\bf 1}((\Upsilon_u^{(T_a)})_{u\in[a,b]}\in\,\cdot\,)\,;
\end{align*}
in order to improve the readability of the formulas here (and where appropriate) we write $\int\mu(\dint x)\,f(x)$ instead of $\int f(x)\,\mu (\dint x)$ for the integral of a function $f$ with respect to a measure $\mu$.

As pointed out before, the equivalence of this description of the splitting process via its explicit distribution to the description in terms of the generator of the splitting process is shown and discussed in \cite{GST} in greater generality in Euclidean space and we remark that all arguments can be transferred to our spherical set-up.
\end{remark}

\medskip

Recall that $\sN_s(\TT^d)=i_s^{-1}(\TT^d)$ and put $\mu_T:=i_s^{-1}(T)$ for $T\in\TT^d$. Then we can express  the generator $\cA$
in the form
\begin{align*}
(\cA f)(T)&=\int_{\SS_{d-1}} \int_{ \PP^d} \big[ f(\oslash(c,S,T))-f(T)\big]\,\mu_T(\dint c)\, \kappad(\dint S)\\
&= \int_{\PP^d} \int_{\SS_{d-1}} \big[ f(\oslash(c,S,T))-f(T)\big]\, \kappad(\dint S)
\,\mu_T(\dint c)\,,
\end{align*}
which shows that the map $T\mapsto (\cA f)(T)$ is measurable.
Moreover,
$ f(\oslash(c,S,T))-f(T)=0$ unless $c\in T$ and $S\in \SS_{d-1}([c])$. Writing $\lambda(T):=|T|$, $T\in\TT^d$, for the number of cells in $T$ and defining the  transition kernel
$$
q(T,\,\cdot\,):= \int_{\PP^d} \int_{\SS_{d-1}} \mathbf{1}(\oslash(c,S,T)\in\,\cdot)\, \kappad(\dint S)\,\mu_T(\dint c)\,,
$$
we obtain a probability (rate) kernel $\pi(T,\,\cdot\,):=\lambda(T)^{-1} q(T,\,\cdot\,)$ and interpret $\lambda$ as
an intensity (rate) function. Then we also have $\cA(\,\cdot\,)(T)=q(T,\,\cdot\,)-\lambda(T)\delta_T$, that is,
$$
(\cA f)(T)=(\lambda(\pi-I)f)(T)=\lambda(T)((\pi f)(T)-f(T))\,,
$$
where $If:=f$. It is important to emphasize that in the present setting the  intensity function $\lambda$ is unbounded.

\begin{figure}[t]
\begin{center}
\includegraphics[trim = 45mm 175mm 50mm 15mm, clip, width=0.3\columnwidth]{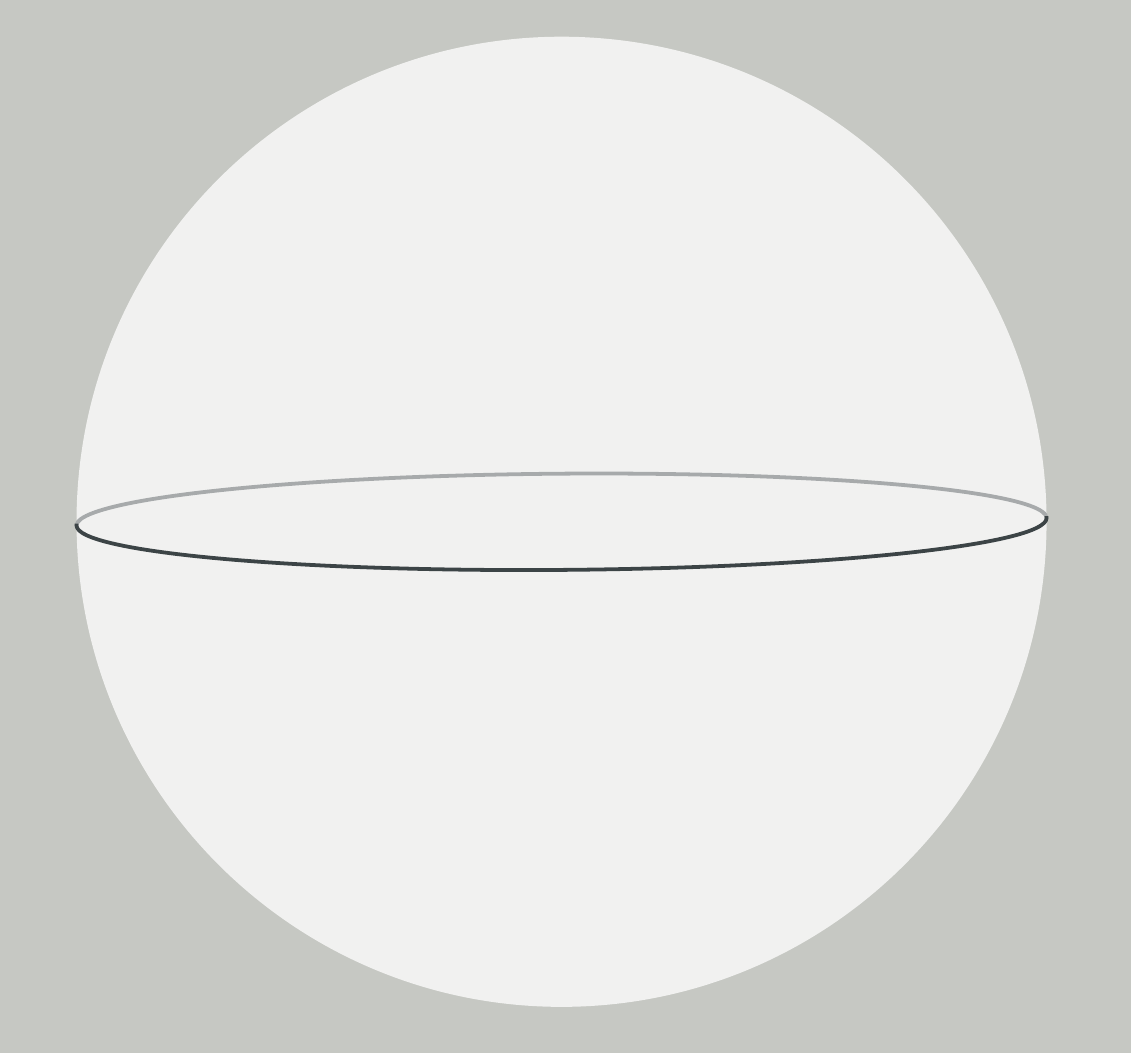}
\includegraphics[trim = 45mm 175mm 50mm 15mm, clip, width=0.3\columnwidth]{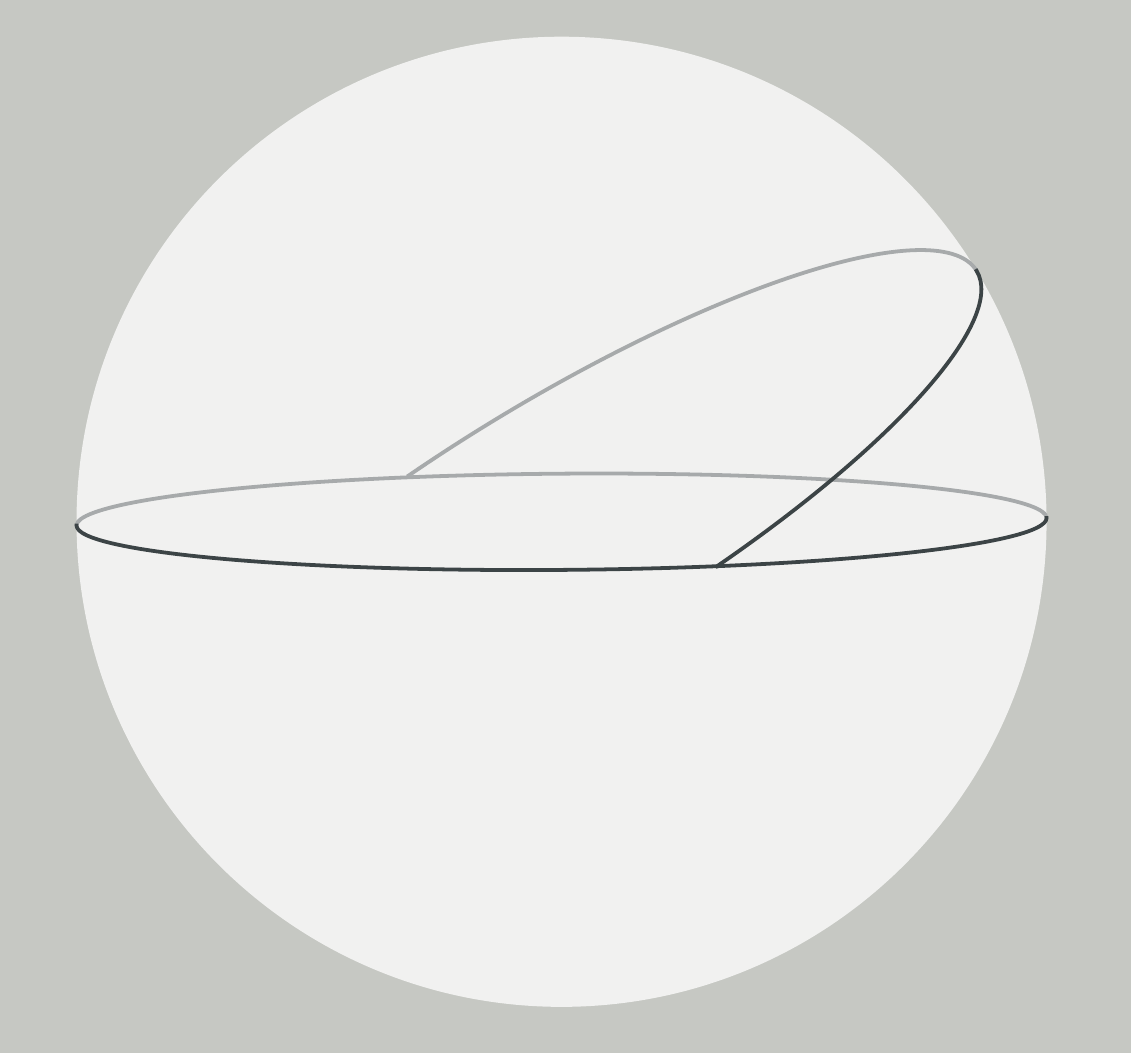}
\includegraphics[trim = 45mm 175mm 50mm 15mm, clip, width=0.3\columnwidth]{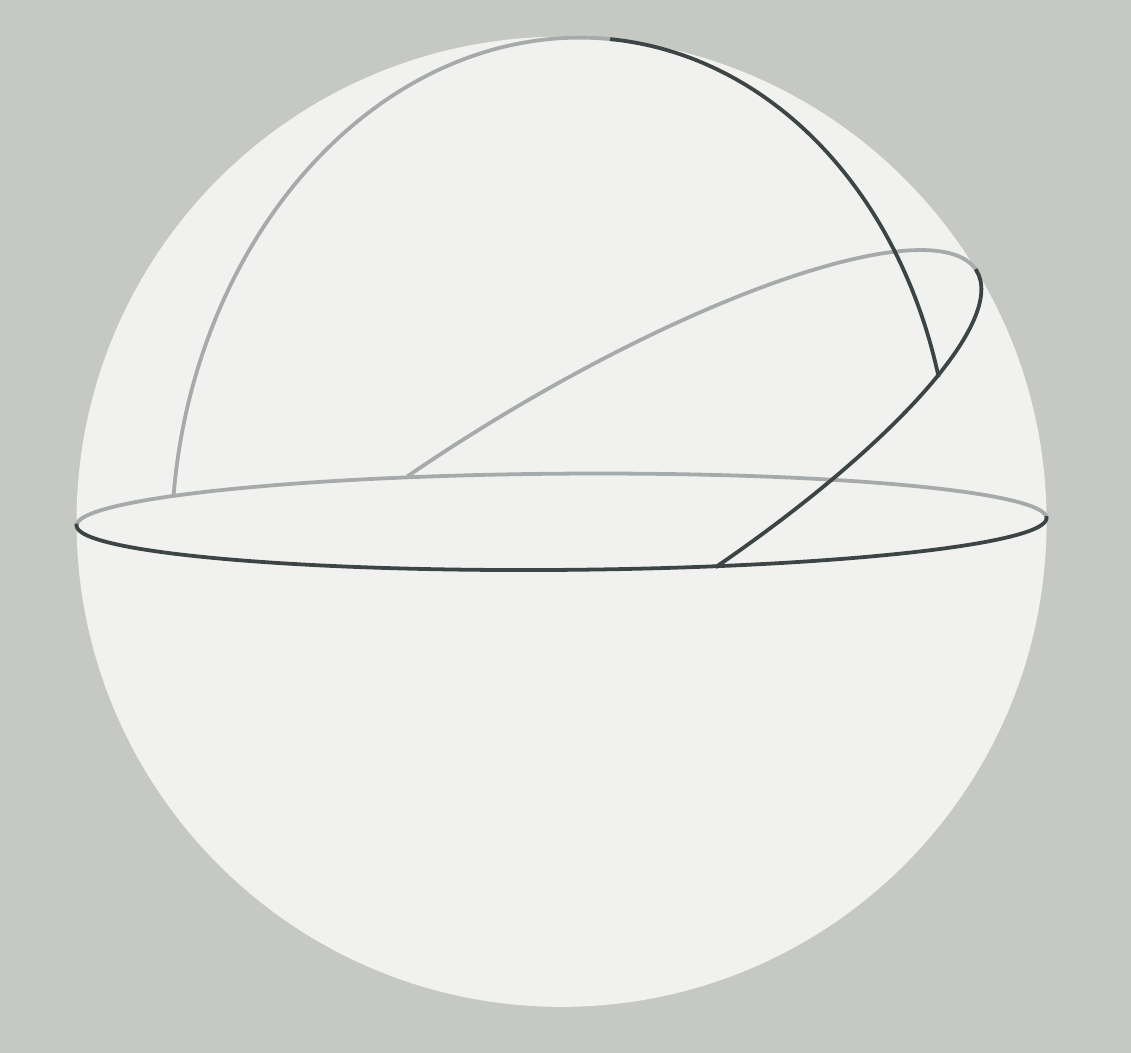}\\
\includegraphics[trim = 45mm 175mm 50mm 15mm, clip, width=0.3\columnwidth]{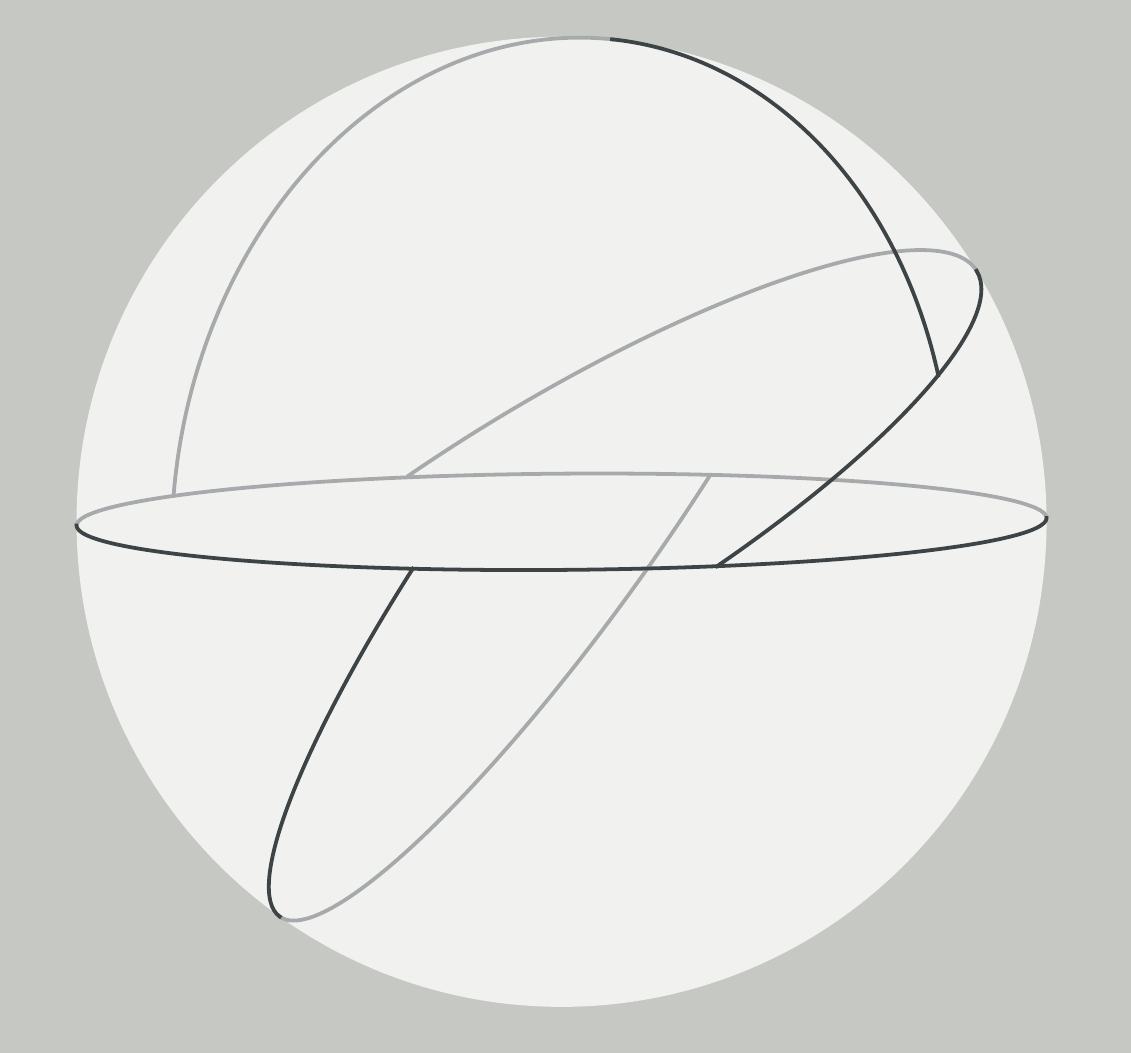}
\includegraphics[trim = 45mm 175mm 50mm 15mm, clip, width=0.3\columnwidth]{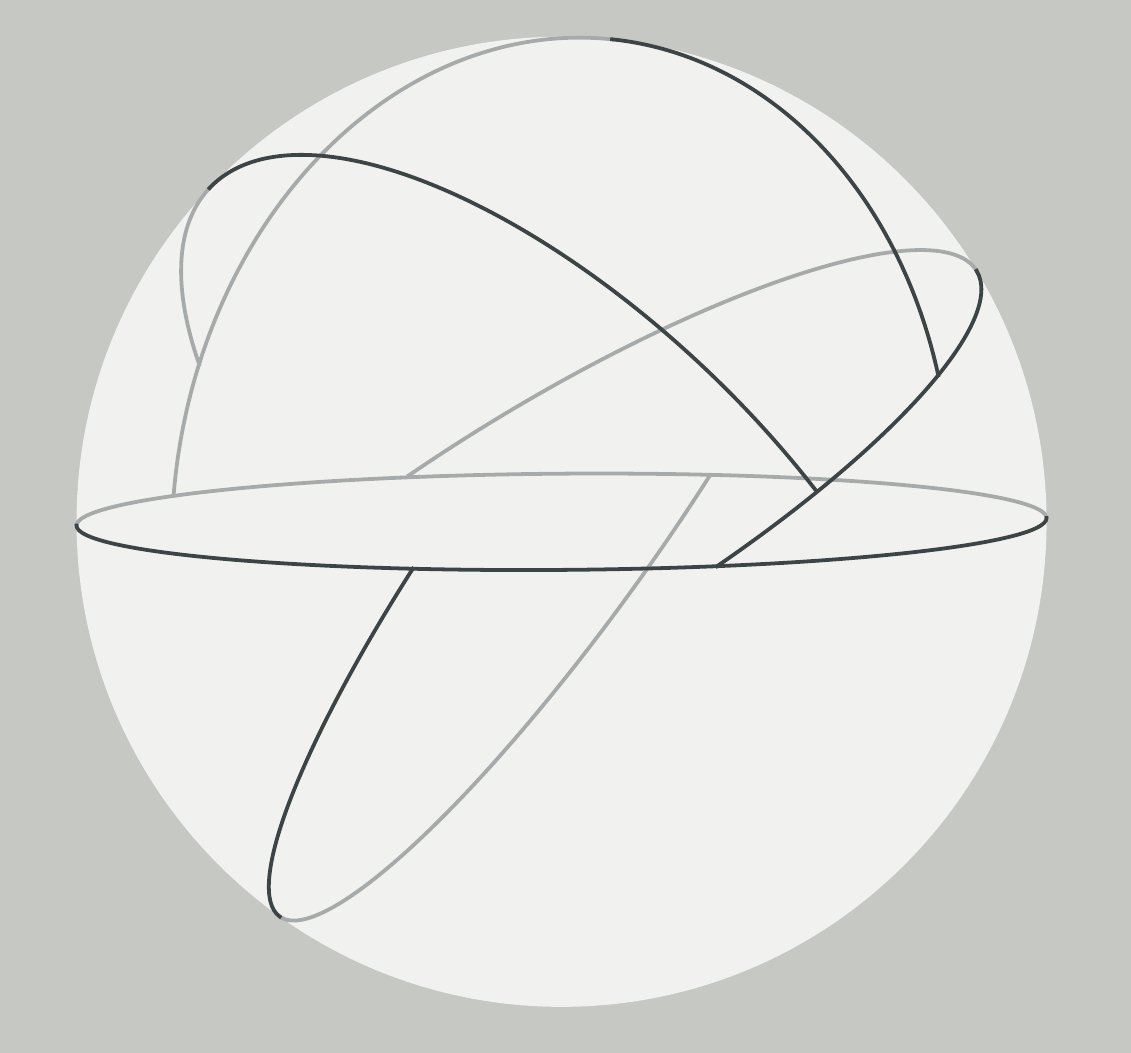}
\includegraphics[trim = 45mm 175mm 50mm 15mm, clip, width=0.3\columnwidth]{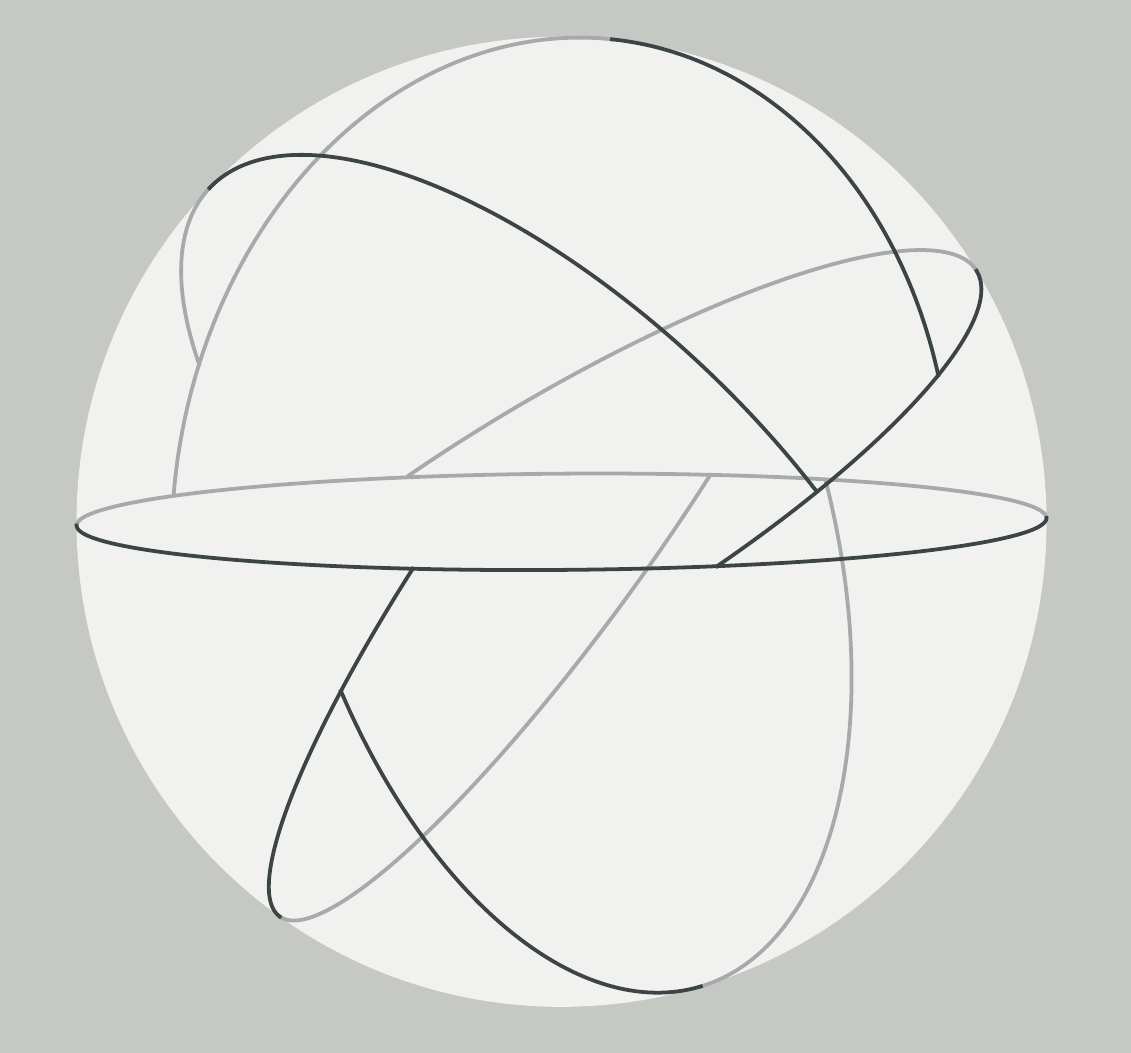}
\end{center}
\caption{Illustration of the splitting tessellation process on $\SS^2$ at different time instants. It starts with the first split and the sphere is rotated in such a way that the first great circle coincides with the equator.}
\label{fig:Process}
\end{figure}

\begin{remark}\label{rembound}\rm
 The transition probabilities $p_h$ of the Markov process $(Y_t)_{t\ge 0}$ generated by $\cA$ are related to $\lambda$ and $q$ by
$p_h(T,B)=q(T,B)h+o(h)$, as $h\downarrow 0$, whenever $B\subset \TT^d$ is measurable and such that $T\notin B$. More generally, if $T\in\TT^d$ and $f\in \cF_b(\TT^d)$ (the set of bounded, measurable functions on $\TT^d$), then $(p_hf)(T)=(1-h\lambda(T))f(T)+h(qf)(T)+o(h)$, see \cite[proof of Theorem 4.3]{Eberle} or
\cite[Corollary 15.29, Proposition 15.30]{Breiman}. Here it is important to note that $o(h)$ depends on $T$ if $\lambda$ is unbounded. If $\lambda$ is bounded, then $o(h)$ is independent of $T$ and we obtain
\begin{equation}\label{DAclass}
\|h^{-1}(p_hf-f)-\cA f\|\to 0\,,\qquad \text{ as $h\downarrow 0$}\,,
\end{equation}
for all $f\in\cF_b(\TT^d)$, where $\|\cdot\|$ is the sup-norm on $\cF_b(\TT^d)$. Without the assumption of a bounded intensity function, the set of all $f\in  \cF_b(\TT^d)$ for which \eqref{DAclass} holds, will be a subset of $\cF_b(\TT^d)$ which is denoted by $D(\cA)$ and is called the domain of the generator $\cA$.
These statements hold for any jump process $(X_t)_{t\ge 0}$ with generator $\cL$ taking values in a Borel space $E$.
While $D(\cA)=\cF_b(\TT^d)$ for bounded $\lambda$, in our applications the intensity function will be unbounded, and therefore we only know the inclusion $D(\cA)\subset\cF_b(\TT^d)$.
This is the reason why we  consider a localization with respect to the values of the intensity function $\lambda$ in the following.
\end{remark}

\begin{remark}\label{rem:Constants}\rm
Splitting tessellations in Euclidean spaces have been introduced in \cite{NW05}, where instead of the probability measure $\kappa$ on $\SS_{d-1}$ a non-normalized
measure on the space of hyperplanes is used for splitting a chosen cell in a bounded observation window. In the current spherical setting, rotation invariance might be considered the appropriate substitute for the translation invariance which is usually assumed in the Euclidean framework. However, the definition of the splitting
process itself, the description in terms of its generator and the martingale properties can be extended to arbitrary regular probability measures $\kappad$ on $\SS_{d-1}$ (instead of the rotation invariant probability measure $\nu_{d-1}$) without additional effort. In the following, in most cases we shall first provide results for
general direction distributions and then specialize these results further.
\end{remark}

\subsection{Auxiliary martingales}\label{subsec:24Martingales}

We use the theory of Markov processes to introduce some classes of martingales associated with the splitting process $(Y_t)_{t\geq 0}$ which is based on a general regular direction distribution $\kappad$ on $\SS_{d-1}$. We start with a preparatory lemma, which is taken from \cite[Proposition (14.13), p.~31]{DavisMarkovModels} to which we also refer for the general definition of a generator  of a Markov process and its domain (see  \cite[Lemma 19.21]{Kallenberg} for the same result under more restrictive assumptions on the state space and the process). The additional assertion concerning jump processes with bounded intensity functions follows from Remark
\ref{rembound}.

\begin{lemma}\label{lem:BasicMartingale}
Let $E$ be a Borel space and let $(X_t)_{t\geq 0}$ be a Markov process with values in $E$ and with generator $\cL$ whose domain is $D(\cL)$. Further, let $f\in D(\cL)$. Then the random process
$$
f(X_t)-f(X_0)-\int_0^t (\cL f)(X_s)\,\dint s\,,\qquad t\geq 0\,,
$$
is a martingale with respect to the filtration induced by $(X_t)_{t\geq 0}$. If $(X_t)_{t\ge 0}$ is a
jump process with bounded intensity function, then $\cF_b(E)= D(\cL)$.
\end{lemma}

In our first application of Lemma \ref{lem:BasicMartingale}, the space $E$ is the space $\TT^d$ of tessellations of $\SSd$. From Lemma \ref{lemmess1} we know that $\TT^d$ is a Borel space as a Borel subset of the Polish space $\cF(\KK^d)$  with the Fell topology, but we do not know whether $\TT^d$ is also a Polish space.
 In view of this, it is important that Lemma \ref{lem:BasicMartingale} is available for a general Borel space $E$ and not just for Polish spaces (which seems to be the common assumption in the literature).  Further,
 $\cL$ will be the generator $\cA$ from Definition \ref{def:GeneratorA}, and $(X_t)_{t\geq 0}$ will be the splitting tessellation process $(Y_t)_{t\geq 0}$. Since we do not know whether $D(\cA)=\cF_b(\TT^d)$ in this setting and since we consider functionals $f=\Sigma_\phi$ which are unbounded, some localization seems to be unavoidable.

\medskip

The next result is (for instance) an  analogue to \cite[Proposition 2]{STBernoulli}. We present a detailed proof in order to fix some inaccuracies in previous proofs. In what follows, we shall write $\filtration_t:=\sigma(Y_s:0\leq s\leq t)$,  $t\geq 0$, for the $\sigma$-field generated by the splitting process until time $t$, and $\filtration:=(\filtration_t)_{t\geq 0}$ for the corresponding filtration. We work with a general regular direction distribution $\kappad$ on $\SS_{d-1}$ which determines the splitting tessellation.

\begin{proposition}\label{prop:Dynkin}
Let $\phi:\PP^d\to\RR$ be bounded and measurable, and define
$$
\Sigma_\phi(T) := \sum_{c\in T}\phi(c)=\int_{\PP^d} \phi \, \dint \mu_T\,,\qquad T\in\TT^d\,.
$$
Then the stochastic process
$$
M_t(\phi):=\Sig_\phi(Y_t)-\Sig_\phi(Y_0)-\int_0^t(\cA\Sig_\phi)(Y_s)\,\dint s\,,\qquad t\geq 0\,,
$$
is a martingale with respect to $\filtration$.
\end{proposition}

\begin{remark}\rm
 In this paper all equalities or inequalities involving random variables are implicitly meant to hold almost surely, referring thereby to the common underlying probability space $(\Omega,\cG,\bP)$ on which all our random objects are defined.
\end{remark}

\begin{proof}[Proof of Proposition \ref{prop:Dynkin}]
To see that $\Sigma_\phi$ is measurable, we extend $\phi$ to $\KK^d$ by setting $\phi(c):= 0$ for $c\in\KK^d\setminus\PP^d$. Then the extension remains bounded and measurable. Let $B\in\cB(\RR)$ and observe that
$$
\{T\in\TT^d:\Sigma_\phi(T)\in B\}=\TT^d\cap i_s\left(  \{\eta\in\sN_s(\KK^d):\int \phi\, \dint\eta\in B\}\right)
\in \TT^d\cap \cB_{lf}.
$$
Since $\phi$ is assumed to be bounded, $\alpha:=\sup\{|\phi(c)|:c\in\PP^d\}<\infty$. For $N\in\NN$, we consider the truncated functional
$$
\Sig_\phi^N(T) := (\Sig_\phi(T)\wedge (N \alpha))\vee(-(N \alpha))\,,\qquad T\in\TT^d\,,
$$
which is measurable and bounded.
In addition to $(Y_t)_{t\ge 0}$, for each $N\in\NN$ we introduce a jump process $(Y_t^N)_{t\ge 0}$ with bounded (truncated) intensity function $\lambda\wedge N$, that is, with transition kernel  $q^N(T,\cdot):=(\lambda(T)\wedge N)\pi(T,\cdot)$ and generator
$\cA^N:=(\lambda\wedge N)(\pi-I)=(\lambda\wedge N)\lambda^{-1}\, \cA$. The processes $(Y_t)_{t\ge 0}$ and $(Y^N_t)_{t\ge 0}$ can be constructed on the same probability space (see \cite[Section 2]{GST}, \cite[Chapter 4]{Eberle}, \cite[Chapter 12]{Kallenberg}).  Let $J_k$, $k\in\NN$, be the time of the $k$th jump of $(Y_t)_{t\ge 0}$ with the convention that $J_0:=0$. By construction we have $Y^N_s=Y_s$ for $s\in [0,J_{N})$.
Let $\filtration^N=(\filtration^N_t)_{t\ge 0}$ denote the filtration  generated by $(Y^N_t)_{t\ge 0}$. By Lemma \ref{lem:BasicMartingale}, since $(Y^N_t)_{t\ge 0}$ has a bounded intensity function and $\Sigma^N_\phi$ is also bounded,
the process
$$
M_t^N(\phi):=\Sig^N_\phi(Y^N_t)-\Sig^N_\phi(Y^N_0)-\int_0^t(\cA^N \Sig^N_\phi)(Y^N_s)\,\dint s\,,\qquad t\geq 0\,,
$$
is a $\filtration^N$-martingale.

\medskip

To relate $M_t^N(\phi)$ to $M_t(\phi)$, for $t\ge 0$, we introduce the almost surely finite random variable
$$
\tau_N:=\inf\{t\ge 0:\lambda(Y_t)\ge N\}.
$$
Then, clearly we have $\tau_N=J_{N-1}$. Since $\{\tau_N>t\}=\{|Y_t|\le N-1 \}=\{|Y_t^N|\le N-1  \}\in\filtration_t\cap\filtration^N_t$, for $t\ge 0$, $\tau_N$ is a stopping time with respect to both filtrations, $\filtration$ and $\filtration^N$. Moreover,
by the optional stopping theorem, $\widetilde M^N_t(\phi):=M^N_{\tau_N\wedge t}(\phi)$, for $t\ge 0$,  also defines a martingale with respect to $\filtration^N$ (see \cite[Theorem 7.1.15]{Strook}).

\medskip

Next, we observe that $\tau_N\wedge t\le J_{N-1}<J_{N}$ and
$$
|\Sigma_\phi(Y_{ \tau_N\wedge t})|\le |Y_{ \tau_N\wedge t}|\alpha\le N\alpha\,,
$$
hence
$$
\Sigma^N_\phi(Y^N_{ \tau_N\wedge t})=\Sigma_\phi(Y_{ \tau_N\wedge t})\,.
$$
If $s<\tau_N$, then $s<J_{N-1}$ and hence $|Y_s|\le N-1$. This yields
$$
|\Sigma_\phi(\oslash(c,S,Y_s^N))|=|\Sigma_\phi(\oslash(c,S,Y_s))|\le N\alpha,\qquad S\in\mathbb{S}_{d-1}\,,
$$
and therefore
$$
\Sigma_\phi^N(\oslash(c,S,Y_s^N))=\Sigma_\phi(\oslash(c,S,Y_s)),\qquad s<\tau_N\,.
$$
Moreover, for $s<\tau_N\wedge t$ we have
\begin{align*}
(\cA^N\Sigma^N_\phi)(Y^N_s)&=\frac{\lambda(Y^N_s)}{\lambda(Y_s)}\int_{\sN_s(\TT^d)} \int_{\SS_{d-1}}
 \left[\Sigma_\phi^N(\oslash(c,S,Y_s^N))-\Sigma_\phi^N(Y^N_s)\right]\, \kappad(\dint S)\,
\mu_{Y^N_s}(\dint c)\\
&=(\cA\Sigma_\phi)(Y_s)\,.
\end{align*}
Thus we obtain
$$\widetilde M^N_t(\phi)=M^N_{\tau_N\wedge t}(\phi)=M_{\tau_N\wedge t}(\phi)=M^{\tau_N}_t(\phi),\qquad t\ge 0\,.
$$
This shows that $(M^{\tau_N}_t(\phi))_{t\ge 0}$ is an $\filtration^N$-martingale. Using that $\tau_N$ is an $\filtration^N$-stopping time and that $Y^N_s=Y_s$ for $s<\tau_N$, it is easy to check that then $(M^{\tau_N}_t(\phi))_{t\ge 0}$ is also a $\filtration$-martingale. Since $\tau_N\to\infty$, as $N\to\infty$, we conclude that $(M_t(\phi))_{t\ge 0}$ is a local $\filtration$-martingale with respect to the localizing sequence $(\tau_N)_{N\in\NN}$.

\medskip

We next argue that the local $\filtration$-martingale $(M_t(\phi))_{t\geq 0}$ is in fact a (proper) martingale by showing that it is of class DL (see \cite[Definition 4.8 and Problem 5.19(i)]{KaratzasShreve}), that is, for each $a>0$ the family
$$
(M_\tau(\phi):\tau\text{ is a stopping time with }\tau\leq a\text{ almost surely})
$$
is uniformly integrable. For this it is sufficient to prove that for each $a>0$,
\begin{equation}\label{eq:SupEstimate}
\bE\sup_{0\leq t\leq a}|M_t(\phi)|<\infty\,,
\end{equation}
since  $|M_\tau(\phi)|\leq\sup_{0\leq t\leq a}|M_t(\phi)|$ almost surely for each stopping time $\tau$ with $\tau\leq a$. To verify \eqref{eq:SupEstimate} we note that, almost surely,
$$
|\Sig_\phi(Y_t)| \leq \alpha|Y_t|\leq\alpha|Y_a|
$$
and, since $\kappad$ is a probability measure,
\begin{align*}
\Big|\int_0^t(\cA\Sig_\phi)(Y_s)\,\dint s\Big|
&\leq\int_0^t \sum_{c\in Y_s}\int_{\SS_{d-1}[c]}|\Sig_\phi(\oslash(c,S,Y_s))-\Sig_\phi(Y_s)|\,\kappad(\dint S) \dint s\\
&\leq\int_0^t\sum_{c\in Y_s}3\alpha\,\dint s\leq 3\alpha a|Y_a|\,.
\end{align*}
During the splitting process all cells divide with a rate less than or equal to one. As a consequence, for all $a\geq 0$ the number of cells of $Y_a$ is stochastically dominated by the number of individuals in a linear pure birth process with birth rate $1$. However, the latter random variable $G_a$ has a geometric distribution with parameter $e^{-a}$, see \cite[Example 5.1.1 and Section 5.11]{ResnickAdventures}. Thus, we have that $|Y_a|\leq G_a$ almost surely, which implies that $\bE|Y_a|^m<\infty$ for all $m\geq 0$ and $a\geq 0$ (see also \cite[Lemma 1]{NW05} for a similar argument in the Euclidean case). Thus, using the triangle inequality we find that
\begin{align*}
\bE\sup_{0\leq t\leq a} |M_t(\phi)| &= \bE\sup_{0\leq t\leq a}\Big|\Sig_\phi(Y_t)-\Sig_\phi(Y_0)-\int_0^t(\cA\Sig_\phi)(Y_s)\,\dint s\Big|\\
&\leq \bE[\alpha|Y_a|+\alpha+3\alpha a|Y_a|]<\infty\,.
\end{align*}
This completes the argument.
\end{proof}

To deal below also with second-order properties of splitting tessellations, we  first note another consequence
of Lemma \ref{lem:BasicMartingale} which can be proved in essentially the same way.

\begin{proposition}\label{prop:Dynkinextended}
Let $\phi_i:\PP^d\to\RR$ for $i\in\{1,2\}$ be bounded and measurable, and define
$$
\Sigma_{\phi_1,\phi_2}(T) := \Sigma_{\phi_1}(T)\,\Sigma_{ \phi_2}(T)\,,\qquad T\in\TT^d\,.
$$
Then the stochastic process
$$
M_t(\phi_1,\phi_2):=\Sig_{\phi_1,\phi_2}(Y_t)-\Sig_{\phi_1,\phi_2}(Y_0)-\int_0^t(\cA\Sig_{\phi_1,\phi_2})(Y_s)\,\dint s\,,\qquad t\geq 0\,,
$$
is a martingale with respect to $\filtration$.
\end{proposition}

\begin{proof} The argument is analogous to the one for Proposition \ref{prop:Dynkin}. Therefore we merely point out
the relevant modifications.
Since $\phi_i$ is bounded, $\alpha_i:=\sup\{|\phi_i(c)|:c\in\PP^d\}<\infty$ for $i\in\{1,2\}$. For $N\in\NN$, we consider the truncated functional
$$
\Sig_{\phi_1,\phi_2}^N(T) := (\Sig_{\phi_1,\phi_2}(T)\wedge (N^2 \alpha_1\alpha_2))\vee(-(N^2 \alpha_1\alpha_2))\,,\qquad T\in\TT^d\,,
$$
which is measurable and bounded. The truncated jump process $(Y_t^N)_{t\ge 0}$ and related quantities (filtrations, stopping times) are defined as before. Then
$$
M_t^N(\phi_1,\phi_2):=\Sig^N_{\phi_1,\phi_2}(Y^N_t)-\Sig^N_{\phi_1,\phi_2}(Y^N_0)-\int_0^t(\cA^N \Sig^N_{\phi_1,\phi_2})(Y^N_s)\,\dint s\,,\qquad t\geq 0\,,
$$
is an $\filtration^N$-martingale, which can be related to $M_t(\phi_1,\phi_2)$ as in the preceding proof. When showing that
$(M_t(\phi_1,\phi_2))_{t\ge 0}$ is indeed a proper martingale (not just a local martingale), we use that $\bE |Y_a|^2<\infty$ and the bounds
$$
|\Sigma_{\phi_1,\phi_2}(Y_t)|\le \alpha_1\alpha_2|Y_a|^2\,,\qquad t\in[0,a]\,,
$$
and
$$
|\Sigma_{\phi_1,\phi_2}(\oslash(c,S,Y_s))-\Sigma_{\phi_1,\phi_2}(Y_s)|\le 6\alpha_1\alpha_2|Y_s|+9\alpha_1\alpha_2\,,
$$
from which we conclude that the moment condition corresponding to \eqref{eq:SupEstimate} is satisfied.
\end{proof}

It is clear that the same approach yields corresponding martingale properties of a variety of functionals.
In order to deal also with covariances of functionals of splitting tessellations, we next consider the family of time-augmented martingales. We write $C^1([0,\infty))$ for the set of all real-valued continuously differentiable functions on $[0,\infty)$, and $C^1_0([0,\infty))\subset C^1([0,\infty))$ for the subset of functions with compact support.

\begin{lemma}\label{lem:BasicMartingale2}
Let $F$ be a Borel space and consider $E:=F\times[0,\infty)$. Let $(X_t)_{t\geq 0}$ be a Markov process with values in $F$ and generator $\cL$ whose domain is $D(\cL)$. Then the random process $(\widehat{X}_t)_{t\geq 0}$ with $\widehat{X}_t=(X_t,t)$ is a Markov process in $E$. Its generator $\widehat{\cL}$ is such that
\begin{equation}\label{eq:ExtendedGenerator}
(\widehat{\cL}g)(x,s) = (\cL g(\,\cdot\,,s))(x)+\frac{\partial g}{\partial s}(x,s)\,,\qquad (x,s)\in E\,,
\end{equation}
for all functions $g\in D(\cL)\otimes C^1_0([0,\infty))$.
 Moreover, the stochastic process
\begin{equation}\label{eq:BasicMartingale2}
N_t(g):=g(X_t,t)-g(X_0,0)-\int_0^t(\cL g(\,\cdot\,,s))(X_s)+\frac{\partial g}{\partial s} (\,\cdot\,,s)(X_s)\,\dint s
\end{equation}
is a martingale with respect to the filtration induced by $(X_t)_{t\geq 0}$ for all functions $g\in D(\cL)\otimes C^1([0,\infty))$.
\end{lemma}

\begin{proof}
Since $F$ and $[0,\infty)$ are Borel spaces, $E$ is also a Borel space. It is clear that $(\widehat{X}_t)_{t\geq 0}$ is a Markov process. Its generator $\widehat{\cL}$ is given by \eqref{eq:ExtendedGenerator}. This can easily be confirmed for functions $g\in D(\cL)\otimes C^1_0([0,\infty))$, using the definition (see \cite[p.~28]{DavisMarkovModels}) of a generator (see also \cite[Section (31.5)]{DavisMarkovModels}). Finally, we apply Lemma \ref{lem:BasicMartingale} to $(\widehat{X}_t)_{t\geq 0}$ to conclude the martingale property of $(N_t(g))_{t\geq 0}$.
\end{proof}

\begin{remark}\rm
The set $ D(\cL)\otimes C^1([0,\infty))$ consists of all functions on $F=E\times[0,\infty)$ which are finite linear combinations of functions $(x,t)\mapsto g_1(x)g_2(t)$, where $g_1\in D(\cL)$  and $g_2\in C^1([0,\infty))$.
A version of Lemma \ref{lem:BasicMartingale2} for jump processes with bounded intensity function is   contained in \cite[Theorem 4.4]{Eberle}. The state space there can be extended to a general Borel space. Moreover, minor adjustments of the arguments there  show that the result indeed holds for all functions $g$ which are measurable and bounded and such that $\frac{\partial g}{\partial s}$ is also bounded. However, this will not be needed in the following.
\end{remark}

From the previous lemma we shall derive martingale properties which are adjusted to the
  subsequent applications involving certain geometric functionals of the cells of the tessellations. The following proposition is the spherical analogue of   \cite[Equation (7.2)]{STSecondOrder}. Again we provide an argument, since previous arguments require some corrections.

\begin{proposition}\label{prop:Dynkin3}
Let $\phi_1,\phi_2:\PP^d\to\RR$ be bounded and measurable, let $b_1,b_2\in C^1([0,\infty))$, and define
$$
\Psi_{\phi_1,\phi_2}(T,t):=(\Sig_{\phi_1}(T)-b_1(t))(\Sig_{\phi_2}(T)-b_2(t))\,,\qquad T\in\TT^d,t\geq 0\,.
$$
Then the random process $N_t(\Psi_{\phi_1,\phi_2})$ given by \eqref{eq:BasicMartingale2} with $(X_t)_{t\ge 0}$ replaced
by $(Y_t)_{t\ge 0}$ and $\cL$ by $\cA$ is a $\filtration$-martingale.
\end{proposition}

\begin{proof}
Note that, for $T\in\TT^d$ and $t\geq 0$,
$$
\Psi_{\phi_1,\phi_2}(T,t)=\Sigma_{\phi_1}(T)\Sigma_{\phi_2}(T)-\Sigma_{\phi_1}(T)b_2(t)-\Sigma_{\phi_2}(T)b_1(t)
+b_1(t)b_2(t)\,.
$$
Using this, the fact that linear combinations of martingales are martingales, the linearity of the generator and Proposition \ref{prop:Dynkinextended}, it remains to show that if $\phi:\PP^d\to\RR$ is bounded and measurable and $b\in C^1([0,\infty))$, then
$$
K_t(\Sigma_\phi):=\Sigma_\phi(Y_t)b(t)-\Sigma_\phi(Y_0)b(0)-\int_0^t  (\cA \Sigma_\phi)(Y_s)b(s)+\Sigma_\phi(Y_s)b'(s)
 \, \dint s\,, \qquad t\ge 0\,,
$$
is a $\filtration$-martingale.

\medskip

Localizing $\Sigma_\phi$ and $(Y_t)_{t\ge 0}$ as in the proof of Proposition \ref{prop:Dynkin}, we can apply Lemma
\ref{lem:BasicMartingale2} to the jump process $(Y_t^N)_{t\ge 0}$, $N\in\NN$, which has bounded intensity function
and hence its generator has the full domain. Proceeding further as in the proof of Proposition \ref{prop:Dynkin}, we first obtain that $(K_t(\Sigma_\phi))_{t\ge 0}$ is a local martingale and then a martingale.
\end{proof}

The following is a special case (confer \cite[Proposition 3.2]{STSecondOrder} for a Euclidean counterpart).

\begin{corollary}\label{prop:Dynkin2}
Let $\phi:\PP^d\to\RR$ be bounded and measurable, let $b\in C^1([0,\infty))$, and define
$$
\Psi_\phi(T,t):=(\Sig_\phi(T)-b(t))^2\,,\qquad T\in\TT^d,t\geq 0\,.
$$
Then the random process $N_t(\Psi_\phi)$ given by \eqref{eq:BasicMartingale2} with $(X_t)_{t\ge 0}$ replaced
by $(Y_t)_{t\ge 0}$ and $\cL$ by $\cA$ is a $\filtration$-martingale.
\end{corollary}

\begin{remark}\rm
The results presented in this section do actually not use that the splitting tessellation process takes values in the space of tessellations of the \textit{sphere} and as such they carry over to the Euclidean set-up as well. In this form the proofs presented here fix a number of technical inaccuracies in the earlier works \cite{STSecondOrder,STBernoulli} about iteration stable (STIT-) tessellations.
\end{remark}

\section{The capacity functional}\label{sec:3CapacityFunctional}

\subsection{Splitting tessellations as random closed sets}

We fix $t\geq 0$ and consider the splitting tessellation $Y_t$ of $\SSd$ with time parameter $t$ which is derived from a general regular direction distribution $\kappad$ on $\SS_{d-1}$. It is convenient for us to consider the random set
\begin{equation}\label{eq:DefZt}
Z_t:=\bigcup_{c\in Y_t}\partial c\,,
\end{equation}
which consists of the union of all cell boundaries $\partial c$ of cells $c$ in $Y_t$. In particular, $Z_0=\emptyset$. We shall show next that $Z_t$ is a random closed subset of $\SSd$ in the usual sense of stochastic geometry. We recall that a random closed set in $\SSd$ is  a measurable map from an underlying probability space $(\Omega,\sigmaalgebra,\bP)$ into the measurable space $(\cF(\SSd),\cB(\cF(\SSd)))$ (see \cite[Definition 3.1.2]{SW}), where the Borel $\s$-field is based on the Fell topology on $\cF(\SSd)$. Moreover, we show the crucial property that $Z_t$ is isotropic if and only if $\kappad=\nu_{d-1}$, that is, $Z_t$ has the same distribution as the rotated random set $\varrho Z_t$ for all $\varrho\in\SO(d+1)$, in this case.

\begin{lemma}\label{lem:RACSandIsotropic}
For every $t\geq 0$, $Z_t$ is a random closed set in $\SSd$. Moreover, $Z_t$ is isotropic if and only if
$\kappad=\nu_{d-1}$.
\end{lemma}

\begin{proof}
 By construction and by the definition of the required $\s$-fields in Section \ref{Sec2.3}, we know that the map $Y_t:(\Omega,\sigmaalgebra,\bP)\to(\sN_s(\KK^d),\cB_{vg})$ is
measurable. Since the map $(\KK^d,\cB(\KK^d))\to
(\cF(\SSd),\cB(\cF(\SSd)))$, $c\mapsto \partial c$, is measurable (see \cite[Theorem 12.2.6]{SW}), the induced map $\partial: (\sN_s(\KK^d),\cB_{vg})\to (\sN_s(\cF(\SSd)),\cB_{vg}^*)$, $\sum_{c}\delta_c\mapsto \sum_{c}\delta_{\partial c}$,
is also measurable (recall that $\delta_c$ denotes the point mass in $c\in\KK^d$, the sum extends over a finite set of spherically convex bodies, and $\cB_{vg}^*$ denotes the Borel $\s$-field induced by the vague topology on $\sN_s(\cF(\SSd))$). This shows that
$\partial \circ Y_t$ is measurable. Moreover, the union map $(\sN_s(\cF(\SSd)),\cB_{vg}^*)\to
 (\cF(\SSd),\cB(\cF(\SSd)))$, $\sum_{F}\delta_{F}\mapsto\bigcup_{F} F$ is measurable,
where the sum and the union extend over the same finite set of $F\in\cF(\SSd)$ (in fact, the proof of \cite[Theorem 3.6.2]{SW} carries over to the sphere). Composing these measurable maps yields the assertion.

\medskip

Next, we show that $Z_t$ is isotropic if $\kappad=\nu_{d-1}$. To verify this, more generally we prove that $Y_t$ is isotropic, which means that $\varrho Y_t$ has the same distribution as $Y_t$ for all $\varrho\in\SO(d+1)$. Here, we write $\varrho T=\{\varrho c:c\in T\}$  for the rotated tessellation $T\in\TT^d$. Recall the definition of the generator $\cA$ of the splitting tessellation process $(Y_t)_{t\geq 0}$ from the previous section. For a bounded and measurable map $f:\TT^d\to\RR$, the rotation invariance of the measure $\nu_{d-1}$ on $\SS_{d-1}$ implies that
\begin{align}
\nonumber (\cA f)(\varrho T) &= \sum_{c\in\varrho T}\int_{\SS_{d-1}[c]}\big[f(\oslash(c,S,\varrho T))-f(\varrho T)\big]\,\nu_{d-1}(\dint S)\\
\nonumber &= \sum_{\varrho^{-1}c\in T}\int_{\SS_{d-1}[c]}\big[f(\oslash(c,S,\varrho T))-f(\varrho T)\big]\,\nu_{d-1}(\dint S)\\
&=\sum_{\bar c\in T}\int_{\SS_{d-1}[\bar c]}\big[
f(\oslash(\varrho \bar c,\varrho S,\varrho T))-f(\varrho T)\big]\,\nu_{d-1}(\dint S) = (\cA (f\circ\varrho))(T)\,.\label{eq:GeneratorEquality}
\end{align}

On the other hand, if $\cA^\varrho$ denotes the generator of the jump process
$(\varrho Y_t)_{t\ge 0}$, the usual definition of the generator involves a uniform limit in $T\in \TT^d$. However, for the following analysis a pointwise limit as considered in \cite[Equation (15.21)]{Breiman} is sufficient. In this sense, we have
\begin{align}
(\cA^\varrho f)(T)&=\lim_{t\downarrow 0}\frac{1}{t}\left(\bE[f(\varrho Y_t)\mid \varrho Y_0=T]-f(T)\right)\nonumber\\
&=\lim_{t\downarrow 0}\frac{1}{t}\left(\bE[(f\circ\varrho) (Y_t)\mid  Y_0=\varrho^{-1}(T)]-(f\circ\varrho)(\varrho^{-1}(T))\right)\nonumber\\
&=(\cA(f\circ\varrho))(\varrho^{-1}(T))\,.\label{eqgeneq2}
\end{align}
Hence, combining \eqref{eq:GeneratorEquality} and \eqref{eqgeneq2} we conclude that $(\cA f)(T)=(\cA^\varrho f)(T)$ for all functions $f$ and all $T\in\TT^d$, which shows that the generators (as defined in \cite{Breiman}) of $(Y_t)_{t\ge 0}$ and $(\varrho Y_t)_{t\ge 0}$
coincide. The assertion now follows from \cite[Proposition 15.38]{Breiman}.

If $(Z_t)_{t\ge 0}$ is derived from a splitting process with direction distribution $\kappad$, then after an exponential waiting time $\tau_1$
with parameter $1$, we have exactly two cells and $Z_{\tau_1}=S$ is the separating great hypersphere with direction distribution $\kappad$. Since by construction and assumption $Z_{\tau_1}$ has a rotation invariant distribution, the assertion follows.
\end{proof}

\begin{remark}\label{Rem3.2}\rm
If $\kappad=\nu_{d-1}$, then the isotropy of $Z_t$ could also be proved by using \cite[Proposition 3.39]{Breiman} and
induction (over $N$). In any case, the rotation invariance of
$\nu_{d-1}$ and the rotation covariance of the construction are the crucial points. However, the preceding proof more generally shows that the distribution
of $Y_t$ is rotation invariant for each $t\ge 0$ provided that $\kappad=\nu_{d-1}$.
\end{remark}

\subsection{Capacity functional for connected sets}

The most basic quantity associated with a random closed set is its capacity functional. We are interested in the capacity functional of $Z_t$ defined by
$$
T_t(C):=\bP(Z_t\cap C\neq\emptyset)\,,\qquad C\in\cC(\SSd)\,,
$$
where $\cC(\SSd)$ is the system of closed subsets of $\SSd$. In other words $T_t(C)$ is the probability that the compact test set $C$ is hit by the random set $Z_t$. We shall first compute the value of
$$
U_t(C):=1-T_t(C) = \bP(Z_t\cap C=\emptyset)
$$
in the case that the set $C$ is connected. This constitutes a direct generalization of \cite[Theorem 3.5]{DeussHoerrmannThaele}, but some adjustments of technical details are necessary.

In the proof of Theorem \ref{prop:Capacity2} below it will be crucial to deal with a splitting tessellation arising as a result after time $t$ of the splitting process starting with a general initial tessellation $T\in\SSd$, recall Remark \ref{rem:OtherInitialTessellation}. We denote such a tessellation by $Y_t^{(T)}$ and by $Z_t^{(T)}$ the corresponding union set. Note that $Z_0^{(T)}=\bigcup_{c\in T}\partial c\neq\emptyset$ unless $T=\{\SSd\}$. As in Lemma \ref{lem:RACSandIsotropic} one shows that $Z_t^{(T)}$ is indeed a random closed subset of $\SSd$ and it is clear that $Z_t^{(T)}$ can only be isotropic if $T=\{\SSd\}$ and $\kappa=\nu_{d-1}$. In analogy to $U_t$ we define
$$
U_t^{(T)}:=\bP(Z_t^{(T)}\cap C=\emptyset)\,,\qquad C\in\cC(\SSd)\,.
$$
In particular, if $T=\{\SSd\}$, then $Y_t^{(T)}=Y_t$, $Z_t^{(T)}=Z_t$ and $U_t^{(T)}(\,\cdot\,)=U_t(\,\cdot\,)$.

\begin{theorem}\label{prop:Capacity}
Let $T\in\TT^d$ and let $C\in\PP^d$ be such that $C$ is contained in the interior of precisely one cell of $ T$. Then
$$U_t^{(T)}(C)=\exp\big(-\kappad(\SS_{d-1}\blk C\brk)\,t\big)\,,\qquad t\ge 0\,,$$
independently of $T$. If  $\kappad$ is absolutely continuous with respect to $\nu_{d-1}$, then this relation holds for all connected sets $C\in\cC(\SSd)$ satisfying $C\subset{\rm int}(c)$ for some $c\in T$. In particular, if $C=\overline{xy}$ is a spherical segment connecting $x,y\in\SSd$ with $\ell(x,y)\leq\pi$ and $\overline{xy}\subset{\rm int}(c)$ for some $c\in T$ and if $\kappad=\nu_{d-1}$, then
$$U_t^{(T)}(\overline{xy})=\exp\Big(-\frac{1}{\pi}\ell(x,y)\,t\Big)\,.$$
\end{theorem}

Before presenting the details of the proof of Theorem \ref{prop:Capacity}, let us define, for fixed $T\in\TT^d$, the random variable
\begin{equation}\label{eq:DefXit}
\xi_t:=\sum_{c\in Y_t^{(T)}}\mathbf{1}(C\subset c),\qquad t\geq 0,\quad C\in\cC(\SSd)\,,
\end{equation}
which takes values in $\NN_0$ (since we consider $T\in\TT^d$ being fixed, we suppress the dependency of $\xi_t$ on $T$ in our notation). The following fact will be needed below.

\begin{lemma}\label{lem:ClaimForCapacity}
Let $\kappad$ and $C$ be as in the statement of Theorem \ref{prop:Capacity}, but $C \neq \emptyset$. Then,
$\bP$-almost surely $\xi_t=\mathbf{1}(Z_t^{(T)}\cap C=\emptyset)$. In particular, $\xi_t\in\{0,1\}$ holds $\bP$-almost surely. If $\xi_t=1$, then $\bP$-almost surely $C$ is contained in the interior of a unique cell of $Y_t^{(T)}$.
\end{lemma}
\begin{proof}
If $Y_t^{(T)}=T$, then $Z_t^{(T)}=\bigcup_{c\in T}\partial c$ and $\xi_t=1=\mathbf{1}(Z_t^{(T)}\cap C=\emptyset)$. So we restrict ourselves to the cases where $Y_t^{(T)}\neq T$ and
therefore $|Y_t^{(T)}|\ge |T|+1\ge 2$.

\medskip

\textit{Step 1:} If $\xi_t\ge 2$, then there are $c_1,c_2\in Y_t^{(T)}$, $c_1\neq c_2$, with
$C\subset c_1$, $C\subset c_2$, hence $C\subset c_1\cap c_2$. This implies that
$C\subset Z_t^{(T)}$. Let $z_0\in C\neq\emptyset$ be arbitrarily fixed. Then we get
$\bP(\xi_t\ge 2)\le \bP(z_0\in Z_t^{(T)})$.
Considering the algorithmic construction of $(Y_t^{(T)})_{t\ge 0}$ (or the distributional description provided in Remark \ref{rem:OtherInitialTessellation}) it is clear that
there is a first instance $s\le t$ for which $z_0\in Z_s^{(T)}$ and $z_0$ lies
on the newly introduced random separating great hypersphere having distribution $\kappad(\SS_{d-1}[c]\cap \ \cdot \ )/\kappad(\SS_{d-1}[c])$.
Since $\kappad$ is regular, this happens with probability zero. This shows that $\bP(\xi_t\ge 2)\le \bP(z_0\in Z_t^{(T)})=0$, see
\cite[Lemma 4.1]{GST}.

\medskip

\textit{Step 2:} If $\xi_t=0$, then $Z_t^{(T)}\cap C\neq\emptyset$, since otherwise $\{C\cap \text{int}(c):c\in Y_t^{(T)}\}$ yields a decomposition of $C$ into two non-empty  relatively open subsets of $C$, which contradicts the assumption that $C$ is connected. Hence, in this case we conclude that $\mathbf{1}(Z_t^{(T)}\cap C=\emptyset)=0$.

\medskip

\textit{Step 3:} If $\xi_t=1$, then there is precisely one cell $c\in Y_t^{(T)}$ such that $C\subset c$.
If $C\not\subset\text{int}(c)$, then there is a first instance $s\le t$  such that
 the newly introduced random separating great hypersphere $S$ having distribution $\kappad(\SS_{d-1}[c]\cap \ \cdot \ )/\kappad(\SS_{d-1}[c])$ at time $s$
 hits $C$, but $C$ is not contained in one of the two open half-spheres determined by $S$. By Lemma \ref{setzero}, this event has measure zero (cf.~the argument in Step 1). Hence, we conclude that
 $C\subset\text{int}(c)$ with probability one. This finally shows that if $\xi_t=1$, then $Z_t^{(T)}\cap C=\emptyset$ is satisfied
$\bP$-almost surely, and then the unique cell which contains $C$ already contains $C$ in its interior.

\medskip

\textit{Step 4:} Conversely, if $Z_t^{(T)}\cap C=\emptyset$ (and $C\neq\emptyset$, $Y_t^{(T)}\neq T$), then $C\subset \bigcup_{c\in Y_t^{(T)}}\text{int}(c)$. Since $C$ is connected, this implies that $C\subset\text{int}(c)$ for exactly one of the cells $c\in Y_t^{(T)}$, in particular,  $\xi_t=1$.
\end{proof}

\begin{proof}[Proof of Theorem \ref{prop:Capacity}] Let $t\ge 0$ be fixed.  Let $T,C$ and $\kappad$ be as stated in the assumptions of the theorem. The assertion is apparently true if $C=\emptyset$. Hence we assume
$C\neq \emptyset$ in the following. \medskip

The map $\phi:\PP^d\to\RR$ given by $\phi(c):={\bf 1}(C\subset c)$ for $c\in\PP^d$ is measurable and bounded. Hence Proposition \ref{prop:Dynkin} and the fact that $Y_0^{(T)}=T$ show that the random process
\begin{equation*}
\begin{split}
\sum_{c\in Y_t^{(T)}}\phi(c)&-\sum_{c\in T}\phi(c)\\
&-\int_0^t\sum_{c\in Y_s^{(T)}}\int_{\SS_{d-1}[c]}[\phi(c\cap S^+)+\phi(c\cap S^-)-\phi(c)]\,\kappad(\dint S)\, \dint s\,,\qquad t\geq 0\,,
\end{split}
\end{equation*}
is a martingale with respect to the natural filtration $\filtration^T$ induced by the random process $(Y_t^{(T)})_{t\geq 0}$.
 Let $\xi_t$ be the $\NN_0$-valued random variable defined in \eqref{eq:DefXit}.
Since $C$ is contained in the interior of precisely one cell of the initial tessellation $T$, we have that $\sum_{c\in T}\phi(c) = 1$ almost surely and we deduce that
\begin{equation*}
\begin{split}
U_t^{(T)}(C)=1+\int_0^t\bE\sum_{c\in Y_s^{(T)}}\int_{\SS_{d-1}[c]}[{\bf 1}(C\subset c\cap S^+)&+{\bf 1}(C\subset c\cap S^-)\\
&-{\bf 1}(C\subset c)]\,\kappad(\dint S)\,\dint s\,.
\end{split}
\end{equation*}
Fix $s\in[0,t]$ and observe that if $\xi_s=0$, that is, if there is no cell $c\in Y_s^{(T)}$ satisfying $C\subset c$, then the integrand of the inner integral is equal to zero. Lemma \ref{lem:ClaimForCapacity}  then implies that if the expression under the expectation is multiplied with $\xi_s$, then the expectation does not change. Moreover, again by Lemma \ref{lem:ClaimForCapacity}, if $\xi_s=1$, then almost surely
 there is a unique cell $c_0\in Y_s^{(T)}$ with $C\subset \text{int}(c_0)$. Hence,  almost surely the expression under the expectation is equal to
\begin{align*}
&\xi_s\sum_{c\in Y_s^{(T)}}\int_{\SS_{d-1}[c]}[{\bf 1}(C\subset c\cap S^+)+{\bf 1}(C\subset c\cap S^-)-{\bf 1}(C\subset c)]\,\kappad(\dint S)\\
&\qquad =\xi_s\int_{\SS_{d-1}[c_0]}[{\bf 1}(C\subset c_0\cap S^+)+{\bf 1}(C\subset c_0\cap S^-)-{\bf 1}(C\subset c_0)]\,\kappad(\dint S)\\
&\qquad =-\xi_s\int_{\SS_{d-1}[c_0]}{\bf 1}(C\cap S\neq\emptyset)\,\kappad(\dint S)\\
&\qquad =-\kappad(\SS_{d-1}\blk C\brk)\, \mathbf{1}(Z_s^{(T)}\cap C=\emptyset)\,.
\end{align*}
To justify the second equality, we distinguish two cases. If $S\cap C=\emptyset$, then either $C\subset\text{int}(S^+)$ or $C\subset\text{int}(S^-)$, which yields the required equality of the integrands.
On the other hand, if $S\cap C\neq\emptyset$, excluding a set of $S\in \SS_{d-1}[c_0]$ of $\kappad$-measure zero, we deduce that $C\not\subset S^+$ and
$C\not\subset S^-$,  by Lemma \ref{setzero}. This again yields the equality of the integrands for
$\kappad$-almost all $S \in \SS_{d-1}[c_0]$.

\medskip

So, we find that
\begin{align*}
&\int_0^t\bE\sum_{c\in Y_s^{(T)}}\int_{\SS_{d-1}[c]}[{\bf 1}(C\subset c\cap S^+)+{\bf 1}(C\subset c\cap S^-)-{\bf 1}(C\subset c)]\,\kappad(\dint S)\,\dint s\\
&\qquad =-\kappad(\SS_{d-1} \blk C\brk)\int_0^t \bP(Z_s^{(T)}\cap C=\emptyset)\,\dint s\\
&\qquad =-\kappad(\SS_{d-1}\blk C\brk)\,\int_0^t U_s^{(T)}(C)\,\dint s\,,
\end{align*}
and hence
$$
U_t^{(T)}(C)=1-\kappad(\SS_{d-1}\blk C\brk)\,\int_0^t U_s^{(T)}(C)\,\dint s\,.
$$
Together with the initial condition $U_0^{(T)}(C)=\bP(Z_0^{(T)}\cap C=\emptyset)=1$, this equation is easily seen to have the unique solution
$$
U_t^{(T)}(C)=\exp\big(-\kappad(\SS_{d-1}\blk C\brk)\,t\big)\,,
$$
independently of $T$. This concludes the first and second part of the proof. The final assertion is a direct consequence of \eqref{eq:InvMeasureLineSegment}.
\end{proof}

\subsection{Capacity functional for sets with more than one connected component}

Let us now turn to the case where the set $C$ has more than one connected component. In this situation one can find a recursion formula for $U_t^{(T)}(C)$. For $d=2$, $T=\{\SSd\}$ and under the assumption of isotropy, this has been shown in \cite[Theorem 3.5]{DeussHoerrmannThaele}. A similar proof carries over to higher dimensional spherical spaces, but we shall add a number of details that were omitted in \cite{DeussHoerrmannThaele}. In particular, our proof is based on a description of the splitting process as developed in \cite{GST} (and adjusted to the present setting in Remark \ref{rem:OtherInitialTessellation}) and is different from the argument for the STIT-model in the $d$-dimensional Euclidean space $\RR^d$ presented in \cite[Lemma 4]{NW05}.

To state the result we need to introduce some further notation. The {\em closed spherically convex hull} $\overline{C}$ of a set $C\subset \SSd$ is defined as $\overline{C}:=\text{cl}\,\{t_1c_1+\cdots+t_mc_m\in\SSd:t_i\ge 0, c_i\in C,m\in\NN\}$ (which is also equal to the intersection of the closure of the convex cone spanned by $C$ with  $\SSd$). Moreover, for two sets $B_1,B_2\subset \SSd$ we let
$$
\SS_{d-1}\blk B_1|B_2\brk  = \{S\in\SS_{d-1}:B_1\cap S=\emptyset=B_2\cap S,\overline{B_1\cup B_2}\cap S\neq\emptyset\}
$$
be the set of great hyperspheres that separate $B_1$ and $B_2$. Finally, if $C=C_1\cup\ldots\cup C_m\in\cC(\SSd)$ is a union of $m\in\NN$ disjoint non-empty connected subsets $C_1,\ldots,C_m\in\cC(\SSd)$, we shall write $\Pi(C)$ for the set of all proper (unordered) partitions $\{P,\widehat{P}\}$ of $C$. That is, $\{P,\widehat{P}\}\in\Pi(C)$ if and only if there exists some proper subset $\emptyset\neq I\subset\{1,\ldots,m\}$ such that $P=\bigcup_{i\in I}C_i$ and $\widehat{P}=\bigcup_{i\in\{1,\ldots,m\}\setminus I}C_i$. Let $\TT^d_0$
denote the set of all tessellations $T$ of $\SS^d$, where either $T=\{\SS^d\}$ or each cell of $T$ is contained in a hemisphere.

\begin{theorem}\label{prop:Capacity2}
Fix $T\in\TT^d_0$ and suppose that $\kappad$ is absolutely continuous with respect to $\nu_{d-1}$.
Let $C\in\cC(\SSd)$ be such that, for some $m\in\NN$, $C=C_1\cup\ldots\cup C_m$ with pairwise disjoint non-empty connected subsets $C_1,\ldots,C_m\in\cC(\SSd)$. Also assume that $C$ is contained in the interior of a single cell of $T$. If $m=2$ with $C_1=\{x\}$ and $C_2=\{-x\}$, for some $x\in\SSd$, then $U_t^{(T)}(C)=1$ for all $t\geq 1$. In all other cases,
\begin{align*}
U_t^{(T)}(C) = e^{-t\kappad(\SS_{d-1}\blk \overline{C}\brk)}+\sum_{\{P,\widehat{P}\}\in\Pi(C)}\kappad(\SS_{d-1}\blk P|\widehat{P} \brk)\int_0^t e^{-s\kappad(\SS_{d-1}\blk \overline{C}\brk)}U_{t-s}(P)U_{t-s}(\widehat{P})\,\dint s
\end{align*}
for $t\geq 0$, independently of $T$.
\end{theorem}
\begin{proof}
We shall prove the recursion formula for $U_t^{(T)}(C)$ and at the same time that $U_t^{(T)}(C)$ is independent of $T$ (as long as $C$ is contained in the interior of a single cell of $T$) by induction over the number $m$ of connected components of $C$. We start with the case that $m=2$ and $C=C_1\cup C_2$ with $C_1=\{x\}$ and $C_2=\{-x\}$ for some $x\in\SSd$. Since $\kappa$ is regular, we conclude that in this case $U_t^{(T)}(C)=1$ for all $t\geq 0$. Next, we assume that $C$ is not of this particular form. Then $C$ contains at least three points, which implies that $\overline{C}$ is connected (since it is pathwise connected). If $t=0$ there is nothing to prove and so we assume that $t>0$. We start by writing
\begin{align}
\nonumber U_t^{(T)}(C) &= \bP(Z_t^{(T)}\cap C = \emptyset)\\
&=\bP(Z_t^{(T)}\cap\overline{C}=\emptyset)+\bP(Z_t^{(T)}\cap C = \emptyset,Z_t^{(T)}\cap\overline{C}\neq\emptyset)\,.\label{eq:CapacityDecomposition}
\end{align}
Since $\overline{C}$ is compact and connected and with $C$ also $\overline{C}$ is contained in the interior of a cell of $T$, Theorem \ref{prop:Capacity} yields that
\begin{equation}\label{eq:Capacity1stTerm}
\bP(Z_t^{(T)}\cap\overline{C}=\emptyset) = e^{-t\kappa(\SS_{d-1}\blk\overline{C}\brk)}\,,
\end{equation}
which is independent of $T$.

To compute the remaining probability in \eqref{eq:CapacityDecomposition}, we now use the explicit description of the distribution of $(Y_u^{(T_a)})_{u\in[a,b]}$ given in Remark \ref{rem:OtherInitialTessellation} to determine the second term in \eqref{eq:CapacityDecomposition}. For this we notice first that the event that $Z_t^{(T)}\cap C=\emptyset$ and $Z_t^{(T)}\cap\overline{C}\neq\emptyset$ can only occur if up to time $t$ there was at least one jump in the splitting process. So, using the 
the distributional description in Remark \ref{rem:OtherInitialTessellation} with $a=0$, $b=t$ and $T_0=T$, writing $Z(Y):=\bigcup_{c\in Y}\partial c$ for the random closed set induced by a tessellation $Y\in\TT^d$,  using a decomposition according to the first time $s_i$ when the set $\overline{C}$ gets hit by a splitting great hypersphere, and applying Fubini's theorem, we deduce that
\begin{align*}
&\bP(Z_t^{(T)}\cap C = \emptyset,Z_t^{(T)}\cap\overline{C}\neq\emptyset) \\
&=\sum_{i=1}^\infty\int\limits_{0}^t\dint s_i\hspace{0.4cm}\int\cdots\hspace{-0.5cm}\int\limits_{\hspace{-0.8cm}\{0=s_0<s_1<\ldots<s_i\}}\dint s_1\ldots\dint s_{i-1}\,\prod_{j=1}^{i-1}\int\phi(\Upsilon_{s_{j-1}}^{(T)};\dint(c_{j-1},S_{j-1}))\,e^{-\int_0^{s_i}\phi(\Upsilon_u^{(T)})\,\dint u}\\
&\qquad\qquad\times{\bf 1}((\Upsilon_u^{(T)})_{u\in[0,s_i]}\in\cD(T;[0,s_i];(s_j,c_j,S_j)_{1\leq j\leq i-1}))\,{\bf 1}(Z(\Upsilon_{s_{i-1}}^{(T)})\cap\overline{C}=\emptyset)\\
&\quad\times\int\phi(\Upsilon_{s_{i-1}}^{(T)};\dint(c_{i-1},S_{i-1}))\,{\bf 1}(Z(\Upsilon_{s_i}^{(T)})\cap\overline{C}\neq\emptyset,Z(\Upsilon_{s_i}^{(T)})\cap C=\emptyset)\\
&\quad\times\sum_{n=i}^\infty\hspace{0.7cm}\int\cdots\hspace{-0.5cm}\int\limits_{\hspace{-0.8cm}\{s_i<s_{i+1}<\ldots<s_n<t\}}\dint s_{i+1}\ldots\dint s_n\prod_{j=i}^n\int\phi(\Upsilon_{s_j}^{(T)};\dint(c_j,S_j))\,e^{-\int_{s_i}^t\phi(\Upsilon_u^{(T)})\,\dint u}\\
&\qquad\qquad\times{\bf 1}(Z(\Upsilon_t^{(T)})\cap C=\emptyset)\,{\bf 1}((\Upsilon_u^{(\Upsilon_{s_i}^{(T)})})_{u\in[s_i,t]}\in\cD(\Upsilon_{s_i}^{(T)};[s_i,t];(s_j,c_j,S_j)_{i\leq j\leq n}))\,.
\end{align*}
Here, we use the convention that for $i=1$ the integration over $\{0=s_0<s_1<\ldots<s_i\}$ is omitted and that the empty product is interpreted as $1$. Similarly, if $i=n$ we understand that the integration over $\{s_i<s_{i+1}<\ldots<s_n<t\}$ is omitted. The idea behind this decomposition is that $s_i$ is the first time when the set $\overline{C}$ gets hit by a splitting great hypersphere, while the intersection with $C$ is still empty. This means that at time $s_i$ the connected components of $C$ are partitioned by the separating great hypersphere $S_i$ according to a partition in $\Pi(C)$. Thereby, the terms before the sum over $n$ describe the evolution of the splitting process before (the $i-1$ integral terms involving $s_1,\ldots,s_{i-1}$) and at time $s_i$ (the integral term in the third line), while the terms thereafter give a description of the splitting process after time $s_i$ and up to time $t$ (the $n-i$ integral terms involving $s_{i+1},\ldots,s_n$).

To simplify the above expression for the probability $\bP(Z_t^{(T)}\cap C = \emptyset,Z_t^{(T)}\cap\overline{C}\neq\emptyset)$, we start by observing that
\begin{align*}
{\bf 1}(Z(\Upsilon_{s_{i-1}}^{(T)})\cap\overline{C}=\emptyset) = {\bf 1}(\overline{C}\subset c\text{ for precisely one }c\in\Upsilon_{s_{i-1}}^{(T)})\,.
\end{align*}
Also note that according to the definition of the measure $\phi(\Upsilon_{s_{i-1}}^{(T)};\,\cdot\,)$, the integral term describing the cell splitting at time $s_i$ can be rewritten as
\begin{align*}
&\int\phi(\Upsilon_{s_{i-1}}^{(T)};\dint(c_{i-1},S_{i-1}))\,{\bf 1}(Z(\Upsilon_{s_i}^{(T)})\cap\overline{C}\neq\emptyset,Z(\Upsilon_{s_i}^{(T)})\cap C=\emptyset)\\
&=\sum_{c\in\Upsilon_{s_{i-1}}^{(T)}}\sum_{\{P,\widehat{P}\}\in\Pi(C)}{\bf 1}(\overline{C}\subset{\rm int}(c))\int_{\SS_{d-1}[c]}{\bf 1}(P\subset \text{int}(c\cap S^+),\widehat{P}\subset \text{int}(c\cap S^-))\,\kappa(\dint S)\\
&=\sum_{c\in\Upsilon_{s_{i-1}}^{(T)}}{\bf 1}(\overline{C}\subset{\rm int}(c))\sum_{(P,\widehat{P})\in\Pi(C)}\kappa(\SS_{d-1}\blk P|\widehat{P}\brk)\,,
\end{align*}
where $S^+$ and $S^-$ denote the two closed half-spheres determined by the great hypersphere $S$. Interchanging summation and integration (as we may, since the integrand is non-negative) and calling $s_i=s$, we thus conclude that $\bP(Z_t^{(T)}\cap C = \emptyset,Z_t^{(T)}\cap\overline{C}\neq\emptyset)$ is equal to
\begin{align*}
&\int_0^t\dint s\sum_{i=1}^\infty\hspace{0.7cm}\int\cdots\hspace{-0.7cm}\int\limits_{\hspace{-0.8cm}\{0=s_0<s_1<\ldots<s_i=s\}}\dint s_1\ldots\dint s_i\,\prod_{j=1}^{i-1}\int\phi(\Upsilon_{s_{j-1}}^{(T)};\dint(c_{j-1},S_{j-1}))\,e^{-\int_0^{s}\phi(\Upsilon_u^{(T_0)})\,\dint u}\\
&\qquad\qquad\times{\bf 1}(\overline{C}\subset c\text{ for precisely one }c\in\Upsilon_{s_{i-1}}^{(T)})\\
&\qquad\qquad\times{\bf 1}((\Upsilon_u^{(T)})_{u\in[0,s]}\in\cD(T;[0,s];(s_j,c_j,S_j)_{1\leq j\leq i-1}))\\
&\qquad\times\sum_{c\in\Upsilon_{s_{i-1}}^{(T)}}{\bf 1}(\overline{C}\subset{\rm int}(c))\sum_{\{P,\widehat{P}\}\in\Pi(C)}\kappa(\SS_{d-1}\blk P|\widehat{P}\brk)\\
&\qquad\times\sum_{n=i}^\infty\hspace{0.7cm}\int\cdots\hspace{-0.7cm}\int\limits_{\hspace{-0.8cm}\{s<s_{i+1}<\ldots<s_n<t\}}\dint s_{i+1}\ldots\dint s_n\prod_{j=i}^n\int\phi(\Upsilon_{s_j}^{(\Upsilon_s^{(T)})};\dint(c_j,S_j))\,e^{-\int_{s}^t\phi(\Upsilon_u^{(\Upsilon_s^{(T)})})\,\dint u}\\
&\qquad\qquad\times{\bf 1}(Z(\Upsilon_t^{({\Upsilon_{s}^{(T)}})})\cap C=\emptyset)\,{\bf 1}((\Upsilon_u^{({\Upsilon_{s}^{(T)}})})_{u\in[s,t]}\in\cD(\Upsilon_{s}^{(T)};[s,t];(s_j,c_j,S_j)_{i\leq j\leq n}))\\
&\qquad\qquad\times{\bf 1}(P\text{ and }\widehat{P}\text{ are contained in the interiors of two different cells of }\Upsilon_t^{({\Upsilon_{s}^{(T)}})})\,.
\end{align*}
Next, we observe that
\begin{align*}
{\bf 1}(Z(\Upsilon_t^{({\Upsilon_{s}^{(T)}})})\cap C=\emptyset) = {\bf 1}(Z(\Upsilon_t^{({\Upsilon_{s}^{(T)}})})\cap P=\emptyset){\bf 1}(Z(\Upsilon_t^{({\Upsilon_{s}^{(T)}})})\cap \widehat{P}=\emptyset)\,.
\end{align*}
We also notice that starting from time $s$ and up to time $t$ the evolution of the splitting process in the two different cells containing $P$ and $\widehat{P}$ is independent and coincides with the evolution of a splitting process whose initial tessellation is equal to the tessellation at time $s$. Thus, after a time shift by $-s$ in each of the integrals with respect to $s_{i+1},\ldots,s_n$ we see that the sum over $n$ from $i$ to $\infty$ in the last term is equal to
\begin{align*}
&\bP(Z_{t-s}^{({\Upsilon_{s}^{(T)}})}\cap P=\emptyset,Z_{t-s}^{({\Upsilon_{s}^{(T)}})}\cap\widehat{P}=\emptyset)\\
&\qquad\times {\bf 1}(P\text{ and }\widehat{P}\text{ are contained in the interiors of two different cells of }\Upsilon_{s}^{(T)})\\
&=\bP(Z_{t-s}^{({\Upsilon_{s}^{(T)}})}\cap P=\emptyset)\bP(Z_{t-s}^{({\Upsilon_{s}^{(T)}})}\cap\widehat{P}=\emptyset)\\
&\qquad\times {\bf 1}(P\text{ and }\widehat{P}\text{ are contained in the interiors of two different cells of }\Upsilon_{s}^{(T)})\\
&=U_{t-s}^{({\Upsilon_{s}^{(T)}})}(P)U_{t-s}^{({\Upsilon_{s}^{(T)}})}(\widehat{P})\\
&\qquad\times {\bf 1}(P\text{ and }\widehat{P}\text{ are contained in the interiors of two different cells of }\Upsilon_{s}^{(T)})\,.
\end{align*}
By construction, the number of connected components of the two sets $P$ and $\widehat{P}$ is strictly less than the number of connected components of $C$. So, by the induction hypothesis $U_{t-s}^{({\Upsilon_{s}^{(T)}})}(P)$ and $U_{t-s}^{({\Upsilon_{s}^{(T)}})}(\widehat{P})$ are independent of ${\Upsilon_{s}^{(T)}}$, as long as $P$ and $\widehat{P}$ are contained in the interiors of two different cells of ${\Upsilon_{s}^{(T)}}$. Plugging this into the last expression for the probability $\bP(Z_t^{(T)}\cap C = \emptyset,Z_t^{(T)}\cap\overline{C}\neq\emptyset)$, we see that
\begin{align*}
&\bP(Z_t^{(T)}\cap C = \emptyset,Z_t^{(T)}\cap\overline{C}\neq\emptyset) \\
&\qquad= \int_0^t\dint s\,\bP(Z_u^{(T)}\cap\overline{C}=\emptyset\text{ for }u\leq s)\sum_{\{P,\widehat{P}\}\in\Pi(C)}\kappa(\SS_{d-1}\blk P|\widehat{P}\brk)\,U_{t-s}(P)U_{t-s}(\widehat{P})\\
&\qquad= \int_0^t\dint s\,\bP(Z_s^{(T)}\cap\overline{C}=\emptyset)\sum_{\{P,\widehat{P}\}\in\Pi(C)}\kappa(\SS_{d-1}\blk P|\widehat{P}\brk)\,U_{t-s}(P)U_{t-s}(\widehat{P})\\
&\qquad= \int_0^t\dint s\,e^{-s\kappa(\SS_{d-1}\blk\overline{C}\brk)}\sum_{\{P,\widehat{P}\}\in\Pi(C)}\kappa(\SS_{d-1}\blk P|\widehat{P}\brk)\,U_{t-s}(P)U_{t-s}(\widehat{P})\,.
\end{align*}
In combination with \eqref{eq:CapacityDecomposition} and \eqref{eq:Capacity1stTerm} this proves the formula and at the same time the claim that $U_t^{(T)}(C)$ is independent of $C$ if $C$ is contained in the interior of a single cell of the initial tessellation $T$. The proof of the theorem is thus complete.
\end{proof}

In particular, Theorem \ref{prop:Capacity2} allows a recursive computation of $U_t(C)$ if $C\in\cC(\SSd)$ is a union of $m\in\NN$ disjoint and connected closed subsets of $\SSd$. However, the resulting formulas become quickly rather involved when $m$ is large. We illustrate the method by carrying out the first step of the recursion, that is, by taking $m=2$.

\begin{example}{\rm
Let us assume that $C=C_1\cup C_2\in\cC(\SSd)$ is a union of two disjoint non-empty connected components $C_1,C_2\in\cC(\SSd)$ and $C$ contains at least three points. Then in Theorem \ref{prop:Capacity2} the only possibility for the partition $P_1$ and $P_2$ is $P_1=C_1$ and $P_2=C_2$. So, combined with Theorem \ref{prop:Capacity} we see that, for $t\geq 0$,
\begin{align*}
U_t(C) &= e^{-t\kappa(\SS_{d-1}\blk\overline{C}\brk)}+\kappa(\SS_{d-1}\blk C_1|C_2\brk)\int_0^t e^{-s\kappa(\SS_{d-1}\blk\overline{C}\brk)}U_{t-s}(C_1)U_{t-s}(C_2)\,\dint s\\
&= e^{-t\kappa(\SS_{d-1}\blk\overline{C}\brk)}+\kappa(\SS_{d-1}\blk C_1|C_2\brk)\int_0^t e^{-s\kappa(\SS_{d-1}\blk\overline{C}\brk)}\,e^{-(t-s)\kappa(\SS_{d-1}\blk C_1\brk)}\\
&\hspace{8cm}\times e^{-(t-s)\kappa(\SS_{d-1}\blk C_2\brk)}\,\dint s\\
&= e^{-t\kappa(\SS_{d-1}\blk\overline{C}\brk)}+\kappa(\SS_{d-1}\blk C_1|C_2\brk)\,e^{-t(\kappa(\SS_{d-1}\blk C_1\brk)+\kappa(\SS_{d-1}\blk C_2\brk))}\\
&\hspace{4cm}\times\int_0^t e^{s(\kappa(\SS_{d-1}\blk C_1\brk)+\kappa(\SS_{d-1}\blk C_2\brk)-\kappa(\SS_{d-1}\blk\overline{C}\brk))}\,\dint s\\
&=e^{-t\kappa(\SS_{d-1}\blk\overline{C}\brk)}+\kappa(\SS_{d-1}\blk C_1|C_2\brk)\frac{e^{-t\kappa(\SS_{d-1}\blk\overline{C}\brk)}-e^{-t(\kappa(\SS_{d-1}\blk C_1\brk)+\kappa(\SS_{d-1}\blk C_2\brk))}}{ \kappa(\SS_{d-1}\blk C_1\brk)+\kappa(\SS_{d-1}\blk C_2\brk)-\kappa(\SS_{d-1}\blk\overline{C}\brk)}\,,
\end{align*}
provided that $\kappa(\SS_{d-1}\blk C_1\brk)+\kappa(\SS_{d-1}\blk C_2\brk)-\kappa(\SS_{d-1}\blk\overline{C}\brk)\neq 0$. On the other hand, if $\kappa(\SS_{d-1}\blk C_1\brk)+\kappa(\SS_{d-1}\blk C_2\brk)-\kappa(\SS_{d-1}\blk\overline{C}\brk)=0$, then the computation above shows that
\begin{align*}
U_t(C) &= e^{-t\kappa(\SS_{d-1}\blk\overline{C}\brk)}+t\kappa(\SS_{d-1}\blk C_1|C_2\brk)e^{-t(\kappa(\SS_{d-1}\blk C_1\brk)+\kappa(\SS_{d-1}\blk C_2\brk))}
\end{align*}
for $t\geq 0$.}
\end{example}

Taking the initial tessellation $T$ to be equal to $\{\SSd\}$, Theorem \ref{prop:Capacity} and Theorem \ref{prop:Capacity2} especially provide a description of the capacity functional $T_t$ (respectively\ $U_t$) of the random closed set $Z_t$ on the class of sets consisting of finite unions of pairwise disjoint connected subsets of $\SSd$. We remark that this class of subsets of $\SSd$ is in fact a separating class, that is to say, it is rich enough to determine the capacity functional $T_t(C)$ uniquely for all $C\in\cC(\SSd)$. Indeed, the sets consisting of finite unions of pairwise disjoint connected open subsets of $\SSd$ form a base of the standard topology on $\SSd$ and so the claim follows from \cite[Proposition 1.1.53]{Molchanov} (in the statement of this result,
$\mathcal{K}_0$ should be called a pre-separating class, as is clear from the discussion in \cite[Section 1.1.5]{Molchanov}).

\begin{remark}\rm
Theorem \ref{prop:Capacity} and Theorem \ref{prop:Capacity2} together imply that if $\kappad=\nu_{d-1}$, then the capacity functional  $T_t$ of $Z_t$  satisfies $T_t(\varrho C)=T_t(C)$, $C\in\cC(\SSd)$, for all $\varrho\in\SO(d+1)$.  This again yields the isotropy of $Z_t$, which has already been proved in Lemma \ref{lem:RACSandIsotropic} in a more direct way.
\end{remark}

\section{Expected spherical curvature measures}\label{sec:4Expectations}

In this section we consider the expectation of the sum of all localized spherical intrinsic volumes, where the sum runs over all cells of a splitting tessellation with time parameter $t\geq 0$. Formally, we define for $t\geq 0$, $j\in\{0,\ldots,d\}$ and $A\in\cB(\SSd)$ the random variables
$$
\Sig_j(t;A) := \sum_{c\in Y_t}\phi_j(c,A)\,,
$$
where $(Y_t)_{t\ge 0}$  is a splitting tessellation process based on a regular direction distribution $\kappad$ on $\SS_{d-1}$.
The next theorem provides an exact formula for the expectation of $\Sig_j(t;A)$. More generally, we will consider the following set-up. Let $h:\SSd\to\RR$ be bounded and Borel measurable. For a finite Borel measure $\mu$ on $\SSd$, we write
$$
\mu(h):=\int_{\SSd}h\, \dint \mu\,.
$$
In particular, this notation will be applied in writing $\phi_j(c,h)$, $\Sigma_j(t;h)$ and $\cH^d(h)$. We notice that $\Sig_j(t;h)$ reduces to $\Sig_j(t;A)$ for the special choice $h={\bf 1}_A$ with $A\in\cB(\SSd)$.

\begin{theorem}\label{thm:Expectation}
Let $t\geq 0$ and $j\in\{0,\ldots,d\}$. If $\kappad=\nu_{d-1}$, then
$$
\bE\Sig_j(t;h)=\frac{t^{d-j}}{ (d-j)!}\,\frac{\cH^d(h)}{\b_d}\,,
$$
where  $h:\SS^d\to\RR$ is bounded and measurable. For  $j=d$ the result holds for a general
regular direction distribution $\kappad$.
\end{theorem}

\begin{proof}
The case $j=d$ is obviously true for a general regular direction distribution $\kappad$. Hence,
let $j\in\{0,\ldots,d-1\}$ and $A\in\cB(\SSd)$. Using that $V_j(\SSd)=\phi_j(\SSd,\cdot)=0$ for $j\in\{0,\ldots,d-1\}$
and the martingale property stated in Proposition \ref{prop:Dynkin}, with the bounded  and measurable functional
$\phi(c) = \phi_j(c,h)$, $ c\in\PP^d$,
we see that the random process
\begin{equation}\label{eq:ExpectationStart}
\Sig_j(t;h) - \int_0^t\sum_{c\in Y_s}\int_{\SS_{d-1}[c]}[\phi_j(c\cap S^+,h)+
\phi_j(c\cap S^-,h)-\phi_j(c,h)]\,\kappad(\dint S)\,\dint s\,,\quad t\geq 0\,,
\end{equation}
is a $\filtration$-martingale (that is, a martingale with respect to the filtration induced by the splitting process $(Y_t)_{t\geq 0}$). The valuation property \eqref{eqvaluation} of the localized spherical intrinsic volumes yields that
$$
\phi_j(c\cap S^+,h)+\phi_j(c\cap S^-,h)-\phi_j(c,h) = \phi_j((c\cap S^+)\cap(c\cap S^-),h) = \phi_j(c\cap S,h )\,.
$$
Thus, taking expectations in \eqref{eq:ExpectationStart}, we deduce that
\begin{equation}\label{stillgeneral}
\bE\Sig_j(t;h) = \bE\int_0^t\sum_{c\in Y_s}\int_{\SS_{d-1}[c]}\phi_j(c\cap S,h)\,\kappad(\dint S)\,\dint s\,.
\end{equation}
While \eqref{stillgeneral} holds for a general regular direction distribution $\kappad$, the subsequent application of the Crofton formula requires that we specify $\kappad=\nu_{d-1}$. Then, an application of the local spherical Crofton formula \eqref{eq:CroftonOnSphere2} yields
$$
\bE\Sig_j(t;h) = \bE\int_0^t\sum_{c\in Y_s}\phi_{j+1}(c,h)\,\dint s = \bE\int_0^t\Sig_{j+1}(s;h)\,\dint s\,.\nonumber
$$
Continuing this recursion  we eventually reach the functional $\Sig_{d}$, and using Fubini's theorem, we arrive at
\begin{align*}
\bE\Sig_j(t;h) = \int_0^t\int_0^{s_1}\cdots\int_0^{s_{d-j-1}}\bE\Sig_{d}(s_{d-j};h)\,\dint s_{d-j}\ldots\dint s_1\,.
\end{align*}
Thus it remains to compute $\bE\Sig_{d}(s;h)$. However, with probability one,
$$
\Sig_{d}(s;h)=\sum_{c\in Y_s}\phi_d(c,h)=\sum_{c\in Y_s}\frac{\cH^{d}(h\mathbf{1}_c)}{\b_d}=\frac{\cH^{d}(h)}{\b_d}\,,
$$
since $Y_s$ is almost surely a tessellation for each $s$. This immediately implies that
$$
\bE\Sig_j(t;h) = \frac{\cH^{d}(h)}{\b_d}\int_0^t\int_0^{s_1}\cdots\int_0^{s_{d-j-1}}\,
\dint s_{d-j}\ldots\dint s_1 = \frac{t^{d-j}}{ (d-j)!}\,\frac{\cH^{d}(h)}{\b_d}\,,
$$
which completes the proof.
\end{proof}

For $j\in\{d-1,d\}$, the preceding theorem has an extension to general regular direction distributions. This has
already been remarked for $j=d$ in  Theorem \ref{thm:Expectation}. To treat the case $j=d-1$, for
a regular probability measure $\kappad$ on $\SS_{d-1}$ and a bounded measurable function $h:\SSd\to\RR$, we define
\begin{equation}\label{defoverkappad}
 {\overkappad}(h):=\frac{1}{\beta_{d-1}}\int_{\SS_{d-1}}\int_S h(u)\, \cH^{d-1}(\dint u)\, \kappad(\dint S)\,.
\end{equation}
Since
\begin{equation}\label{eq:IntegrationsformelBeweis}
\frac{1}{\beta_{d-1}}\int_{\SS_{d-1}}\int_S h(u)\, \cH^{d-1}(\dint u)\, \nu_{d-1}(\dint S)
=\frac{1}{\beta_d}\int_{\SSd}h(u)\, \cH^d(\dint u)\,,
\end{equation}
as both sides of the equation define rotation invariant probability measures on $\SSd$, we obtain
${\overnud}=\beta_d^{-1}\cH^d$, which shows that the following theorem is consistent with Theorem \ref{thm:Expectation}
if $j=d-1$.

\begin{theorem}\label{thmspaet1}
If $\kappad$ is a regular direction distribution, and $h:\SSd\to\RR$ is bounded and measurable, then $\EE\Sigma_{d-1}(t;h)=t\,\overkappad(h)$  for $t\ge 0$.
\end{theorem}

\begin{proof}
By \eqref{stillgeneral}, we have
$$
\bE\Sig_{d-1}(t;h) = \bE\int_0^t\sum_{c\in Y_s}\int_{\SS_{d-1}[c]}\phi_{d-1}(c\cap S,h)\,\kappad(\dint S)\,\dint s\,.
$$
Since $\kappad(\{S\in\SS_{d-1}:S\cap c\neq\emptyset, S\cap\text{int}(c)=\emptyset\})=0$ for $c\in\PP^d$, by
Lemma \ref{setzero}, the integration over $\SS_{d-1}[c]$, for all $c\in Y_s$ (for each $s$ and each realization of $Y_s$), can be extended over the larger domain $\SS_{d-1}$ without changing
the integral, and then integration and summation can be interchanged. Using again Lemma \ref{setzero}, it follows
that for $\kappad$-almost all $S\in\SS_{d-1}$, $S\cap c=\emptyset$ or $S\cap\text{int}(c)\neq\emptyset$, so that the
non-empty sets $S\cap c$, $c\in Y_s$, partition $S$ in the sense that for any two such sets the intersection has $\cH^{d-1}$-measure zero.
Hence,
\begin{align}
\int_{\SS_{d-1}}\sum_{c\in Y_s}\phi_{d-1}(c\cap S,h)\, \kappad(\dint S)&=\int_{\SS_{d-1}}\sum_{c\in Y_s}\frac{1}{\beta_{d-1}}\int_{c\cap S}h(u)\, \cH^{d-1}(\dint u)\, \kappad(\dint S)\nonumber\\
&=\frac{1}{\beta_{d-1}}\int_{\SS_{d-1}}\int_{S}h(u)\, \cH^{d-1}(\dint u)\, \kappad(\dint S)=\overkappad(h)\,.\label{refhilf}
\end{align}
Combining these arguments, we finally obtain
$$
\bE\Sig_{d-1}(t;h)=\frac{1}{\beta_{d-1}}\int_0^t\int_{\SS_{d-1}}\int_{S}h(u)\, \cH^{d-1}(\dint u)\, \kappad(\dint S)\, \dint s
=t\,\overkappad(h)\,,
$$
which proves the assertion.
\end{proof}

A quantity of particular interest is the total $(d-1)$-dimensional
Hausdorff measure of all great hyperspherical pieces that have been constructed
by the splitting process up to time $t$ within a set $A\in\cB(\SSd)$. Formally, we define
\begin{equation}\label{eq:DefTotalSurfaceArea}
\cH^{d-1}(t;A):=\frac{1}{ 2}\sum_{c\in Y_t}\cH^{d-1}(\partial c\cap A)=\cH^{d-1}(Z_t\cap A)\,,
\end{equation}
where we recall from \eqref{eq:DefZt} that $Z_t$ is the random closed set induced by the splitting tessellation $Y_t$ which is based on a regular
direction distribution $\kappad$.
Using Theorem \ref{thmspaet1}, we can easily compute the expectation of $\cH^{d-1}(t;A)$. More generally, we compute the expected $h$-weighted total Hausdorff measure of $Z_t$.

\begin{corollary}\label{cor:HausdorffMeasure}
Let $\kappad$ be a regular direction distribution and $t\ge 0$. Let $h:\SSd\to\RR$ be bounded and measurable. Then
$$
\bE\cH^{d-1}(t;h)=t\beta_{d-1}\overkappad(h)=t\int_{\SS_{d-1}}\int_{S}h(u)\, \cH^{d-1}(\dint u)\, \kappad(\dint S)\,.
$$
In particular,  $\bE \cH^{d-1}(Z_t)=\beta_{d-1}t$ for $t\ge 0$.
If $\kappad=\nu_{d-1}$, then
$$
\bE\cH^{d-1}(t;h)=\frac{\beta_{d-1}}{\b_d}\cH^d(h)\,t\,.
$$
\end{corollary}
\begin{proof}
Relations \eqref{locbdmeasure} and  \eqref{eq:DefTotalSurfaceArea}  imply that
\begin{equation}\label{basicrelation}
\cH^{d-1}(t;\,\cdot\,)=(\cH^{d-1}\llcorner Z_t )(\,\cdot\,) =\beta_{d-1}\Sigma_{d-1}(t;\,\cdot\,)\,.
\end{equation}
In particular, we get
$\bE\cH^{d-1}(t;h)=\beta_{d-1}\bE\Sig_{d-1}(t;h)$. Hence the assertion follows from
Theorems \ref{thm:Expectation} and \ref{thmspaet1}.
\end{proof}

First-order properties, that is expectations, of Euclidean intrinsic volumes asso\-ciated with STIT-tessellation in a bounded window in $\RR^d$ have been studied in \cite{STSTITPlane,STSTITHigher}. While \cite{STSTITPlane} treats the planar case, but general direction distributions, the paper  \cite{STSTITHigher} is restricted to translation and rotation invariant direction distributions in general dimensions.   The more general case of localized intrinsic volumes has not been investigated in the Euclidean setting. A comparison of these results with Theorems \ref{thm:Expectation} and \ref{thmspaet1} and Corollary \ref{cor:HausdorffMeasure} shows that -- up to dimension dependent constants (see Remark \ref{rem:Constants}) -- the results for STIT-tessellations in $\RR^d$ and splitting tessellations of $\SSd$ are the same. This means that first-order properties are not sensitive enough to `feel' the curvature of the underlying space. This will change with the analysis of second-order parameters in the next section.

\section{Variances and covariances}

After having investigated the expectation of the functionals $\Sig_{j}(t;h)$, for $t\ge 0$ and a bounded measurable
function $h:\SSd\to\RR$, our next goal is to analyse their variances as well as the covariances of $\Sig_i(t;h)$ and $\Sig_j(t;h)$ for $i\neq j$. We shall start with $\Sig_{d-1}(t;h)$ and then turn to the general case.

\subsection{A spherical integral-geometric transformation formula}\label{subsec:IntegralGeometricTrafo}

As described above, our goal is to establish a formula for the variance of $\Sig_{d-1}(t;h)$ for regular direction distributions $\kappad$ on $\SS_{d-1}$. In the special case $\kappad=\nu_{d-1}$, the result has a simpler form, which is due to the following spherical integral-geometric transformation formula of Blaschke-Petkantschin type. In principle, such a result could be derived from the very general kinematic formulas in \cite{ArbeiterZaehle,JensenKieu} that have been obtained using tools from geometric measure theory. However, we prefer to give an elementary and direct proof, which is  based on the linear Blaschke-Petkantschin formula in Euclidean spaces (for which an elementary proof is available).

\begin{proposition}\label{prop:Trafo}
Let $d\ge 2$.
If $g:\SSd\times\SSd\to[0,\infty]$ is  measurable, then
\begin{equation}\label{eq:BlaPetSphere}
\begin{split}
& \int_{\SS_{d-1}}\int_S\int_S g(x,y)\,\cH^{d-1}(\dint x)\,\cH^{d-1}(\dint y)\,\nu_{d-1}(\dint S)\\
&\qquad=\frac{\b_{d-2}}{\b_d}\int_{\SSd}\int_{\SSd}g(x,y)\,\sin(\ell(x,y))^{-1}\,\cH^d(\dint x)\,\cH^d(\dint y)\,.
\end{split}
\end{equation}
\end{proposition}

\begin{remark} \rm
Observe that the integrals on both sides of \eqref{eq:BlaPetSphere} are well defined, but possibly they are both infinite. However, since the left-hand side is finite if $g$ is bounded, the same is true for the integral on the right-hand side (in spite of the unbounded integrand $\sin(\ell(x,y))^{-1}$, which remains undefined on a set of measure zero).
\end{remark}

We prepare the proof of Proposition \ref{prop:Trafo} with the following spherical version of the linear Blasch\-ke-Petkantschin formula. For $d$-fold integrals over $\SSd$ this has been proved in \cite[Lemma 3.2]{BaranyHugReitznerSchneider} by similar arguments. In what follows, $\Grass(d+1,q)$ denotes the space of $q$-dimensional linear subspaces of $\RRd1$ together  with the rotation invariant Haar probability measure whose infinitesimal element is simply denoted by $\dint L$. Moreover,  $\nabla_q(x_1,\ldots,x_q)$ stands for the $q$-volume of the parallelepiped spanned by $q$ vectors $x_1,\ldots,x_q\in \RR^{d+1}$.

\begin{lemma}\label{lem:SphericalBlaschkePetkantschin}
Fix $q\in\{1,\ldots,d\}$. If $f:(\SSd)^q\to [0,\infty]$ is measurable, then
\begin{align*}
&\int_{(\SSd)^q}f(u_1,\ldots,u_q)\,\cH^{qd}(\dint(u_1,\ldots,u_q))\\
&\quad= \frac{\b_{d+1-q}\cdots\b_d}{\b_0\cdots\b_{q-1}}\int_{\Grass(d+1,q)}
\int_{(L\cap\SSd)^q}f(u_1,\ldots,u_q)\\[2mm]
&\hspace{4cm}\times\nabla_q(u_1,\ldots,u_q)^{d-q+1}\,
\cH^{q(q-1)}(\dint(u_1,\ldots,u_q))\,\dint L\,.
\end{align*}
\end{lemma}

\begin{proof}
We use the linear Blaschke-Petkantschin formula in $\RRd1$ from \cite[Theorem 7.2.1]{SW} with $d$ replaced by $d+1$ there. In our notation it says that
\begin{equation}\label{eq:BlPetKlassisch}
\begin{split}
&\int_{(\RRd1)^q}g(x_1,\ldots,x_q)\,\cH^{q(d+1)}(\dint(x_1,\ldots,x_q))\\
&\qquad= \frac{\b_{d+1-q}\cdots\b_d}{\b_0\cdots\b_{q-1}}\int_{\Grass(d+1,q)}\int_{L^q}g(x_1,\ldots,x_q)\\[2mm]
&\hspace{4cm}\times\nabla_q(x_1,\ldots,x_q)^{d-q+1}\,\cH^{q^2}(\dint(x_1,\ldots,x_q))\,\dint L
\end{split}
\end{equation}
whenever $g:(\RRd1)^q\to[0,\infty]$ is  measurable. Let $h_1,\ldots,h_q:[0,\infty)\to[0,\infty)$ be measurable functions which satisfy
$$
\int_0^\infty h_j(r)\,r^d\,\dint r =1\,,\qquad j\in\{1,\ldots,q\}\,,
$$
and let $f:(\SSd)^q\to [0,\infty)$ be measurable. Now, apply \eqref{eq:BlPetKlassisch} to
$$
g(x_1,\ldots,x_q)=f\Big(\frac{x_1}{ \|x_1\|},\ldots,\frac{x_q}{ \|x_q\|}\Big)\,\prod_{j=1}^qh_j(\|x_j\|)\,,\qquad x_1,\ldots,x_q\in\RRd1\setminus\{o\}\,,
$$
and if $x_i=o$ for some $i\in\{x_1,\ldots,q\}$, then we define $g$ as zero (say).
By our assumption on $h_1,\ldots,h_q$ and using spherical coordinates in $\RRd1$ and $L\in \Grass(d+1,q)$, respectively, for the left-hand side of \eqref{eq:BlPetKlassisch} we get
\begin{align*}
&\int_{(\RRd1)^q}g(x_1,\ldots,x_q)\,\cH^{q(d+1)}(\dint(x_1,\ldots,x_q))\\
&= \int_{(\SSd)^q} f(u_1,\ldots,u_q) \bigg[\int_{(0,\infty)^q} \prod_{j=1}^qh_j(s_j)\,s_j^d\,\cH^{q}(\dint(s_1,\ldots,s_q))\bigg]\,\cH^{qd}(\dint(u_1,\ldots,u_q))\\
&=\int_{(\SSd)^q}f(u_1,\ldots,u_q)\,\cH^{qd}(\dint(u_1,\ldots,u_q))\,,
\end{align*}
while for the right-hand side we obtain
\begin{align*}
&\frac{\b_{d+1-q}\cdots\b_d}{\b_0\cdots\b_{q-1}}\int_{\Grass(d+1,q)}\int_{L^q}g(x_1,\ldots,x_q)\,
\nabla_q(x_1,\ldots,x_q)^{d-q+1}\,\cH^{q(d+1)}(\dint(x_1,\ldots,x_q))\,\dint L\\
&=\frac{\b_{d+1-q}\cdots\b_d}{\b_0\cdots\b_{q-1}}\int_{\Grass(d+1,q)}\int_{(L\cap\SSd)^q}\int_{(0,\infty)^q} f(u_1,\ldots,u_q)\,\prod_{j=1}^q h_j(s_j)\,s_j^{q-1}\\
&\qquad\qquad\times\nabla_q(s_1u_1,\ldots,s_qu_q)^{d-q+1}\,
\cH^{q}(\dint(s_1,\ldots,s_j))\,\cH^{q(q-1)}(\dint(u_1,\ldots,u_q))\,\dint L\\
&=\frac{\b_{d+1-q}\cdots\b_d}{\b_0\cdots\b_{q-1}}\int_{\Grass(d+1,q)}\int_{(L\cap\SSd)^q} f(u_1,\ldots,u_q)\,\nabla_q(u_1,\ldots,u_q)^{d-q+1}\\
&\qquad\qquad\times\bigg[\int_{(0,\infty)^q} \prod_{j=1}^q h_j(s_j) s_j^d\,\cH^{q}(\dint(s_1,\ldots,s_q))\bigg]\cH^{q(q-1)}(\dint(u_1,\ldots,u_q))\,\dint L\\
&=\frac{\b_{d+1-q}\cdots\b_d}{\b_0\cdots\b_{q-1}}\int_{\Grass(d+1,q)}
\int_{(L\cap\SSd)^q}f(u_1,\ldots,u_q)\,\nabla_q(u_1,\ldots,u_q)^{d-q+1}\,\cH^{q(q-1)}(\dint u)\,\dint L\,,
\end{align*}
since $\nabla_q(s_1u,\ldots,s_qu_q)=s_1\cdots s_q \,\nabla_q(u_1,\ldots,u_q)$. This proves the formula.
\end{proof}

\begin{proof}[Proof of Proposition \ref{prop:Trafo}]
To derive \eqref{eq:BlaPetSphere}, we apply the spherical Blaschke-Petkantschin formula from Lemma \ref{lem:SphericalBlaschkePetkantschin} with $q=2$ to the function
$$
f(x,y):=\sin(\ell(x,y))^k\,g(x,y)\,,
$$
where $g:\SSd\times\SSd\to[0,\infty]$ is measurable  and $k\ge -d+1$. Since $\SSd\cap L$ is an element of $\SS_1$, this gives
\begin{equation}\label{eq:Trafo1}
\begin{split}
& \int_{\SSd}\int_{\SSd}g(x,y)\,\sin(\ell(x,y))^k\,\cH^d(\dint x)\,\cH^d(\dint y)\\
&\qquad=\frac{\b_d\b_{d-1}}{ 4\pi}\int_{\SS_1}\int_{T}\int_{T}g(x,y)\,\sin(\ell(x,y))^{d+k-1}
\,\cH^1(\dint x)\,\cH^1(\dint y)\,\nu_1(\dint T)\,,
\end{split}
\end{equation}
since $\nabla_2(x,y)=\sin(\ell(x,y))$. Especially for $k=-1$ we find that
\begin{equation}\label{eq:Trafo1a}
\begin{split}
& \int_{\SSd}\int_{\SSd}g(x,y)\,\sin(\ell(x,y))^{-1}\,\cH^d(\dint x)\,\cH^d(\dint y)\\
&\qquad=\frac{\b_d\b_{d-1}}{ 4\pi}\int_{\SS_1}\int_{T}\int_{T}g(x,y)\,\sin(\ell(x,y))^{d-2}\,
\cH^1(\dint x)\,\cH^1(\dint y)\,\nu_1(\dint T)\,.
\end{split}
\end{equation}
We now apply \eqref{eq:Trafo1} once again, but this time with $k=0$, with $d$ replaced by $d-1$ and with
$\SSd$ replaced by some fixed great hypersphere $S\in\SS_{d-1}$. This gives
\begin{align*}
\begin{split}
& \int_{S}\int_{S}g(x,y)\,\cH^{d-1}(\dint x)\,\cH^{d-1}(\dint y)\\
&\qquad=\frac{\b_{d-1}\b_{d-2}}{ 4\pi}\int_{\SS_1[S]}\int_{T}\int_{T}g(x,y)\,
\sin(\ell(x,y))^{d-2}\,\cH^1(\dint x)\,\cH^1(\dint y)\,\nu_1^S(\dint T)\,,
\end{split}
\end{align*}
where $\SS_1[S]$ denotes the set of all elements $T\in\SS_1$ satisfying $T\subset S$ and $\nu_1^S$ is the corresponding invariant Haar probability measure. Note that for $d=2$
this holds trivially and for $d\ge 3$ we have $d-1\ge 2$ so that \eqref{eq:Trafo1} can indeed be applied.
Integrating over all $S\in\SS_{d-1}$, we thus get
\begin{align}
\nonumber & \int_{\SS_{d-1}}\int_S\int_S g(x,y)\,\cH^{d-1}(\dint x)\,\cH^{d-1}(\dint y)\,\nu_{d-1}(\dint S)\\
\nonumber &\qquad=\frac{\b_{d-1}\b_{d-2}}{ 4\pi}\int_{\SS_{d-1}}\int_{\SS_1[S]}\int_{T}\int_{T}g(x,y)\\[2mm]
\nonumber &\hspace{3cm}\times\sin(\ell(x,y))^{d-2}\,\cH^1(\dint x)\,\cH^1(\dint y)\,\nu_1^S(\dint T)\,\nu_{d-1}(\dint S)\\
&\qquad =\frac{\b_{d-1}\b_{d-2}}{ 4\pi}\int_{\SS_1}\int_{T}\int_{T}g(x,y)\,\sin(\ell(x,y))^{d-2}\,\cH^1(\dint x)\,\cH^1(\dint y)\,
\nu_1(\dint T)\,,\label{eq:Trafo2}
\end{align}
where we used the relation
$$
\int_{\SS_{d-1}}\int_{\SS_1[S]}f(T)\,\nu_1^S(\dint T)\nu_{d-1}(\dint S) = \int_{\SS_1}f(T)\,\nu_1(\dint T)\,,
$$
valid for measurable functions $f:\SS_1\to[0,\infty]$, which follows from \cite[Theorem 7.1.1]{SW}.
A comparison of \eqref{eq:Trafo1a} with \eqref{eq:Trafo2} finally yields the spherical Blaschke-Petkantschin-type identity
\begin{equation*}
\begin{split}
& \int_{\SS_{d-1}}\int_S\int_S g(x,y)\,\cH^{d-1}(\dint x)\,\cH^{d-1}(\dint y)\,\nu_{d-1}(\dint S)\\
&\qquad=\frac{\b_{d-2}}{\b_d}\int_{\SSd}\int_{\SSd}g(x,y)\,\sin(\ell(x,y))^{-1}\,\cH^d(\dint x)\,\cH^d(\dint y)\,,
\end{split}
\end{equation*}
which completes the proof.
\end{proof}

\subsection{The variance of the cell surface measure}
\label{subsec:51VSA}

After the integral-geometric preparations presented in the previous section, we are now in the position to derive a formula for the variance of
the cell surface measure $\Sig_{d-1}(t;h)$, and thus especially for the total $h$-weighted Hausdorff measure $\cH^{d-1}(t;h)$, of a  splitting tessellation $Y_t$ with regular direction distribution $\kappad$ on $\SS_{d-1}$. The following result is the spherical counterpart
of  \cite[Theorem 5.1]{STSecondOrder} (derived there under the assumption of stationarity and isotropy). We replace the bounded observation window  and the Euclidean distance in \cite{STSecondOrder} by the unit sphere and the geodesic distance on $\SSd$, respectively. Comparing \cite[Theorem 5.1]{STSecondOrder} with Theorem \ref{varformula} below for $\kappad=\nu_{d-1}$, we first observe that different constants are  due to different normalizations of
underlying integral-geometric formulas. However, instead of the numerator $\|x-y\|^2$, which appears in the integral representation
in the corresponding Euclidean formula, the result below involves the expression $\ell(x,y)\sin(\ell(x,y))$. The subsequent proof will show that this is due to the application of the spherical integral-geometric transformation formula in Proposition \ref{prop:Trafo}. In addition, we obtain
 a more general version for general direction distributions $\kappad$. To state the general result, we define $\kappad(x,y):=\kappad(\SS_{d-1}\blk \overline{xy} \brk )$ for $x,y\in\SSd$.

\begin{theorem}\label{varformula}
Let  $h:\SSd\to\RR$ be bounded and measurable, and let $t\ge 0$. If $\kappad$ is a regular direction distribution, then
\begin{align*}
\Var\Sig_{d-1}(t;h)&=\frac{1}{\beta_{d-1}^2}\int_{\SS_{d-1}}\int_S\int_S\frac{1-\exp\left(-\kappad(x,y)t\right)}{\kappad(x,y)}\\[2mm]
&\hspace{3cm}\times
h(x)h(y)\,\cH^{d-1}(\dint x)\, \cH^{d-1}(\dint y)\, \kappad(\dint S)<\infty\,.
\end{align*}
In particular,  $\Var\cH^{d-1}(t;h)=\beta_{d-1}^2\Var\Sig_{d-1}(t;h)$.

If $\kappad=\nu_{d-1}$, then
\begin{align*}
\Var\Sig_{d-1}(t;h) &= \frac{\pi \b_{d-2}}{\b_d\b_{d-1}^2}\int_{\SSd}\int_{\SSd}\frac{1-\exp\big(-\frac{1}{\pi}
\ell(x,y)t\big)}{\ell(x,y)\sin(\ell(x,y))}\\[2mm]
&\hspace{3cm}\times h(x)h(y)\,\cH^d(\dint x)\,\cH^d(\dint y)<\infty\,.
\end{align*}
\end{theorem}

\begin{proof}
We define
$$
\bar\Sig_{d-1}(t;h):=\Sig_{d-1}(t;h)-\bE\Sig_{d-1}(t;h)=\sum_{c\in Y_t}\phi_{d-1}(c,h)-t\,\overkappad(h)\,,
$$
where Theorem \ref{thmspaet1} was used. For $T\in \TT^d$ and $t\ge 0$ we put
$$
g_h(T,t):=\left(\sum_{c\in T}\phi_{d-1}(c,h)-t\,\overkappad(h)\right)^2
$$
so that $\bar\Sig_{d-1}(t;h)^2=g_h(Y_t,t)$ and $g_h(Y_0,0)=0$. We have
$$
\frac{\partial g_h}{\partial s}(T,s)=-2\left(\sum_{c\in T}\phi_{d-1}(c,h)-s\,\overkappad(h)\right)
\,\overkappad(h)\,,
$$
hence
$$
\frac{\partial g_h}{\partial s}(\,\cdot\,,s)(Y_s)=-2\bar\Sig_{d-1}(s;h)\overkappad(h)\,.
$$
Now, Corollary \ref{prop:Dynkin2} with $b(t)=\overkappad(h)t$  shows that
\begin{align*}
&\bar\Sig_{d-1}(t;h)^2-0-\int_0^t\Bigg\{\sum_{c\in Y_s}\int_{\SS_{d-1}[c]}\Big[
\Big(\sum_{\bar c\in Y_s\setminus \{c\}\cup\{c\cap S^+, c\cap S^-\}}\phi_{d-1}(\bar c,h)-s\,\overkappad(h)
\Big)^2\\
&\qquad -\Big(\sum_{\bar c\in Y_s}\phi_{d-1}(\bar c,h)-s\,\overkappad(h)\Big)^2\Big]\, \kappad
(\dint S)\Bigg\}-2\,\overkappad(h)\bar\Sig_{d-1}(s;h)\, \dint s\,,\qquad t\geq 0\,,
\end{align*}
is a $\filtration$-martingale. The expression in brackets $[\ldots]$ inside the integral equals
\begin{align*}
&\left(\bar\Sig_{d-1}(s;h)+\phi_{d-1}(c\cap S^+,h)+\phi_{d-1}(c\cap S^-,h)-\phi_{d-1}(c ,h)\right)^2-
\bar\Sig_{d-1}(s;h)^2\\
&=\left(\bar\Sig_{d-1}(s;h)+\phi_{d-1}(c\cap S,h)\right)^2-
\bar\Sig_{d-1}(s;h)^2\\
&=\phi_{d-1}(c\cap S,h)^2+2\phi_{d-1}(c\cap S,h)\bar\Sig_{d-1}(s;h)\,,
\end{align*}
where the valuation property \eqref{eqvaluation} of $\phi_{d-1}$ is used for the first equality.
Using \eqref{refhilf} (as before, the integration over $\SS_{d-1}[c]$ can be replaced by an integration over
$\SS_{d-1}$ without changing the integral), we get
$$
\sum_{c\in Y_s}\int_{\SS_{d-1}}\phi_{d-1}(c\cap S,h)\, \kappad(\dint S)=\overkappad(h)
$$
and conclude that the random process
\begin{equation}\label{eq:VarMartingale}
\bar\Sig_{d-1}(t;h)^2-
\int_0^t\sum_{c\in Y_s}\int_{\SS_{d-1}}\phi_{d-1}(c\cap S,h)^2\,\kappad(\dint S)\,\dint s\,,\qquad t\geq 0\,,
\end{equation}
is a $\filtration$-martingale.
Taking expectations in \eqref{eq:VarMartingale} yields that
\begin{align}\label{eq:VarianceZwischenschritt}
\Var\Sig_{d-1}(t;h) = \int_0^t\bE\sum_{c\in Y_s}\int_{\SS_{d-1}}\phi_{d-1}(c\cap S,h)^2\,\kappad(\dint S)\,\dint s\,.
\end{align}
For $c\in Y_s$, it follows from Lemma \ref{setzero} that for $\kappad$-almost all $S\in\SS_{d-1}$ with $S\cap c\neq \emptyset$, we have
$\cH^{d-1}(\partial c\cap S)=0$. Using this fact and applying \eqref{eq:PhiD-1} together with Fubini's theorem, we get
\begin{align*}
\Var\Sig_{d-1}(t;h) &=\int_0^t\bE\sum_{c\in Y_s}\int_{\SS_{d-1}}\frac{1}{\b_{d-1}^2}\int_S\int_S {\bf 1}(x,y\in \text{int}(c))\\[2mm]
&\hspace{3cm}\times
h(x)h(y)\,\cH^{d-1}(\dint x)\,\cH^{d-1}(\dint y)\,\kappad(\dint S)\,\dint s\\
 &=\frac{1}{\b_{d-1}^2}\int_0^t\int_{\SS_{d-1}}\int_{S}\int_{S}\bE\sum_{c\in Y_s}
{\bf 1}(x,y\in \text{int}(c), S\cap c\neq \emptyset)\\[2mm]
&\hspace{3cm}\times
h(x)h(y)\,\cH^{d-1}(\dint x)\,\cH^{d-1}(\dint y)\,\kappad(\dint S)\,\dint s\\
&=\frac{1}{\b_{d-1}^2}\int_0^t\int_{\SS_{d-1}}\int_{S}\int_{S}
\bP(\overline{xy}\cap Z_s=\emptyset)\\[2mm]
&\hspace{3cm}\times h(x)h(y)\,\cH^{d-1}(\dint x)\,\cH^{d-1}(\dint y)\,\kappad(\dint S)\,\dint s\,,
\end{align*}
where  $Z_s$ is the random closed set induced by the random tessellation $Y_s$ (see \eqref{eq:DefZt}). By Theorem \ref{prop:Capacity}, we have
$\bP(\overline{xy}\cap Z_s=\emptyset)=\exp\big(-\kappad(x,y)s\big)$, and hence
\begin{align*}
\Var\Sig_{d-1}(t;h)&=\frac{1}{\b_{d-1}^2}\int_0^t\int_{\SS_{d-1}}\int_{S}\int_{S}\exp\Big(-\kappad(x,y)s\Big)\\
&\qquad\qquad\times h(x)h(y)\,\cH^{d-1}(\dint x)\,\cH^{d-1}(\dint y)\,\kappad(\dint S)\,\dint s\,.
\end{align*}
Using Fubini's theorem and carrying out the integration with respect to $s$, we obtain the first part of the theorem. \medskip

The second relation easily follows from the first one and from \eqref{basicrelation}.

\medskip

Now we assume that $\kappad=\nu_{d-1}$. Then \eqref{eq:InvMeasureLineSegment} shows that $\kappad(x,y)=\ell(x,y)/\pi$.
Thus we obtain
\begin{align*}
\Var\Sig_{d-1}(t;h) &= \frac{\pi}{\b_{d-1}^2}\int_{\SS_{d-1}}\int_{S}\int_{S}\frac{1-\exp\big(-\frac{1}{\pi}\ell(x,y)t\big)}{\ell(x,y)}\\
&\qquad\qquad\qquad\times h(x)h(y)\,\cH^{d-1}(\dint x)\,\cH^{d-1}(\dint y)\,\nu_{d-1}(\dint S)\,.
\end{align*}
Note that the right-hand side is finite for non-negative and bounded functions $h$. This implies the integrability needed for applying Fubini's theorem in the preceding argument.
This expression is now transformed by means of  \eqref{eq:BlaPetSphere}, which  shows that
\begin{equation*}
\Var\Sig_{d-1}(t;h) = \frac{\pi\b_{d-2}}{\b_d\b_{d-1}^2}\int_{\SSd}\int_{\SSd}\frac{1-\exp\big(-\frac{1}{\pi}
\ell(x,y)t\big)}{\ell(x,y)\sin(\ell(x,y))}h(x)h(y)\,\cH^d(\dint x)\,\cH^d(\dint y)
\end{equation*}
and completes the proof of the theorem.
\end{proof}

As a simple consequence, we get the following corollary, which will be used subsequently.

\begin{corollary}\label{covcor}
Let $t\ge 0$, and let $h_1,h_2:\SSd\to\RR$ be bounded and measurable. Suppose that $\kappad$
is a regular direction distribution. Then
\begin{align*}
&\Cov(\Sig_{d-1}(t;h_1),\Sig_{d-1}(t;h_2))\\
&\qquad = \frac{1}{\beta_{d-1}^2}\int_{\SS_{d-1}}\int_S\int_S\frac{1-\exp\left(-\kappad(x,y)t\right)}{\kappad(x,y)}
h_1(x)h_2(y)\,\cH^{d-1}(\dint x)\, \cH^{d-1}(\dint y)\, \kappad(\dint S)<\infty\,.
\end{align*}
In particular, $\Cov(\cH^{d-1}(t;h_1),\cH^{d-1}(t;h_2))=\beta_{d-1}^2\Cov(\Sig_{d-1}(t;h_1),\Sig_{d-1}(t;h_2))$.

\noindent If $\kappad=\nu_{d-1}$, then
\begin{align*}
&\Cov(\Sig_{d-1}(t;h_1),\Sig_{d-1}(t;h_2))\\
&\qquad = \pi\frac{\b_{d-2}}{\b_d\beta_{d-1}^2}\int_{\SSd}
\int_{\SSd}\frac{1-\exp\big(-\frac{1}{\pi}\ell(x,y)t\big)
}{\ell(x,y)\sin(\ell(x,y))}h_1(x)h_2(y)\,\cH^d(\dint x)\,\cH^d(\dint y)<\infty\,.
\end{align*}
\end{corollary}

\begin{proof}
We apply Theorem \ref{varformula} with  $h=h_1+h_2$, $h=h_1$ and $h=h_2$, observe that $\Sig_{d-1}(t;h)=\Sig_{d-1}(t;h_1)+\Sig_{d-1}(t;h_2)$ so that
\begin{align*}
\Var\Sig_{d-1}(t;h_1+h_2)&=\Var\Sig_{d-1}(t;h_1 )+\Var\Sig_{d-1}(t; h_2)\\
&\hspace{4cm}+2\Cov(\Sig_{d-1}(t;h_1),\Sig_{d-1}(t;h_2))
\end{align*}
and then expand also $(h_1+h_2)^2$ on the left-hand side.
\end{proof}

\begin{remark}\label{Remna}{\rm  For $\kappad=\nu_{d-1}$, Theorem \ref{varformula} and Corollary \ref{covcor} can be reformulated by introducing a spherical analogue of
the isotropized set-covariance function of the observation window which was used in the Euclidean setting in \cite[Theorem 5.1]{STSecondOrder}. Let $g:\SSd\times \SSd\to [0,\infty]$
be measurable. For $x,e\in\SSd$ we then define
$$
\gamma(g;x,e):=\int_{{\rm SO}(d+1)}g(\varrho x,\varrho e)\, \nu(\dint \varrho)\,.
$$
Fix $e\in\SSd$. Then, for a measurable function $f:[0,\pi]\to[0,\infty]$, applying  the invariance properties of $\ell$ and of the
spherical Hausdorff measure and Fubini's theorem,  we get
\begin{align*}
&\int_{\SSd}\int_{\SSd} f(\ell(x,y))g(x,y)\,\cH^{d}(\dint y)\,\cH^{d}(\dint x)\\
&\qquad =\beta_d\int_{{\rm SO}(d+1)}\int_{\SSd}f(\ell(x,\varrho e))g(x,\varrho e)\, \cH^{d}(\dint x)\, \nu(\dint \varrho)\\
&\qquad =\beta_d\int_{{\rm SO}(d+1)}\int_{\SSd}f(\ell(\varrho x,\varrho e))g(\varrho x,\varrho e)\, \cH^{d}(\dint x)\, \nu(\dint \varrho)\\
&\qquad =\beta_d \int_{\SSd}f(\ell(x,e)) \int_{{\rm SO}(d+1)}g(\varrho x,\varrho e)\, \nu(\dint \varrho)\, \cH^{d}(\dint x)\\
&\qquad =\beta_d \int_{\SSd}f(\ell(x,e)) \gamma(g;x,e)\, \cH^{d}(\dint x)\\
&\qquad =\beta_d \int_{\SS^d\cap e^\perp}\int_0^\pi f(\ell(\varphi))\gamma(g;\cos(\varphi) \, e+\sin(\varphi)\, u,e)(\sin\varphi)^{d-1}\, \dint\varphi\, \cH^{d-1}(\dint u)\,,
\end{align*}
where we used that the  transformation $(\SSd\cap e^\perp)\times (0,\pi)\to\SSd$, $(u,\varphi)\mapsto \cos(\varphi) \, e+\sin(\varphi) \,u$ has the Jacobian  $(\sin\varphi)^{d-1}$. In the special case $g\equiv 1$, we have $\gamma(g;x,e)=1$ for $x,e\in\SSd$, hence the formulas simplify further.
}
\end{remark}

\begin{example}\label{rem:d=2Variance}{\rm
To illustrate Theorem \ref{varformula}, we discuss the special case $\kappad=\nu_{d-1}$ and $g\equiv 1$. Fix $e\in\SSd$ arbitrarily.  Using Remark \ref{Remna},
 we obtain
\begin{align}
\Var\cH^{d-1}(Z_t)&= \pi \b_{d-2}  \int_{\SSd}\frac{1-\exp\big(-\frac{1}{\pi}\ell(x,e)t\big)
}{\ell(x,e)\sin(\ell(x,e))} \,\cH^d(\dint x)\nonumber\\
&=\pi \b_{d-2} \beta_{d-1}\int_0^\pi (\sin\varphi)^{d-2}\,\frac{1-\exp\left(-\frac{\varphi  t}{\pi}\right)}{\varphi}\, \dint \varphi\nonumber\\
&=\frac{(2\pi)^d}{(d-2)!}\int_0^1 \sin(\pi z)^{d-2}\,\frac{1-\exp\left(-zt\right)}{ z}\, \dint z\label{eq5.6neu}\\
&=\frac{(2\pi)^d}{(d-2)!}\sum_{i=0}^\infty \frac{(-t)^{i+1}}{(i+1)!}\int_0^1z^i\sin(\pi z)^{d-2}\,\dint z\,.\nonumber
\end{align}
For $d=2$ and $t>0$, this yields
$$
\Var\cH^{1}(Z_t)=4\pi^2\int_0^1 \frac{1-e^{-tz}}{ z}\, \dint z=4\pi^2 \sum_{j=1}^\infty\frac{(-1)^{j-1}}{j}\frac{t^j}{j!}=4\pi^2\big(\gamma+\ln t+E_1(t)\big)\,,
$$
where $\gamma\approx 0.5772$ is the Euler-Mascheroni constant and $E_1(t):=\int_t^\infty s^{-1}e^{-s}\, \dint s$ is the exponential integral (see \cite[Chapter 6]{OlverEtAl}).
For instance, we thus obtain $\Var\cH^{1}(Z_1)\approx 4\pi^2 \cdot 0.7566\approx 31.4485$. For two functions $f,g:[0,\infty)\to\RR$ let us write $f(t)\sim g(t)$ whenever $f(t)/g(t)\to 1$ as $t\to\infty$. Using that $E_1(t)\sim \frac{e^{-t}}{ t}$ by
\cite[(6.12.1)]{OlverEtAl}, we thus conclude  for $d=2$ that
$$
\Var\cH^{1}(Z_t )\sim 4\pi^2\ln t\to\infty\,,\qquad\text{as }t\to\infty\,.
$$
In contrast, for $d\ge 3$ we have
$$
\Var\cH^{d-1}(Z_t)\le \frac{(2\pi)^d}{(d-2)!}\int_0^1 \pi\sin(\pi z)^{d-3}\, \dint z<\infty
$$
for any $t\ge 0$. On the other hand, for fixed $t\ge 0$,  $\Var\cH^{d-1}(Z_t)\to 0$, as $d\to\infty$, at an exponential rate. Using \eqref{eq5.6neu} and numerical integration (Mathematica), the numerical values of $\Var\cH^{d-1}(Z_t)$ can be determined for $d\ge 2$ and $t>0$, see Figure \ref{fig:VarianceD}.

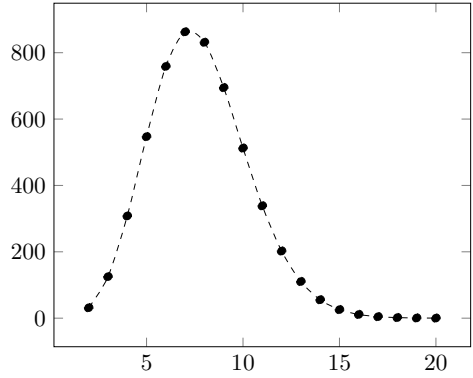
\begin{figure}[t]\label{fig:VarianceD}
\begin{center}
\begin{tikzpicture}[scale=0.8]
\begin{axis}
    \addplot[smooth,mark=*,black,dashed] plot coordinates {
        (2,31.449)
        (3,125.133)
        (4,308.091)
        (5,547.089)
        (6,758.774)
        (7,862.902)
        (8,831.403)
        (9,694.804)
        (10,512.613)
        (11,338.515)
        (12,202.312)
        (13,110.419)
        (14,55.453)
        (15,25.789)
        (16,11.168)
        (17,4.524)
        (18,1.772)
        (19,0.618)
        (20,0.209)
    };
\end{axis}
\end{tikzpicture}
\end{center}
\caption{The variance $\Var\cH^{d-1}(Z_1)$ of the total surface area of all cells of the splitting tessellation $Y_1$ as a function of $d$ for $d\in\{2,3,\ldots,20\}$.}
\end{figure}}
\end{example}

As above, let us denote by $\cH^{d-1}(t;h)$, for a bounded and measurable function $h:\SSd\to\RR$, the $h$-weighted total Hausdorff measure of all great hyperspherical pieces of a splitting tessellation constructed up to time $t$ with underlying direction distribution $\nu_{d-1}$. Then the random process $\widetilde{\cH^{d-1}}(t;h):=\cH^{d-1}(t;h)-\bE\cH^{d-1}(t;h)$, $t\geq 0$, is a martingale with respect to the natural filtration $\filtration$ induced by the splitting tessellation process. Since, by the previous theorem,
$$
\sup_{t\geq 0}\bE\widetilde{\cH^{d-1}}(t;h)^2 = \sup_{t\geq 0}\Var\cH^{d-1}(t;h)<\infty\,,
$$
provided that $d\ge 3$,
we can apply the martingale convergence theorem \cite[Theorem II.3.1]{RevusYor} to conclude that there exists a centred and square-integrable random variable $h_Z$ for which
$$
\lim_{t\to\infty}\widetilde{\cH^{d-1}}(t;h) = h_Z
$$
almost surely and in the $L^2$ sense. Moreover, and again from the previous theorem, it follows that
$$
\Var h_Z = \lim_{t\to\infty}\Var\cH^{d-1}(t;h) = \pi\frac{\b_{d-2}}{\b_d}\int_{\SSd}\int_{\SSd}\frac{h(x)h(y)}{\ell(x,y)\sin(\ell(x,y))}\,\cH^d(\dint x)\,\cH^d(\dint y)<\infty\,,
$$
provided that $d\ge 3$. Note that if $h$ is positive on an open set,
then this asymptotic variance is infinite for $d=2$.

\subsection{Further variances and covariances}\label{subsec:52FurtherVariancesCovariances}

Our next goal is to compute the variance of $\Sig_j(t;h)$ and, more generally, the covariances of the functionals $\Sig_i(t;h_1)$ and $\Sig_j(t;h_2)$ under the assumption that the  direction distribution is $\nu_{d-1}$. In order to present the result, we need to introduce some notation. First, let us define for $f:[0,t]\to\RR$ and $n\in\NN$ the iterated integral
$$
I^n(f,t) := \int_0^t\int_0^{s_1}\ldots\int_0^{s_{n-1}}f(s_n)\,\dint s_n\ldots \dint s_2\dint s_1 = \frac{1}{(n-1)!}\int_0^t(t-s)^{n-1}\,f(s)\,\dint s\,,
$$
whenever this is well defined. Moreover, for $i,j\in\{0,\ldots,d-1\}$ we put
$$
A_{i,j}(T; h_1,h_2) := \sum_{c\in T}\int_{\SS_{d-1}[c]}\phi_i(c\cap S,h_1)\,\phi_j(c\cap S,h_2)\,\nu_{d-1}(\dint S)\,,\qquad T\in\TT^d\,,
$$
where here and in what follows $h_1,h_2:\SSd\to\RR$ are bounded, measurable functions. We are now in the position to present closed formulas for the variances and covariances in terms of iterated integrals of $\bE A_{i,j}(s;h_1,h_2)$, where $A_{i,j}(s;h_1,h_2):=A_{i,j}(Y_s;h_1,h_2)$. The result is the spherical analogue of \cite[Theorem 1]{STSTITHigher}.

\begin{theorem}\label{thm:Covariances}
Let $k,\ell\in\{0,\ldots,d-1\}$ and $t\geq 0$. Let $h,h_1,h_2:\SSd\to\RR$ be bounded and measurable functions. Suppose that $\kappad=\nu_{d-1}$. Then
\begin{align*}
&\Cov(\Sig_{d-1-k}(t;h_1),\Sig_{d-1-\ell}(t;h_2)) \\
&\qquad = \sum_{m=0}^k\sum_{n=0}^\ell\binom{k+\ell-m-n}{ k-m}\,I^{k+\ell-m-n+1}(\bE A_{d-1-m,d-1-n}(\,\cdot\,;h_1,h_2),t)
\end{align*}
and, in particular,
$$
\Var \Sig_{d-1-k}(t;h) = \sum_{m,n=0}^k \binom{2k-m-n}{ k-m}\,I^{2k-m-n+1}(\bE A_{d-1-m,d-1-n}(\,\cdot\,;h,h),t)\,.
$$
\end{theorem}

\begin{proof}
For a tessellation $T\in\TT^d$ and $t\ge 0$, we consider the function
$$
g(T,t):=\left[\sum_{c\in T}\phi_i(c,h_1)-\frac{t^{d-i}}{(d-i)!}\frac{\cH^{d}(h_1)}{\beta_d}\right]
\left[\sum_{c\in T}\phi_j(c,h_2)-\frac{t^{d-j}}{(d-j)!}\frac{\cH^{d}(h_2)}{\beta_d}\right] \,,
$$
where the dependence of $g$ on $i,j,h_1,h_2$ is not indicated. Recalling that
$$
\bar\Sigma_i(t;h)=\Sigma_i(t;h)-\bE\Sigma_i(t;h)=\sum_{c\in Y_t}\phi_i(c,h)-\frac{t^{d-i}}{(d-i)!}
\frac{\cH^{d}(h)}{\beta_d}\,,
$$
we have
$$
g(Y_t,t)=\bar\Sigma_i(t;h_1)\bar\Sigma_j(t;h_2)\qquad\text{and}\qquad g(Y_0,0)=0\,.
$$
Furthermore, we get
$$
\frac{\partial g}{\partial s}(\,\cdot\,,s)(Y_s)=-\frac{s^{d-i-1}}{(d-i-1)!}
\frac{\cH^{d}(h_1)}{\beta_d}\bar\Sigma_j(s;h_2)-\frac{s^{d-j-1}}{(d-j-1)!}
\frac{\cH^{d}(h_2)}{\beta_d}\bar\Sigma_i(s;h_1)\,.
$$
Now, applying  Proposition \ref{prop:Dynkin3} with  $b_1(t)=\frac{\cH^d(h_1)}{(d-i)!\beta_d}t^{d-i}$ and $b_2(t)=\frac{\cH^d(h_2)}{(d-j)!\beta_d}t^{d-j}$, the valuation property of the spherical curvature measures $\phi_i,\phi_j$, and the fact that $\bar\Sigma_i(0;h_1)=0=\bar\Sigma_j(0;h_2)$, we obtain that
\begin{align*}
&\bar\Sigma_i(t;h_1)\bar\Sigma_j(t;h_2) -\int_0^t\sum_{c\in Y_s}\int_{\SS_{d-1}[c]}
\left\{\left(\bar\Sigma_i(s;h_1)+\phi_i(c\cap s,h_1)\right)\left(\bar\Sigma_j(s;h_2)+\phi_j(c\cap s,h_2)\right)\right.\\
&\qquad\qquad\qquad\qquad\qquad\qquad\left.
-\bar\Sigma_i(s;h_1)\bar\Sigma_j(s;h_2)\right\}\, \nu_{d-1}(\dint S)\, \dint s \\
&\qquad\qquad+
 \int_0^t \frac{s^{d-i-1}}{(d-i-1)!}
\frac{\cH^{d}(h_1)}{\beta_d}\bar\Sigma_j(s;h_2)+\frac{s^{d-j-1}}{(d-j-1)!}
\frac{\cH^{d}(h_2)}{\beta_d}\bar\Sigma_i(s;h_1)\, \dint s\\
&=\bar\Sigma_i(t;h_1)\bar\Sigma_j(t;h_2)-\int_0^t A_{i,j}(s;h_1,h_2)\, \dint s\\
&\qquad -\int_0^t\left\{\bar\Sigma_i(s;h_1)\sum_{c\in Y_s}\int_{\SS_{d-1}[c]}\phi_j(c\cap S,h_2)\, \nu_{d-1}(\dint S)
\right.\\
&\qquad\qquad\qquad\qquad\left.+\bar\Sigma_j(s;h_2)\sum_{c\in Y_s}\int_{\SS_{d-1}[c]}\phi_i(c\cap S,h_1)\, \nu_{d-1}(\dint S)\right\}\\
&\qquad\qquad +\left\{
\bar\Sigma_i(s;h_1)\frac{s^{d-j-1}}{(d-j-1)!}
\frac{\cH^{d}(h_2)}{\beta_d}
+
\bar\Sigma_j(s;h_2)\frac{s^{d-i-1}}{(d-i-1)!}
\frac{\cH^{d}(h_1)}{\beta_d}
\right\}
\, \dint s
\end{align*}
is a $\filtration$-martingale.
Using again the Crofton formula \eqref{eq:CroftonOnSphere2} for the spherical curvature measures, we conclude that
\begin{equation*}
\begin{split}
\bar\Sig_i(t;h_1)\bar\Sig_j(t;h_2) &- \int_0^t A_{i,j}(s;h_1,h_2)\,\dint s \nonumber\\
&-\int_0^t[\bar \Sig_i(s;h_1)\bar \Sig_{j+1}(s;h_2)+\bar \Sig_{i+1}(s;h_1)\bar \Sig_j(s;h_2)]\,\dint s\nonumber
\end{split}
\end{equation*}
is a $\filtration$-martingale. Taking expectations yields the recursion formula
\begin{align*}
&\Cov(\Sig_i(t;h_1),\Sig_j(t;h_2)) = \int_0^t\bE A_{i,j}(s;h_1,h_2)\,\dint s \\
&\qquad\qquad\qquad +\int_0^t[\Cov(\Sig_i(s;h_1),\Sig_{j+1}(s;h_2))+\Cov(\Sig_{i+1}(s;h_1),\Sig_j(s;h_2))]\,\dint s\,,
\end{align*}
which expresses the covariance of $\Sig_i(t;h_1)$ and $\Sig_j(t;h_2)$ by means of $\bE A_{i,j}(s;h_1,h_2)$ as well as covariances with one index increased by one. Continuing this recursion, one eventually arrives at covariances formally involving $\Sig_d(t;h_i)$, $i\in\{1,2\}$, which is identically zero. This in turn shows that the recursion terminates after finitely many steps. Arguing now exactly as at the beginning of Section 3.1 of \cite{STSTITHigher}, one arrives at the desired formula after a change of variables.
\end{proof}

We notice that Theorem \ref{thm:Covariances} is related to the variance formula from the previous section. Indeed, putting $k=0$ in the formula for the variance in Theorem \ref{thm:Covariances}, we get
\begin{align*}
\Var \Sig_{d-1}(t;h)& = I^1(\bE A_{d-1,d-1}(\,\cdot\,;h,h),t)\\
& = \int_0^t\bE\sum_{c\in Y_s}\int_{\SS_{d-1}[c]}\phi_{d-1}(c\cap S,h)^2\,\nu_{d-1}(\dint S)\,\dint s\,,
\end{align*}
which is just Equation \eqref{eq:VarianceZwischenschritt} for $\kappad=\nu_{d-1}$. A similar remark applies to Corollary \ref{covcor}.

\medbreak

Finally in this section we present two examples illustrating how one can use Theorem \ref{thm:Covariances} for explicit variance and covariance computations. Recall that in Example \ref{rem:d=2Variance} in the previous section we already computed the variance $\Var\cH^1(Z_t)$ of the Hausdorff measure of cell boundaries induced by an isotropic splitting tessellation on $\SS^2$ (which can be combined with  \eqref{basicrelation}). Together with the  following two examples, we thus obtain  a complete and fully explicit second-order description of the spherical intrinsic volumes of isotropic splitting tessellations on the $2$-dimensional unit sphere. This can be regarded a spherical analogue of one of the main results from \cite{STSTITPlane} for STIT tessellations in the plane.

\begin{example}\label{ex5.9}{\rm
Our goal is to compute $\Var\Sig_0(t)$ with $\Sig_0(t)=\sum_{c\in Y_t}V_0(c)$ for an isotropic splitting tessellation $Y_t$ with time parameter $t$ on the $2$-dimensional unit sphere $\SS^2$. Using Theorem \ref{thm:Covariances} we first obtain the representation
\begin{align*}
\Var\Sig_0(t) = 2I^3(\bE A_{1,1}(\,\cdot\,),t)+2I^2(\bE A_{1,0}(\,\cdot\,),t)+I^1(\bE A_{0,0}(\,\cdot\,),t)\,,
\end{align*}
where we use the abbreviation $A_{i,j}(\,\cdot\,):=A_{i,j}(\,\cdot\,;1,1)$. For a tessellation $T\in\TT^2$ with $T\neq\{\SS^2\}$ we have
\begin{align*}
A_{0,0}(T) &= \sum_{c\in T}\int_{\SS_{1}([c])}V_0(c\cap S)^2\,\nu_{1}(\dint S) = \frac{1}{ 4}\sum_{c\in T}\nu_{1}([c])\\
& = \frac{1}{ 4}\sum_{c\in T}2V_1(c) = \frac{1}{ 2}\sum_{c\in T}V_1(c)\,,
\end{align*}
where we used that $V_0(s)=1/2$ for any spherical line segment $s$ to obtain the second equality, and Equation \eqref{nonloc} for the third equality.
This equality also holds if $T=\{\SS^2\}$, since $V_0(S)=0$ for $S\in \SS_1$ and $V_1(\SS^2)=0$ imply that both sides are zero.

Using once more that $V_0(s)=1/2$ for any spherical line segment $s$ together with the spherical Crofton formula and the fact that $V_2(c)=\cH^2(c)/\beta_2$, we also obtain, for $T\in\TT^2$ with $T\neq\{\SS^2\}$,
\begin{align}\label{eq:A10}
A_{1,0}(T) &= \sum_{c\in T}\int_{\SS_{1}([c])}V_0(c\cap S)V_1(c\cap S)\,\nu_{1}(\dint S)
= \frac{1}{2}\sum_{c\in T}\int_{\SS_{1}}V_1(c\cap S)\,\nu_{1}(\dint S)\nonumber\\
&=\frac{1}{2}\sum_{c\in T}V_2(c) = \frac{1}{ 2\beta_2}\sum_{c\in T}\cH^2(c) = \frac{1}{2}\,,
\end{align}
where we replaced $\SS_{1}([c])$ by $\SS_{1}$, which does not change the value of the integral. In addition, we have $A_{1,0}(T)=0$ for $T=\{\SS^2\}$,
since $V_0(S)=0$ for $S\in \SS_1$.

Finally, we note that
\begin{align*}
A_{1,1}(T) = \sum_{c\in T}\int_{\SS_{1}([c])}V_1(c\cap S)^2\,\nu_{1}(\dint S)\,.
\end{align*}
Now we compute the iterated integrals $I^1(\bE A_{0,0}(\,\cdot\,),t)$, $I^2(\bE A_{1,0}(\,\cdot\,),t)$ and $I^3(\bE A_{1,1}(\,\cdot\,),t)$. Using Theorem \ref{thm:Expectation},  we obtain
\begin{align*}
I^1(\bE A_{0,0}(\,\cdot\,),t) = \frac{1}{2}\int_0^t \bE \sum_{c\in Y_s}V_1(c)\,\dint s = \frac{1}{2}\int_0^t s\,\dint s = \frac{1}{ 4}t^2\,.
\end{align*}
Moreover,
\begin{align*}
I^2(\bE A_{1,0}(\,\cdot\,),t) &= \int_0^t (t-s) \left[\frac{1}{2}\bP(Y_s\neq\{\SS^2\})+0\cdot\bP(Y_s=\{\SS^2\})\right]\,\dint s \\
&= \frac{1}{2}\int_0^t(t-s)\left(1-e^{-s}\right)\,\dint s = \frac{1}{2}\left(1-t+\frac{1}{2}t^2-e^{-t}\right)\,.
\end{align*}
To compute $I^3(\bE A_{1,1}(\,\cdot\,),t)$ we use the computations already carried out in the proof of Theorem \ref{varformula}. Using then also
Proposition \ref{prop:Trafo} and the transformation explained in Remark \ref{Remna}, we obtain
\begin{align*}
I^3(\bE A_{1,1}(\,\cdot\,),t) &= \frac{1}{2}\int_0^t (t-s)^2\,\bE \sum_{c\in Y_s}\int_{\SS_{1}}V_1(c\cap S)^2\,\nu_{1}(\dint S)\,\dint s\\
&=\frac{1}{2\beta_1^2}\int_0^t (t-s)^2\int_{\SS_{1}}\int_S\int_S e^{-\frac{\ell(x,y)}{\pi}s}\,\cH^1(\dint x)\,\cH^1(\dint y)\,\nu_1(\dint S)\,\dint s\\
&=\frac{1}{2\pi}\int_0^t(t-s)^2\int_0^\pi e^{-\frac{\varphi}{\pi}s}\, \dint\varphi\, \dint s\\
&=\frac{1}{2}t^2\int_0^1\frac{(1-z)^2\left(1-e^{-tz}\right)}{z}\, \dint z\,.
\end{align*}
Collecting all contributions, we arrive at
\begin{align*}
\Var\Sig_0(t) = t^2\int_0^1\frac{(1-z)^2\left(1-e^{-tz}\right)}{z}\, \dint z +1-t+\frac{3}{4}t^2-e^{-t}\,.
\end{align*}
In terms of the Euler-Mascheroni constant $\gamma$ and the exponential integral $E_1$ (recall their definitions from Example \ref{rem:d=2Variance}) this can also be expressed as
\begin{align*}
\Var\Sig_0(t) = t^2\ln t+t^2\Big(\gamma-\frac{3}{ 4}+E_1(t)\Big)+t\left(1-e^{-t}\right)\sim t^2\ln t\,.
\end{align*}
}
\end{example}

\begin{example}{\rm
We are now going to determine the covariance $\Cov(\Sig_0(t),\Sig_1(t))$, where $\Sig_j(t)=\sum_{c\in Y_t}V_j(c)$, $j\in\{0,1\}$, for an isotropic splitting tessellation $Y_t$ on $\SS^2$. Using the notation introduced in the previous example, Theorem \ref{thm:Covariances} yields
$$
\Cov(\Sig_0(t),\Sig_1(t)) = I^2(\bE A_{1,1}(\,\cdot\,),t)+I^1(\bE A_{0,1}(\,\cdot\,),t)\,.
$$
The iterated integrals can be computed as follows. From \eqref{eq:A10} we first have that
$$
I^1(\bE A_{0,1}(\,\cdot\,),t) = \frac{1}{ 2}\int_0^t 1- e^{-s}\,\dint s=\frac{1}{2}\left(t-1+e^{-t}\right)\,.
$$
Next, we compute $I^2(\bE A_{1,1}(\,\cdot\,),t)$ in the same way as in the previous example (using the transformation formulas from Proposition \ref{prop:Trafo} and Remark \ref{Remna}):
\begin{align*}
I^2(\bE A_{1,1}(\,\cdot\,),t) &= \frac{1}{\beta_1^2}\int_0^t(t-s)\int_{\SS_1}\int_S\int_Se^{-\frac{\ell(x,y)}{\pi}s}\,\cH^1(\dint x)\,\cH^1(\dint y)\,\nu_1(\dint S)\,\dint s\\
&=\int_0^t(t-s)\int_0^1e^{-xs}\, \dint x\, \dint s\\
&=t\int_0^1\frac{(1-z)\left(1-e^{-tz}\right)}{z}\, \dint z\,.
\end{align*}
As a result, we obtain the following explicit covariance formula in terms of the Euler-Mascheroni constant $\gamma$ and the exponential integral $E_1$:
\begin{align*}
\Cov(\Sig_0(t),\Sig_1(t)) &= t\int_0^1\frac{(1-z)\left(1-e^{-tz}\right)}{z}\, \dint z+\frac{1}{2}\left(t-1+e^{-t}\right)\\
&=t\ln t+t\left(\gamma -\frac{1}{2} +E_1(t)\right)+\frac{1}{ 2}\left(1-e^{-t}\right)\sim t\ln t\,.
\end{align*}
}
\end{example}

\subsection{Spherical pair-correlation function}\label{sec:PCF}

The purpose of this section is to compute the $K$-function and the spherical pair-correlation function of the $(d-1)$-dimensional random Hausdorff measure induced by a splitting tessellation on $\SSd$. Before we present our result, we shall first introduce the functions we are interested in. To this end, we use the concept of Palm distributions in homogeneous spaces (see \cite{RotherZaehle}) and introduce the $K$-function and the spherical pair-correlation function as the spherical analogues of corresponding functions of a stationary random measure in $\RR^d$. For an introduction of these concepts in a Euclidean space, background information and applications to spatial statistics, we refer to Ripley's original contribution  \cite{Ripley} and to the discussion included there, to \cite[Chapter 7.2.2]{CSKM} (for general stationary random measures) and to \cite[Chapters 4.5-6]{CSKM} and \cite[Chapters 7.3-10]{BT} (for stationary point processes).

\medskip

Let $\MM(\SSd)$ be the space of finite Borel measures on $\SSd$ and denote by $\cM(\SSd)$ the $\sigma$-field on $\MM(\SSd)$ generated by the evaluation mappings $m\mapsto m(A)$, $m\in\MM(\SSd)$, for each $A\in\cB(\SSd)$. Let $\bM$ be a random measure on $\SSd$, that is, a measurable mapping from an underlying probability space $(\O,\sigmaalgebra,\bP)$ into $\MM(\SSd)$, and denote by $\bP_\bM:=\bP\circ\bM^{-1}$ the distribution of $\bM$. In what follows, we assume that $\bM$ has finite and non-zero intensity measure $\bE[\bM(\,\cdot\,)]$. We call $\mu:=\bE[\bM(\SSd)]\in(0,\infty)$ the intensity of
$\bM$. If $\bM$ is isotropic, then  $\bE[\bM(\,\cdot\,)]=\mu\, \b_d^{-1}\cH^d(\,\cdot\,)$ on $\SSd$.
 As usual, for a set $B\subset \MM(\SSd)$ and $\varrho\in \SO(d+1)$,
we define $\varrho B:=\{\varrho(m):m\in B\}$ and $\varrho(m):=\varrho m:=m\circ\varrho^{-1}$.

\medskip

The unit sphere $\SSd$ is a homogeneous space on which $\SO(d+1)$ acts transitively, and $\SO(d)$ can be interpreted as the stabilizer of the north pole $e:=(0,\ldots,0,1)\in\RRd1$ (or any other fixed point of $\SSd$). Following \cite{RotherZaehle}, for $x\in\SSd$ we put $\Theta_x:=\{\varrho\in\SO(d+1):\varrho e=x\}$, denote by $\nu_e$ the unique Haar probability measure on $\Theta_e$ and define $\nu_x:=\nu_e\circ\varrho^{-1}_x$ for an arbitrary $\varrho_x\in\Theta_x$ (the definition is in fact  independent of the particular choice of $\varrho_x$). We are now prepared to define the Palm distribution of the random measure $\bM$ with respect to $e$ by
\begin{align*}
\bP_\bM^e(B):&=\frac{1}{\mu}\int_{\MM(\SSd)}\int_{\SSd}\int_{\Theta_x}{\bf 1}(\varrho^{-1} m\in B)\,\nu_x(\dint\varrho)
\,m(\dint x)\,\bP_\bM(\dint m)\\
&=\frac{1}{\mu}\bE\int_{\SSd}\int_{\Theta_x}{\bf 1}(\varrho^{-1} \bM\in B)\,\nu_x(\dint\varrho)
\,\bM(\dint x)\,,\qquad B\in\cM(\SSd)\,.
\end{align*}
Intuitively speaking, the Palm distribution of $\bM$ is the conditional distribution of $\bM$ under the condition that $e$ is in the support of $\bM$.
We emphasize that this definition does not require that $\bM$ is isotropic (here we take advantage of the compactness of $\SSd$). Moreover, a routine argument shows that
\begin{equation}\label{eqinvariant}
\bP_\bM^e(B)=\bP_{ \bM}^{\sigma e}(\sigma B)\,,\qquad B\in\cM(\SSd)\,,\sigma\in \SO(d+1)\,.
\end{equation}

Using the concept of Palm distributions, we introduce the spherical $K$-function  as well as the spherical pair-correlation function of $\bM$.

\begin{definition}{\rm
Let $\bM$ be a random measure on $\SSd$ with intensity $\mu\in(0,\infty)$ and Palm distribution $\bP_\bM^e$. The spherical $K$-function of $\bM$ is defined as
$$
K_{\bM}(r) := \frac{1}{\mu}\int_{\MM(\SSd)}m(B(e,r))\,\bP_\bM^e(\dint m)\,,\qquad r\in(0,\pi)\,,
$$
where $B(e,r)=\{x\in\SSd:\ell(x,e)\leq r\}$ is the geodesic ball centred at $e$ and with radius $r$. If $K_{\bM} $ is differentiable, then we call
$$
g_{\bM}(r) := \frac{\b_d}{\b_{d-1}(\sin r)^{d-1}} \, K_{\bM}'(r)\,,\qquad r\in (0,\pi)\,,
$$
the spherical pair-correlation function of $\bM$.}
\end{definition}

The value $K_{\bM}(r)$ of the $K$-function of $\bM$ can be interpreted as the mean measure of $B(e,r)$ under the condition that $e$
is in the support of $\bM$. It follows from \eqref{eqinvariant} that the $K$-function and the pair-correlation function thus defined are
independent of the choice of the reference point $e\in\SSd$.

In the following it is useful to rewrite the $K$-function of $\bM$. This again shows that the preceding definitions are independent of the choice of the reference point $e$. We get
\begin{align}
\nonumber K_{\bM}(r)&=\frac{1}{\mu^2}\bE \int_{\SSd}\int_{\Theta_x} (\varrho^{-1} \bM)(B(e,r))\,  \nu_x(\dint\varrho)\,\bM(\dint x)\\
\nonumber &=\frac{1}{\mu^2} \bE \int_{\SSd} \int_{\Theta_x} \bM(B(x,r))\,  \nu_x(\dint\varrho)\,\bM(\dint x)\\
\nonumber &=\frac{1}{\mu^2} \bE \int_{\SSd}  \bM(B(x,r))\,\bM(\dint x)\\
&=\frac{1}{\mu^2} \bE \int_{(\SSd)^2}{\bf 1}(\ell(x,y)\le r)\, \bM^2(\dint(x,y))\,.\label{eq:KMAlternativeRepresentation}
\end{align}
This representation describes $K_{\bM}(r)$ as the normalized second moment measure of $\bM$ of the set of all pairs $(x,y)\in(\SSd)^2$
having distance at most $r$. Thus $K_{\bM}(r)$ quantifies the chance that two points chosen independently from the support of $\bM$ have distance at most $r$ from each other.

\begin{remark}\rm
In the Euclidean case, the $K$-function is defined in a similar way. However, the pair-correlation function is defined by
$$
g(r) = \frac{1}{ d\kappa_dr^{d-1}}  K'(r)\,,
$$
if $K$ is differentiable.
Since $d\kappa_dr^{d-1}$ is the surface area of a $(d-1)$-sphere of radius $r$ in Euclidean space $\RR^d$, in spherical space  we divide  by $\b_{d-1}(\sin r)^{d-1}$, the $(d-1)$-dimensional Hausdorff measure of the boundary of a geodesic ball at distance $r$ to $e$. The additional factor $\beta_d$ arises, since we are working with the normalized Hausdorff measure.
\end{remark}

As anticipated above, we next compute the spherical $K$-function $K_{d,t}$ and the pair-correlation
function $g_{d,t}(r)$ of the $(d-1)$-dimensional random Hausdorff measure induced by the splitting tessellation $Y_t$ with direction distribution $\kappad=\nu_{d-1}$ on $\SSd$. In other words, we consider the isotropic random measure $\bM=\cH^{d-1}\llcorner Z_t$ with $Z_t$ defined in  \eqref{eq:DefZt}. For general direction distributions, the $K$-function does not reduce to a one-dimensional integral and therefore will only be provided at the end of this section.

\begin{theorem}\label{thm:PCF}
Suppose that $\kappad=\nu_{d-1}$.
If $t>0$ and $r\in(0,\pi)$, then
$$
K_{d,t}(r) =\frac{\b_{d-1}}{\b_d}
\int_0^r \Bigg(1+\pi\frac{\b_{d-2}\b_d}{\b_{d-1}^2}\frac{1-\exp\big(-\frac{t}{\pi}\varphi\big)}{ t^2\varphi\sin\varphi}\Bigg)(\sin\varphi)^{d-1}\,\dint\varphi
$$
and
$$
g_{d,t}(r) = 1 + \pi\frac{\b_{d-2}\b_d}{\b_{d-1}^2}\,\frac{1-\exp(\frac{-tr}{ \pi})}{ t^2r\sin r}\,.
$$
\end{theorem}

\begin{proof}
Using $\bE\cH^{d-1}(Z_t\cap A)=(\beta_{d-1}/\beta_d)\cH^d(A)t$ for Borel sets $A\subset\SSd$, which follows from Corollary \ref{cor:HausdorffMeasure}, and Corollary \ref{covcor},
we obtain for  measurable functions $h_1,h_2:\SSd\to[0,\infty)$ that
\begin{align*}
&\bE \int_{Z_t}\int_{Z_t} h_1(x)h_2(y)\, \cH^{d-1}(\dint x)\,\cH^{d-1}(\dint y)\\[2ex]
&\qquad = \int_{\SSd}\int_{\SSd}{\pi}\frac{\b_{d-2}}{\b_d} \frac{1-\exp\big(-\frac{1}{\pi}
\ell(x,y)t\big)}{\ell(x,y)\sin(\ell(x,y))}h_1(x)h_2(y)\,\cH^d(\dint x)\,\cH^d(\dint y)\\
&\qquad\qquad + \left(\frac{\b_{d-1}}{\b_d}t\right)^2\int_{\SSd}\int_{\SSd} h_1(x)h_2(y)\,\cH^d(\dint x)\,\cH^d(\dint y)\\[2ex]
&\qquad=\left(\frac{\b_{d-1}t}{\b_d}\right)^2
\int_{\SSd}\int_{\SSd}\left(1+{\pi}\frac{\b_{d-2}\b_d}{\b_{d-1}^2} \frac{1-\exp\big(-\frac{1}{\pi}
\ell(x,y)t\big)}{ t^2\ell(x,y)\sin(\ell(x,y))}\right)\\[2mm]
&\hspace{3cm}\times h_1(x)h_2(y)\,\cH^d(\dint x)\,\cH^d(\dint y)\,,
\end{align*}
which extends to arbitrary bounded  measurable functions $(x,y)\mapsto h(x,y)$ in the usual way.

Thus, using   \eqref{eq:KMAlternativeRepresentation} and   $\mu=\bE \cH^{d-1}(Z_t)  =\beta_{d-1}t$, we find that
\begin{align*}
K_{d,t}(r) &=\frac{1}{\b_d^2}\int_{\SSd}\int_{\SSd}\left(1+{\pi}\frac{\b_{d-2}\b_d}{\b_{d-1}^2} \frac{1-\exp\big(-\frac{1}{\pi}
\ell(x,y)t\big)}{ t^2\ell(x,y)\sin(\ell(x,y))}\right)\\[2mm]
&\hspace{3cm}\times{\bf 1}(\ell(x,y)\le r)\,\cH^d(\dint y)\,\cH^d(\dint x)\,.
\end{align*}
Remark \eqref{Remna} implies that
\begin{equation*}
K_{d,t}(r) =\frac{\b_{d-1}}{\b_d}
\int_0^r \Bigg(1+\pi\frac{\b_{d-2}\b_d}{\b_{d-1}^2}\frac{1-\exp\big(-\frac{t}{\pi}\varphi\big)}{ t^2\varphi\sin\varphi}\Bigg)(\sin\varphi)^{d-1}\,\dint\varphi\,.
\end{equation*}
Consequently, $K_{d,t}(r)$ is differentiable with respect to $r$,
$$
g_{d,t}(r) = \frac{\b_d}{\b_{d-1}(\sin r)^{d-1}} K_{d,t}'(r)
= 1 + \pi\frac{\b_{d-2}\b_d}{\b_{d-1}^2}\,\frac{1-\exp\left(-\frac{1}{\pi}tr \right)}{ t^2r\sin r}\,,
$$
and thus the proof is complete.
\end{proof}

For a STIT tessellation in Euclidean space (assuming stationarity and isotropy), it was shown in \cite[Theorem 7.1]{STSecondOrder}
that the pair-correlation function is given by
$$
g_{d,t}^{{\rm euc}}(r)=1+\frac{d-1}{2}\frac{1-\exp\left(-\frac{2\beta_{d-2}}{(d-1)\beta_{d-1}}tr\right)}{t^2r^2}\,.
$$
Although the structure of this formula is similar to the one obtained in the spherical setting, there is a crucial difference.
While  $g_{d,t}^{{\rm euc}}(r)$ is jointly scaling invariant in the sense that $g_{d, t}^{{\rm euc}}(\lambda r)=g_{d,\lambda t}^{{\rm euc}}(r)$ for
all $\lambda, t,r>0$, this  property is not available for $g_{d,t}(r)$. Moreover, in the numerator of the formulas, the factor $r^2$ in the Euclidean case is replaced by $r\sin r$ in the spherical case, which yields an
additional singularity at $r=\pi$.

\begin{remark}\label{remKdtstit} {\rm
Arguing as in the isotropic case and using Corollary \ref{cor:HausdorffMeasure} and Corollary \ref{covcor}, we obtain
for a general regular direction distribution $\kappad$  that
\begin{align*}
K_{d,t}(r)&=\overline{\kappad}^2(\{(x,y)\in (\SSd)^2:\ell(x,y)\le r\})\\
& +\frac{1}{\beta_{d-1}^2}\int_{\SS_{d-1}}\int_{S}\int_{S}\frac{1-\exp\left(-\kappa(x,y)t\right)}{t^2\kappa(x,y)}
\mathbf{1}(\ell(x,y)\le r)\,\cH^{d-1}(\dint x)\, \cH^{d-1}(\dint y)\, \kappad(\dint S)\,.
\end{align*}
In view of this, the simple form of the $K$-function in the isotropic case is remarkable.
}
\end{remark}

\section{Comparison with Poisson great hypersphere tessellations}\label{sec:Comparison}

\subsection{Construction, surface intensity and capacity functional}

\begin{figure}[t]
\begin{center}
\includegraphics[width=0.5\columnwidth]{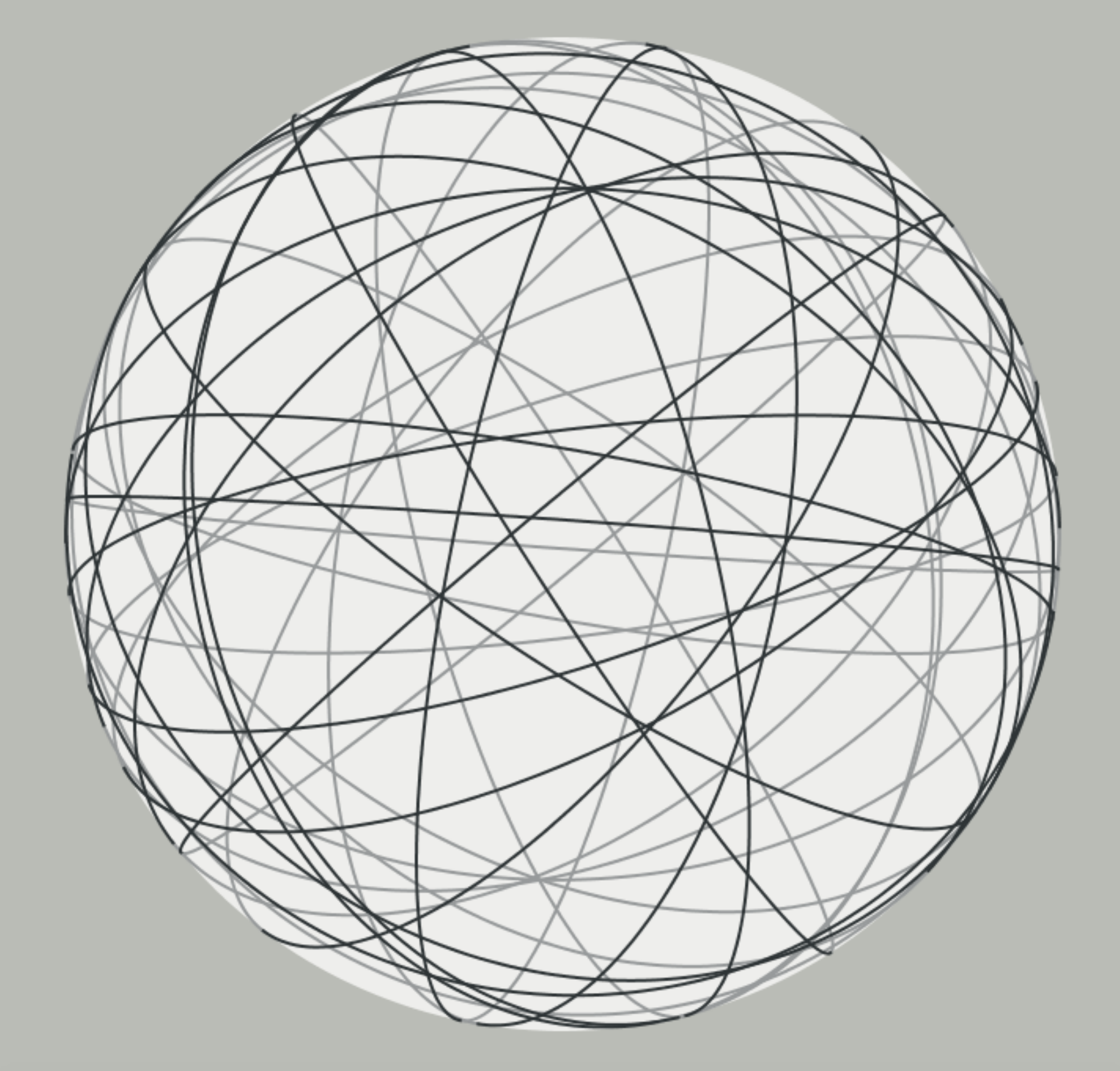}
\end{center}
\caption{Illustration of a Poisson great circle tessellation on $\SS^2$.}
\label{fig:Poisson}
\end{figure}

We now compare the pair-correlation function related to the splitting process $(Y_t)_{t\ge 0}$ at time $t$ (as discussed in Section
\ref{sec:PCF}) to the pair-correlation function of a tessellation induced by a Poisson process of great hyperspheres in $\SSd$ with intensity $t>0$.
Although we shall mainly focus on the isotropic case in this section, we present a general framework involving a general direction distribution in view of what is needed in Section \ref{sec:7TypicalCellsFacesAndDistributions}.

\medskip

To describe the model, let $\eta_t$ be a  Poisson point process on $\SSd$ with intensity measure $t\,\kappad^\circ$ (as in \cite{LP}, we consider $\eta_t$ as a random measure, but still write $x\in\eta_t$ provided that $\eta_t(\{x\})>0$), where $\kappa^\circ$ is a regular symmetric probability measure on $\SS^d$ (cf.~Section \ref{subsec21:BasicSphericalGeometry}). Let $F:\SSd\to\SS_{d-1}$ be given by $F(u):=u^\perp\cap\SSd$ and
 put $\Phi_t:=\eta_t\circ F^{-1}$. Then we denote by $\overline{Y}_t$ the tessellation of $\SSd$ induced by the Poisson process  $\Phi_t$, that is, $\overline{Y}_t$ is a Poisson great hypersphere tessellation. The associated random closed set
$$
\overline{Z}_t:=\bigcup_{u\in\eta_t}F(u)=\bigcup_{u\in\eta_t}(u^\perp\cap\SSd)
$$
is equal to the union of the cell boundaries of $\overline{Y}_t$, see Figure \ref{fig:Poisson} for an illustration. We call $\kappad^\circ$ the direction distribution and $t$ the intensity of $\eta_t$ and of $\Phi_t$ (we also
refer to the regular measure $\kappad$ on $\SS_{d-1}$ associated with $\kappad^\circ$ as the direction distribution of $\Phi_t$). Clearly, $\kappad^\circ$ is rotation invariant (that is, $\kappad^\circ=\beta_d^{-1}\cH^d$)  if and only if one (and then all) of the random objects $\eta_t$, $\Phi_t$, and $\overline{Y}_t$ is isotropic.

\medskip

The next lemma ensures that the great hyperspheres of $\Phi_t$ are almost surely in general position, which is (implicitly) used in some of the combinatorial considerations in the following.

\begin{lemma}\label{genposition}
Let $\kappa^\circ$ be a regular symmetric probability measure on $\SSd$, and let $\overline{Y}_t$ be the Poisson great hypersphere tessellation induced by a Poisson point process $\eta_t$ on $\SSd$ with intensity measure $t\kappa^\circ$ (or from a Poisson great hypersphere process on $\SS_{d-1}$
with intensity measure $t\kappa$). Then any $k$-face of $\overline{Y}_t$ is almost surely contained in precisely $d-k$ great hyperspheres of $\Phi_t$, for $k\in\{0,\ldots,d-1\}$.
Moreover, almost surely any $j$ points of $\eta_t$ are linearly independent (if considered as unit vectors), for $j\in\{1,\ldots,d+1\}$.
\end{lemma}

\begin{proof}
The assertion is a  consequence of Theorem 3.1.3 and Corollary 3.2.4 in \cite{SW} and the regularity of $\kappad^\circ$.
See \cite[Lemma 1]{S18} for a similar argument concerning hyperplane tessellations in a Euclidean space.
\end{proof}

Next we determine the intensity $\overline{\mu} := \bE\cH^{d-1}(\overline{Z}_t)$ of the  random measure $\cH^{d-1}\llcorner\overline{Z}_t$. This is the surface intensity of the tessellation $\overline{Y}_t$.

\begin{proposition}\label{newprop6.2}
The intensity measure  of the  random measure $\cH^{d-1}\llcorner\overline{Z}_t$ is given by
$$\bE\cH^{d-1}(\overline{Z}_t\cap \ \cdot\ )=t\overkappad(\ \cdot\ )
=\bE\cH^{d-1}( {Z}_t\cap \ \cdot\ )\,,
$$
where $\kappad$ is defined via \eqref{relme} and $Z_t$ is the union of the cell boundaries of a splitting tessellation $Y_t$ which is derived from the regular direction distribution $\kappad$. In particular,  $\overline{\mu}=t\beta_{d-1}=\mu$.

Moreover, $t\,\kappad$ is the intensity measure of $\Phi_t$ and
$$
\overline{U}_t(B):=\bP(\overline{Z}_t\cap B=\emptyset)=\exp\left(-t\,\kappad(\SS_{d-1}\blk B\brk)\right)\,,\quad B\in \cB(\SS^d)\,,
$$
\end{proposition}

\begin{proof}
For each $A\in\cB(\SSd)$, we obtain from Campbell's theorem \cite[Theorem 3.1.2]{SW} that
\begin{align*}
\bE\cH^{d-1}(\overline{Z}_t\cap A) &= \bE\int_{\SSd}\cH^{d-1}(u^\perp\cap A)\,\eta_t(\dint u)\\
& = t\int_{\SSd}\cH^{d-1}(u^\perp\cap A)\,\kappad^\circ(\dint u)\\
& = t\int_{\SS_{d-1}}\cH^{d-1}(S\cap A )\, \kappad(\dint S)\,,
\end{align*}
where $\kappad$ is defined via \eqref{relme}. The assertion concerning the intensity measures of $\cH^{d-1}\llcorner\overline{Z}_t$ and $\cH^{d-1}\llcorner {Z}_t$ now follows  by an application of Corollary \ref{cor:HausdorffMeasure} and definition \eqref{defoverkappad}.

By the mapping theorem (see \cite[Theorem 5.1]{LP}), $t\,\kappad$ is the intensity measure of $\Phi_t$.  Since  $\overline{Z}_t\cap B=\emptyset$ if and only if $|\{S\in\Phi_t:S\cap B\neq \emptyset\}|=0$, the remaining assertion then follows from the Poisson property
of $\Phi_t$.
\end{proof}

\begin{remark} Comparing the last assertion of Proposition \ref{newprop6.2} and Theorem \ref{prop:Capacity}, it follows that the capacity functionals of $Z_t$ and $\overline{Z}_t$
coincide for spherically convex polytopes,if $\kappad$ is general, or even for compact connected sets  if $\kappad$ is absolutely continuous with respect to
$\nu_{d-1}$. In general, of course, the capacity functionals are different (which can also be seen from Theorem \ref{prop:Capacity2}).
\end{remark}

We have just shown that the intensity $\overline{\mu}$  of the random measure $\cH^{d-1}\llcorner\overline{Z}_t$ coincides with the  intensity of the random measure $\cH^{d-1}\llcorner{Z}_t$, where $Z_t$ is the random set corresponding to a splitting tessellation $Y_t$ with regular direction distribution $\kappad$ on $\SS_{d-1}$ and $\kappad^\circ$ and $\kappad$ are related via \eqref{relme}. In particular,  $\kappad^\circ$  is rotation invariant on $\SS^d$ (that is,   $\kappad^\circ=\beta_d^{-1}\cH^d$) if and only if $\kappad$ is rotation invariant on $\SS_{d-1}$ (that is,  $\kappad=\nu_{d-1}$).

\subsection{$K$-function and pair-correlation function}

Next, we determine the $K$-function $\overline{K}_{d,t}$ of $\cH^{d-1}\llcorner\overline{Z}_t$ as well as the corresponding pair-correla\-tion function $\overline{g}_{d,t}$ in the rotation invariant case. The general case is described briefly in a subsequent remark.

\begin{proposition}\label{meanmeasureofcellboundaries}
For a Poisson great hypersphere tessellation $\overline{Y}_t$ on $\SSd$ induced by a Poisson point process $\eta_t$ on $\SSd$ with intensity measure $t \beta_d^{-1}\cH^d$ (or from a Poisson great hypersphere process
with intensity measure $t\nu_{d-1}$ on $\SS_{d-1}$), the $K$-function and
the pair-correlation function of the random measure $\cH^{d-1}\llcorner\overline{Z}_t$ are given by
$$
\overline{K}_{d,t}(r) = \frac{\b_{d-1}}{\b_d}\int_0^r(\sin\varphi)^{d-1}\,\dint\varphi+\frac{1}{ t}\frac{\b_{d-2}}{\b_{d-1}}\int_0^r(\sin\varphi)^{d-2}\,\dint\varphi\,,\qquad r\in(0,\pi)\,,
$$
and
$$
\overline{g}_{d,t}(r) = 1+\frac{\b_{d-2}\b_d}{\b_{d-1}^2}\frac{1}{ t\sin r}\,,\qquad r\in(0,\pi)\,.
$$
\end{proposition}
\begin{proof}
Starting from \eqref{eq:KMAlternativeRepresentation} we get, for $r\in(0,\pi)$,
\begin{align*}
\overline{\mu}^2\,\overline{K}_{d,t}(r) &= \bE\int_{\overline{Z}_t}\int_{\overline{Z}_t}{\bf 1}(\ell(x,y)\leq r)\,\cH^{d-1}(\dint x)\,\cH^{d-1}(\dint y)\\
&=\bE\int_{\SSd}\int_{\SSd} H(u,v;r)\,\eta_t(\dint u)\,\eta_t(\dint v)
\end{align*}
with
$$
H(u,v;r) := \int_{u^\perp\cap\SSd}\int_{v^\perp\cap\SSd}{\bf 1}(\ell(x,y)\leq r)\,\cH^{d-1}(\dint x)\,\cH^{d-1}(\dint y)\,.
$$
We split the integral into two parts according to whether $u\neq v$ or $u=v$, and denote by $\eta_{t,\neq}^2$ the set of pairs of distinct points of $\eta_t$. Using Mecke's formula for Poisson processes (see \cite[Theorem 3.2.5 and Corollary 3.2.3]{SW} or \cite[Chapter 4, Theorem 4.4]{LP}), we see that
\begin{align*}
\overline{\mu}^2\,\overline{K}_{d,t}(r) &= \bE\int_{(\SS^d)^2}H(u,v;r)\,\eta_{t,\neq}^2(\dint(u,v))+\bE\int_{\SSd}H(u,u;r)\,\eta(\dint u)\\
&=\Big(\frac{t}{\b_d}\Big)^2\int_{\SSd}\int_{\SSd}H(u,v;r)\,\cH^d(\dint u)\,\cH^d(\dint v) + \frac{t}{\b_d}\int_{\SSd}H(u,u;r)\,\cH^d(\dint u)\\
&=\Big(\frac{t}{\b_d}\Big)^2\b_{d-1}^2\int_{\SSd}\int_{\SSd}{\bf 1}(\ell(x,y)\leq r)\,\cH^d(\dint x)\,\cH^d(\dint y)\\
&\qquad\qquad + t\int_{e^\perp\cap\SS^d}\int_{e^\perp\cap\SS^d}{\bf 1}(\ell(x,y)\leq r)\,\cH^{d-1}(\dint x)\,\cH^{d-1}(\dint y)
\end{align*}
with $e=(0,\ldots,0,1)\in\SSd$ (or any other deterministic point on $\SSd$), where we used \eqref{eq:IntegrationsformelBeweis} twice and the fact that $H(u,u;\,\cdot\,)$ is independent of $u\in\SS^d$. By Remark \eqref{Remna}, we deduce that
\begin{align*}
\overline{\mu}^2\,\overline{K}_{d,t}(r) &= \Big(\frac{t}{\b_d}\Big)^2\b_{d-1}^2\b_d\b_{d-1}\int_0^r(\sin\varphi)^{d-1}\,\dint\varphi +t\b_{d-1}\b_{d-2}\int_0^r(\sin\varphi)^{d-2}\,\dint\varphi\,,
\end{align*}
and therefore
$$
\overline{K}_{d,t}(r) = \frac{\b_{d-1}}{\b_d}\int_0^r(\sin\varphi)^{d-1}\,\dint\varphi+\frac{1}{ t}\frac{\b_{d-2}}{\b_{d-1}}\int_0^r(\sin\varphi)^{d-2}\,\dint\varphi\,,\qquad r\in(0,\pi)\,.
$$
From the expression for $\overline{K}_{d,t}(r)$ we finally conclude that the pair-correlation function $\overline{g}_{d,t}(r)$ corresponding to $\overline{K}_{d,t}(r)$ is given by
$$
\overline{g}_{d,t}(r) = 1+\frac{\b_{d-2}\b_d}{\b_{d-1}^2}\frac{1}{ t\sin r}\,,\qquad r\in(0,\pi)\,.
$$
This completes the proof.
\end{proof}

For comparison with the Euclidean case, we recall from \cite[p.~299]{STSecondOrder} that
$$
\overline{g}^{\text{euc}}_{d,t}(r)=1+\frac{\beta_{d-2}}{\beta_{d-1}}\frac{1}{tr}\,.
$$
As in the case of splitting tessellations,  the factor $1/r$ (Euclidean case) replaces the factor $1/\sin r$ (spherical case) and the
scaling invariance is available only in the Euclidean setting.

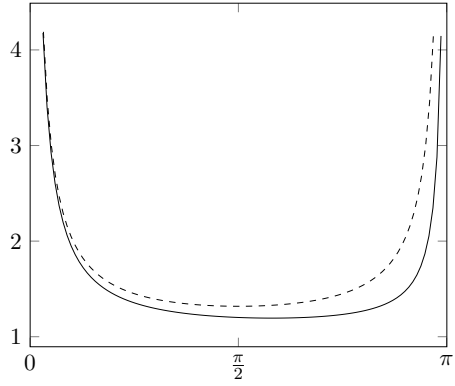
\begin{figure}[t]
\begin{center}
  \begin{tikzpicture}[scale=0.8]
    \begin{axis}[
     clip=false,
     xmin=0,xmax=3.1415,
     xtick={0,1.57,3.14},
     xticklabels={$0$, $\frac{\pi}{2}$,$\pi$}
     ]
      \addplot[domain=0.1:3.1415-0.1,samples=100,black,dashed]{1+(2/(3.1415*2))*1/(sin(deg(x)))};
      \addplot[domain=0.097:3.1415-0.0441,samples=100,black]{1+2/(x*2^2*sin(deg(x)))*(1-exp(-2*x/3.1415))}; 
    \end{axis}
  \end{tikzpicture}
\end{center}
\caption{The spherical pair-correlation functions $g_{2,2}$ (solid curve) and $\overline{g}_{2,2}$ (dashed curve).}
\label{fig}
\end{figure}

It is instructive to rewrite the $K$-function in the form
\begin{equation}\label{okdtr}
\overline{K}_{d,t}(r) = \frac{\cH^d(B(e,r))}{\b_d}+\frac{1}{ t}\frac{\cH^{d-1}(B(e,r)\cap e^\perp_0)}{\b_{d-1}}\,,
\end{equation}
where $e_0\in \SSd\cap e^\perp$ is arbitrarily chosen,
which can be interpreted by means of Slivnyak's theorem for Poisson processes on the sphere. In fact, one concludes from \cite[Corollary 2.35]{Kallenberg2} that the Palm distribution $\bP_\eta^e$ of a Poisson process on $\eta$ on $\SSd$ with distribution $\bP_\eta$ coincides with $\bP_\eta\ast\delta_{\delta_e}$, where $\delta_{\delta_e}$ corresponds to the point process that has $e$ as its single atom and $\ast$ stands for the convolution. Namely, if $e\in\SSd$ is a `typical' point of $\overline{Y_t}$ with respect to the $(d-1)$-dimensional Hausdorff measure on $\overline{Z}_t$ (in the sense of Palm distributions), the value for $\overline{K}_{d,t}(r)$ arises from two contributions. First, one has to take into account the Hausdorff measure on $B(e,r)$ of the single great hypersphere on which the typical point is located. This yields the second term, while the first term reflects the contribution in $B(e,r)$ of an independent Poisson process of great hyperspheres, which is superimposed on the great hypersphere containing the typical point. This should be compared to the corresponding $K$-function
$$
K_{d,t}(r) =  \frac{\cH^d(B(e,r))}{\b_d} + \frac{1}{ t}\frac{\b_{d-2}}{\b_{d-1}}\int_0^r \frac{\pi}{ t\varphi}\left(1-e^{-\frac{t\varphi}{\pi}}\right)\,(\sin \varphi)^{d-2}\,\dint\varphi
$$
for the splitting tessellation $Y_t$, which follows from Theorem \ref{thm:PCF}. Since $1-e^{-t}\le t$, $t \in\RR$, it follows that $K_{d,t}\le \overline{K}_{d,t}$.
For the same reason, we also have $g_{d,t}\le \overline{g}_{d,t}$. While this (in fact strict) inequality might be plausible on an intuitive level, it should be recalled that Proposition \ref{meanmeasureofcellboundaries} shows that the mean measures of the cell boundaries are the same for both types of random tessellations. It is clear that $\overline{g}_{d,t}$ is symmetric with respect to
$\pi/2$ in $(0,\pi)$, which is in contrast to the behaviour of ${g}_{d,t}$. Finally, we observe that
$$
\frac{\overline{g}_{d,t}(r)-1}{{g}_{d,t}(r)-1}=\frac{\frac{tr}{\pi}}{1-\exp\left(-\frac{tr}{\pi}\right)}\,
,\quad r>0\,,
$$
from which the asymptotic behaviour of the ratio, as $tr\to 0$ or $tr\to \infty$, can be seen.
A more detailed study of properties of Poisson great hypersphere tessellations and a comparison to
splitting tessellations is provided in the following section and otherwise will be the subject of a subsequent project.

\medskip

Finally, let us compare the pair-correlation functions $g_{d,t}$ and $\overline{g}_{d,t}$ for the particular case $d=2$.
For the Poisson great circle tessellation $\overline{Y}_t$ on $\SS^2$, we have
$$
\overline{g}_{2,t}(r) = 1 + \frac{2}{ \pi t\,\sin r}\,,\qquad r\in(0,\pi)\,,
$$
while we get
$$
g_{2,t}(r) = 1+\frac{2}{ t^2r\sin r}\left(1-e^{-\frac{rt}{\pi}}\right)\,,\qquad r\in (0,\pi)\,,
$$
for the splitting tessellation $Y_t$, for $t>0$. Figure \ref{fig} shows a plot of these two functions for the special choice $t=2$.

\begin{remark} {\rm
By the same arguments as in the isotropic case, we obtain
for a general regular direction distribution $\kappad^\circ$ and the associated direction distribution $\kappad$ that
\begin{align}
\overline{K}_{d,t}(r)&=\overline{\kappad}^2(\{(x,y)\in (\SSd)^2:\ell(x,y)\le r\})\nonumber\\
&\qquad +\frac{1}{\beta_{d-1}^2\, t}\int_{\SS_{d-1}}\int_{S}\int_{S}
\mathbf{1}(\ell(x,y)\le r)\,\cH^{d-1}(\dint x)\, \cH^{d-1}(\dint y)\, \kappad(\dint S)\,.\label{secsummand}
\end{align}
In this form, a comparison with the result obtained for ${K}_{d,t}(r)$ in Remark \ref{remKdtstit} is convenient. However,  the second summand on the right-hand side of \eqref{secsummand} can be simplified further by the argument in Remark \ref{Remna}. In fact, for this  summand we obtain the expression
$$
\frac{1}{\beta_{d-1}\, t}\int_{\SS_{d-1}}\cH^{d-1}(B(e_S,r)\cap S)\,\kappad(\dint S)=\frac{1}{t}\frac{\cH^{d-1}(B(e,r)\cap e_0^\perp)}{\beta_{d-1}}\,,
$$
where $e,e_0\in\SSd$, $e\perp e_0$, and $e_S\in S$ are arbitrarily chosen, which is independent of the direction distribution (compare \eqref{okdtr} in  the isotropic case). Hence, only the first summand is susceptible to the anisotropy of the tessellation (in contrast to what we have obtained for splitting tessellations),
and  finally we get
$$
\overline{K}_{d,t}(r)=\overline{\kappad}^2(\{(x,y)\in (\SSd)^2:\ell(x,y)\le r\})+
\frac{1}{t}\frac{\cH^{d-1}(B(e,r)\cap e_0^\perp)}{\beta_{d-1}}\,.
$$
As in the case of splitting tessellations, in the isotropic case  an even simpler expression is obtained for the first summand.
}
\end{remark}

\section{Typical cells and faces, and their distributions}\label{sec:7TypicalCellsFacesAndDistributions}

In this section, we describe additional relations between spherical splitting tessellations and tessellations
generated by Poisson processes of great hyperspheres.

\subsection{A dynamic Poisson great hypersphere tessellation process}\label{sec:7.1}

Many of our arguments and results we present below are based on a link between splitting tessellations and Poisson great hypersphere tessellations. To establish this link, we use a continu\-ous-time dynamic version of the latter model and introduce in this section a dynamic Poisson great hypersphere tessellation process which is based on a regular direction distribution $\kappad$ on $\SS_{d-1}$ and the associated
distribution $\kappad^\circ$  on $\SS^d$, which are related to each other by \eqref{relme}.

\begin{definition}
{\rm For a tessellation $T\in\TT^d$ and a great hypersphere $S\in\SS_{d-1}$ we define
$$
\otimes(S,T) := (T\setminus \{c\in T:\text{int}(c)\cap S\neq\emptyset\})\cup\bigcup_{\substack{c\in T \\ {\rm int}(c)\cap S\neq\emptyset}}\{c\cap S^+,c\cap S^-\}\in\TT^d.
$$}
\end{definition}

In other words, the tessellation $\otimes(S,T)$ is obtained from $T$ by dividing by $S$ \textit{all} cells of $T$ whose interior has non-empty intersection with the great hypersphere $S$. This operation is similar to the splitting operation $\oslash(c,S,T)$ in the context of splitting tessellations, where only the single cell $c$ gets divided by $S$ provided $S$ intersects the interior of $c$.

\medskip

We now define the dynamic Poisson great hypersphere tessellation process $(\overline{Y}_t)_{t\geq 0}$.

\begin{definition}\label{Def7.2}
{\rm The Poisson great hypersphere tessellation process $(\overline{Y}_t)_{t\geq 0}$ is the continuous-time Markov process in $\TT^d$ with initial tessellation $\overline{Y}_0=\{\SS^d\}$, whose generator $\overline{\cA}$ is given by
$$
(\overline{\cA}f)(T) = \int_{\SS_{d-1}}[f(\otimes(S,T))-f(T)]\,\kappad(\dint S)\,,\qquad T\in\TT^d\,,
$$
where $f:\TT^d\to\RR$ is any bounded and measurable function.
}
\end{definition}

The probabilistic construction which is equivalent to the generator description in Definition \ref{Def7.2} is much simpler than the one for splitting tessellations. Again we start with a single cell $\overline{Y}_0=\{\SS^d\}$ at time zero. Then, at each step, after an exponential holding time (with parameter 1) a great hypersphere $S$ with distribution $\kappa$ is chosen which dissects all cells of the current tessellation whose interior  is intersected by $S$. In particular, in
contrast to the case of a splitting tessellation, the holding time is   independent of the state of the tessellation process.

The next result clarifies why $(\overline{Y}_t)_{t\geq 0}$ is called a dynamic Poisson great hypersphere tessellation process with direction distribution $\kappad$ (or $\kappad^\circ$).

\begin{proposition}
For any $t>0$ the random tessellation $\overline{Y}_t$ has the same distribution as a Poisson great hypersphere tessellation which is derived  from a Poisson process of great hyperspheres with intensity measure $t\kappad$ on $\SS_{d-1}$ or, equivalently, from a Poisson point process on $\SSd$ with intensity measure $t\kappad^\circ$.  Moreover, $\overline{Y}_t$ is isotropic if and only if  $\kappad^\circ$ and $\kappad$ are rotation invariant.
\end{proposition}

\begin{proof}
Let $\eta$ be a Poisson process on $[0,\infty)\times \SSd$ with intensity measure ${\cH^1}\llcorner[0,\infty)\otimes\kappad^\circ$. We put
$\eta_{[a,b]}:=\eta \llcorner ([a,b]\times\cdot)$, for $0\le a\le b$. Then $\eta_{[a,b]}$ is a simple Poisson process on $\SS^d$ with intensity measure $(b-a)\kappad^\circ$, which is equal in distribution to the Poisson process $\eta_{b-a}$ on $\SS^d$ (recall that we regard $\eta_t$ and $\eta$ as simple point processes and as random collections of points).
Let $\widetilde{Y}_t:=\Tess(\eta_{[0,t]})$ be the random tessellation generated by $\{u^\perp:u\in \eta_{[0,t]}\}$. Moreover, we write $\Tess(T,\eta_{[a,b]})$, $0\le a\le b$, for the tessellation obtained from a
 great hypersphere tessellation $T$ by further intersection with great hyperspheres $u^\perp\cap\SS^d$   with $u \in \eta_{[a,b]}$. Let $t\ge 0$, $h>0$, and let $f:\TT^d\to \RR$ be bounded and measurable. Since $\eta_{[0,t+h]}=\eta_{[0,t]}+\eta_{(t,t+h]}$,
  $\eta_{[0,t]}$ and $\eta_{(t,t+h]}$ are stochastically independent, and $\eta_{(t,t+h]}$ and $\eta_{[0,h]}$ are equal in distribution,   we get
 \begin{align*}
\frac{1}{h}\left[\bE(f(\widetilde{Y}_{t+h})\mid \widetilde{Y}_{t}=T)-f(T)\right] &=\frac{1}{h}\left[\bE(f(\Tess(\eta_{[0,t+h]}))\mid \Tess(\eta_{[0,t]})=T)-f(T)\right]\\
 & =\frac{1}{h}\left[\bE(f(\Tess(T,\eta_{(t,t+h]})))-f(T)\right]\\
& =  \frac{1}{h}\left[\bE(f(\Tess(T,\eta_{[0,h]})))-f(T)\right].
 \end{align*}
Thus, for the generator $\widetilde{\cA}$ of $(\widetilde{Y}_t)_{t\ge 0}$  we obtain
\begin{align*}
(\widetilde{\cA}f)(T)&=\lim_{h\downarrow 0}\frac{1}{h}\left[\bE(f(\Tess(T,\eta_{[0,h]})))-f(T)\right]\\
&=\lim_{h\downarrow 0}\frac{1}{h}\left\{e^{-h}f(T)+he^{-h}\int_{\SS_{d-1}}f(\otimes(S,T))\,\kappad(\dint S)+o(h)-f(T)\right\}\\
&=\int_{\SS_{d-1}}f(\otimes(S,T))\, \kappad(\dint S)-f(T)=(\overline\cA f)(T)\,,
\end{align*}
which yields the equality $\widetilde{\cA}=\overline{\cA}$ , since $\kappad$ is a probability measure. From   \cite[Proposition 15.38]{Breiman} we then deduce that the  tessellation-valued processes $(\widetilde{Y}_{t})_{t\ge 0}$ and $(\overline{Y}_{t})_{t\ge 0}$ are identically distributed. Since $\eta_{[0,t]}$ and $\eta_t$ are  equal in distribution, also $\widetilde{Y}_t$ is equal in distribution to the tessellation generated by $\eta_t$, from which the first assertion follows.

The final statement about the isotropy of $\overline{Y}_t$ is clear.
\end{proof}

\subsection{Relationships for intensity measures}\label{subsec:72Relationships}

We denote by $Y_t$ a splitting tessellation with time parameter $t\geq 0$ and direction distribution $\kappad$. Then we define the random measure $\bM_t$ and
its intensity measure $\sfM_t$ on $\PP^d$ by
$$
\bM_t := \sum_{c\in Y_t}\delta_c\qquad\text{and}\qquad\sfM_t:=\bE\bM_t\,.
$$
Similarly, for a Poisson great hypersphere tessellation $\overline{Y}_t$ with  direction distribution $\kappad^\circ$ and time parameter   $t\geq 0$, we put
$$
\overline{\bM}_t := \sum_{c\in \overline{Y}_t}\delta_c\qquad\text{and}\qquad\overline{\sfM}_t:=\bE\overline{\bM}_t\,.
$$
In the following, we always assume that $\kappad$ and $\kappad^\circ$ are related via \eqref{relme}.
Repeating the proof of \cite[Theorem 1]{STBernoulli} we obtain the following result. We shall nevertheless provide the argument for completeness and to complement some details in \cite{STBernoulli} that have been left out.

\begin{theorem}\label{prop:TypicalCellMeasures}
If $t\geq 0$, then $\sfM_t=\overline{\sfM}_t$.
\end{theorem}
\begin{proof}
Let $\phi:\PP^d\to\RR$ be bounded and measurable. Then Proposition \ref{prop:Dynkin} ensures that
\begin{equation}\label{eq:ProofIMMart1}
\Sig_\phi(Y_t) - \Sig_\phi(Y_0)-\int_0^t\sum_{c\in Y_s}\int_{\SS_{d-1}[c]}[\phi(c\cap S^+)+\phi(c\cap S^-)-\phi(c)]\,\kappad(\dint S)\,\dint s\,,
\end{equation}
$t\geq 0$, is a $\filtration$-martingale. Taking expectations and using Campbell's theorem, we get
\begin{equation}\label{eq:ProofIntMeasure1}
\begin{split}
\int \phi(c)\,\sfM_t(\dint c) = \phi(\SSd)+ \int_0^t\int\int_{\SS_{d-1}[c]}&[\phi(c\cap S^+)+\phi(c\cap S^-)\\
&\qquad-\phi(c)]\,\kappad(\dint S)\,\sfM_s(\dint c)\,\dint s\,.
\end{split}
\end{equation}
Let us denote by ${\MM_{\rm bv}(\PP^d)}$ the Banach space of real-valued Borel measures on $\PP^d$ with the total variation norm $\|\cdot\|_{\rm TV}$. Then the linear operator
$$
\Gamma:\MM_{\rm bv}(\PP^d)\to\MM_{\rm bv}(\PP^d),\quad\mu\mapsto\int\int_{\SS_{d-1}}[\delta_{c\cap S^+}+\delta_{c\cap S^-}-\delta_c]\,\kappad(\dint S)\,\mu(\dint c),
$$
is bounded with operator norm $\|\Gamma\|_{\rm op}\le 3$. As observed in the proof of Proposition \ref{prop:Dynkin}, we have $\|{\sfM}_t\|_{\rm TV}={\bE} |Y_t|\le {\bE}|Y_a|=:c_a<\infty$ if $0\le t\le a$. Then \eqref{eq:ProofIntMeasure1} can equivalently be written in the form
$$
\sfM_t=\delta_{\SSd}+\int_0^t\Gamma(\sfM_s)\, \dint s\,,\qquad t\ge 0\,,
$$
and hence $\|\sfM_{t}-\sfM_r\|_{\rm TV}\le 3c_a|t-r|$ for $0\le r\le t\le a$.

\medskip

Next, we consider the dynamic Poisson great hypersphere tessellation process $(\overline{Y}_t)_{t\ge 0}$ we introduced in Section \ref{sec:7.1}. Proposition \ref{prop:Dynkin},  applied to the Markov process $(\overline{Y}_t)_{t\geq 0}$ with generator $\overline{\cA}$, yields that
\begin{align}
&\Sig_\phi(\overline{Y}_t)- \Sig_\phi(Y_0) - \int_0^t\int_{\SS_{d-1}}\sum_{\substack{c\in\overline{Y}_s\\ {\rm int}(c)\cap S\neq\emptyset}}[\phi(c\cap S^+)+\phi(c\cap S^-)-\phi(c)]\,\kappad(\dint S)\,\dint s\nonumber\\
&=\Sig_\phi(\overline{Y}_t)- \Sig_\phi(Y_0) - \int_0^t\sum_{c\in\overline{Y}_s}\int_{\SS_{d-1}[c]}[\phi(c\cap S^+)+\phi(c\cap S^-)-\phi(c)]\,\kappad(\dint S)\,\dint s\,,\label{eq:PoissonProofM}
\end{align}
$t\geq 0$, is a martingale with respect to the filtration induced by $(\overline{Y}_t)_{t\geq 0}$. {In fact, a localization procedure similar to the one used in the proof of Proposition \ref{prop:Dynkin} first shows that this process is a local martingale. In order to verify that this process is of class DL and hence a proper martingale, one also needs that the moments of $|\overline{Y}_t|$ are finite. However, this is the case since the number of cells can be expressed as a deterministic polynomial of the number of great hyperspheres of $\eta_t$ (see \cite[Lemma 8.2.1]{SW}), which in turn is a Poisson random variable having finite moments of all orders.}
In particular, we have $\|\overline{\sfM}_t\|_{\rm TV}=\bE |\overline{Y}_t|\le \bE |\overline{Y}_a|=:  \overline{c}_a<\infty$ if $0\le t\le a$. Since the right-hand side in \eqref{eq:PoissonProofM} is the same as in \eqref{eq:ProofIMMart1}, but with $Y_s$ replaced by $\overline{Y}_s$,
we also obtain that
$$
\overline{\sfM}_t=\delta_{\SSd}+\int_0^t\Gamma(\overline{\sfM}_s)\, \dint s\,,\qquad t\ge 0\,.
$$
Hence $(\sfM_t)_{t\ge0}$ and $(\overline{\sfM}_t)_{t\ge0}$ solve the same initial value problem. In the current situation  the solution of this problem is unique (see \cite[Section 1]{Deimling}), which implies the assertion. In fact,
let $0\le t\le a$ be arbitrary and put $\tilde{c}_a:=c_a+\overline{c}_a<\infty$. Then $\widetilde{\sfM}_t:=\sfM_t-\overline{\sfM}_t\in {\MM_{\rm bv}(\PP^d)} $, $t\ge 0$, satisfies
$$
\|\widetilde{\sfM}_t\|_{\rm TV}\le\int_0^t\|\Gamma(\widetilde{\sfM}_s)\|_{\rm TV}\, \dint s\le 3\tilde{c}_a t\,\qquad  0\le t\le a\,.
$$
By induction we obtain that
$$
\|\widetilde{\sfM}_t\|_{\rm TV}\le \tilde{c}_a\frac{(3t)^n}{n!}\,,\qquad 0\le t\le a\,,n\in\NN\,.
$$
Indeed, the induction step follows from the estimate
$$
\|\widetilde{\sfM}_t\|_{\rm TV}\le\int_0^t\|\Gamma(\widetilde{\sfM}_s)\|_{\rm TV}\, \dint s\le 3\int_0^t \tilde{c}_a\frac{(3s)^n}{n!}\, \dint s
=\tilde{c}_a\frac{(3t)^{n+1}}{(n+1)!}\,,\qquad 0\le t\le a\,,n\in\NN\,.
$$

Thus, taking the limit as $n\to\infty$, we conclude that $\|\widetilde{\sfM}_t\|_{\rm TV}=0$ for all $0\le t\le a$, which proves the assertion.
\end{proof}

Combining Theorem \ref{prop:TypicalCellMeasures} with Theorems \ref{thm:Expectation} and \ref{thmspaet1}, the following consequence is obtained.

\begin{corollary}\label{corconsequence}
Let $h:\SS^d\to\RR$ be bounded and measurable, and let $t\ge 0$.

If $j\in\{0,\ldots,d\}$ and $\kappad=\nu_{d-1}$, then
$$
\bE\overline{\Sigma}_j(t;h)=\frac{t^{d-j}}{(d-j)!}\frac{\cH^{d}(h)}{\beta_d}\,.
$$
For a general regular direction measure $\kappad$, this holds also for $j=d$, and
$$\bE\overline{\Sigma}_{d-1}(t;h)=t\,\overkappad(h)\,.$$
\end{corollary}

Let $T$ be a (deterministic) splitting tessellation (that is, a tessellation obtained by a successive splitting process).
By a maximal spherical face of dimension $d-1$ of $T$ we mean the separating pieces of great hyperspheres that arise in the construction of the splitting tessellation $T$. Further, by a maximal spherical face of dimension $k\in\{0,\ldots,d-2\}$ of $T$ we understand any $k$-face of a maximal spherical face of dimension $d-1$, which arises due to the insertion of this maximal spherical face. Hence any maximal spherical $k$-face belongs to a uniquely determined maximal $(d-1)$-face. We denote by $\sF_k^\ast(T)$ the collection of these maximal spherical $k$-faces of $T$, which is a subset of the space $\PP_k^d$ of $k$-dimen\-sio\-nal spherical polytopes in $\SSd$. We introduce on $\PP_k^d$ the random  measure $\bF_t^{(k)}$
and its intensity measure $\sfF_t^{(k)}$ by
$$
\bF_t^{(k)} := \sum_{f\in\sF_k^\ast(Y_t)}\delta_{f}\qquad\text{and}\qquad\sfF_t^{(k)} := \bE\bF_t^{(k)}\,.
$$
Similarly, for a great hypersphere tessellation $T$, we understand by a spherical $k$-face of $T$ any $k$-face of a cell of $T$ and denote by $\sF_k(T)$ the collection of all such faces {(each $k$-face is included only once in $\sF_k(T)$, although it arises as a $k$-face of precisely $2^{d-k}$ cells)}.  On $\PP_k^d$ we then define the measures
$$
\overline{\bF}_t^{(k)} := \sum_{f\in\sF_k(\overline{Y}_t)}\delta_f\qquad\text{and}\qquad\overline{\sfF}_t^{(k)} := \bE\overline{\bF}_t^{(k)}\,.
$$
The next theorem is the analogue of \cite[Theorem 2]{STBernoulli}. Again, the proof is basically the same as in the Euclidean case, but we give the argument for the sake of completeness.

\begin{theorem}\label{prop:FaceIntensityMeasures}
If $t\geq 0$ and $k\in\{0,\ldots,d-1\}$, then
$$
\sfF_t^{(k)}=(d-k)2^{d-k-1}\int_0^t \frac{1}{ s}\,\overline{\sfF}_s^{(k)}\,\dint s\,.
$$
\end{theorem}

\begin{proof}
Let $k\in \{0,\ldots,d-1\}$ be fixed.
For a bounded and measurable function $\psi:\PP_k^d\to\RR$ and for $f\in\PP_{d-1}^d$ we define $\phi(f) := \sum_{h}\psi(h)$, where the summation extends over all $k$-faces $h$ of $f$. Applying the martingale property in Lemma \ref{lem:BasicMartingale} to
$$
\Psi_\phi(T):=\sum_{f\in\sF_{d-1}^\ast(T)}\phi(f)\,,
$$
if $T$ is a splitting
tessellation (and  zero otherwise),  and arguing as in the proof of Proposition \ref{prop:Dynkin}, we conclude that the random process
$$
\Psi_\phi(Y_t)-\int_0^t\sum_{c\in Y_s}\int_{\SS_{d-1}[c]}\phi(c\cap S)\,\kappad(\dint S)\,\dint s\,,\qquad t\geq 0\,,
$$
is a $\filtration$-martingale. Taking expectations and using Theorem \ref{prop:TypicalCellMeasures}, we get
\begin{align}
\int \phi(f)\,\sfF_t^{(d-1)}(\dint f) &= \int_0^t\int\int_{\SS_{d-1}[c]}\phi(c\cap S)\,\kappad(\dint S)\,\sfM_s(\dint c)\,\dint s \label{exampleref1}\\
&= \int_0^t\int\int_{\SS_{d-1}[c]}\phi(c\cap S)\,\kappad(\dint S)\,\overline{\sfM}_s(\dint c)\,\dint s\,.\nonumber
\end{align}
Let $\eta_t$ be a Poisson process  on $\SS^d$, as defined in Section \ref{sec:Comparison}. Then, for any $s\in(0,t)$,
by Campbell's theorem (see \cite[Theorem 3.1.2]{SW}) and using the Mecke equation (see \cite[Chapter 4, Theorem 4.4]{LP} or \cite[Theorem 3.2.5 and Corollary 3.2.3]{SW})
we obtain that
\begin{align*}
\int\phi(f)\,\overline{\sfF}_s^{(d-1)}(\dint f)&=\bE\sum_{f\in \sF_{d-1}(\overline{Y}_s)}\phi(f)
=\bE \sum_{u\in\eta_s}\sum_{f\in \sF_{d-1}(\Tess(\eta_s))}\mathbf{1}(f\subset u^\perp)\phi(f)\\
&=\bE \int\sum_{f\in \sF_{d-1}(\Tess(\eta_s))}\mathbf{1}(f\subset u^\perp)\phi(f)\, \eta_s(\dint u)\\
&=\int_{\SSd}\bE\Big[\sum_{f\in \sF_{d-1}(\Tess(\eta_s+\delta_u))}\mathbf{1}(f\subset u^\perp)\phi(f)\Big]s\kappad^\circ(\dint u)\\
&=s\bE\int_{\SS_{d-1}}\int \mathbf{1}(\text{int}(c)\cap S\neq\emptyset)\phi(c\cap S)\,\overline{\bM}_s(\dint c)\, \kappad(\dint S)\\
&=s\int\int_{\SS_{d-1}[c]}\phi(c\cap S)\,\kappad(\dint S)\,\overline{\sfM}_s(\dint c)\,.
\end{align*}
We thus conclude that
\begin{equation}\label{eq:ProofIM2}
\int \phi(f)\,\sfF_t^{(d-1)}(\dint f) = \int_0^t \frac{1}{ s}\int \phi(f)\,\overline{\sfF}_s^{(d-1)}(\dint f)\,\dint s\,.
\end{equation}
Moreover,
$$
\int \phi(f)\,\sfF_t^{(d-1)}(\dint f) = \int \psi(h)\,\sfF_t^{(k)}(\dint h)\,,
$$
since any maximal  spherical $k$-dimensional polytope is a $k$-face of precisely one  maximal  spherical face of dimension $d-1$. On the other hand, we have the identity
$$
\int\phi(f)\,\overline{\sfF}_s^{(d-1)}(\dint f) = (d-k)2^{d-k-1}\int\psi(h)\,\overline{\sfF}_s^{(k)}(\dint h)\,,
$$
because each spherical $k$-face of $\overline{Y}_s$ is a $k$-face of precisely $(d-k)2^{d-k-1}$ spherical $(d-1)$-faces
of $\overline{Y}_s$ (see \cite[Equation (10)]{ArbeiterZaehleMosaics}). Plugging these observations into \eqref{eq:ProofIM2} leads to
$$
\int\psi(h)\,\sfF_t^{(k)}(\dint h) = (d-k)2^{d-k-1}\int_0^t \frac{1}{ s}\int\psi(h)\,\overline{\sfF}_s^{(k)}(\dint h)\,\dint s\,.
$$
Since this identity holds for all bounded and measurable $\psi$, the argument is complete.
\end{proof}

\begin{example}{\rm
Using \eqref{exampleref1} with $\phi(f)=\phi_{d-1}(f,\,\cdot\,)$,  it follows that
$$
\bE \sum_{F\in \sF_{d-1}^\ast(Y_t)}\phi_{d-1}(f,\,\cdot\,)=\int\phi_{d-1}(f,\,\cdot\,)\, \sfF_t^{(d-1)}(\dint f)=t\kappad(\,\cdot\,)\,.
$$
Now we assume that $\kappad=\nu_{d-1}$ and use \eqref{exampleref1} with $\phi(f)=\phi_{j}(f,\cdot)$ for $j\in\{0,\ldots,d-1\}$.
Then, by \eqref{stillgeneral} and Theorem \ref{thm:Expectation}, we obtain
\begin{align*}
\int \phi_j(f,\,\cdot\,)\, \sfF_t^{(d-1)}(\dint f)&=\int_0^t\int\int_{\SS_{d-1}[c]}\phi_j(c\cap S,\,\cdot\,)\, \nu_{d-1}(\dint S)\, \sfM_s(\dint c)\, \dint s\\
&=\frac{t^{d-j}}{(d-j)!}\frac{1}{\beta_d}\cH^d(\,\cdot\,)\,.
\end{align*}
This is a local version in spherical space of \cite[Proposition 3]{STSTITHigher}. The additional lower order terms in \cite[Proposition 3]{STSTITHigher} are due to the interaction with
the boundary of the observation window.}
\end{example}

\subsection{Typical  maximal spherical  faces}\label{subsec73:sphericalmaxfaces}

This section requires the concept of a typical object associated with a random tessellation of $\SSd$ (see also \cite{HugReichenbacher,Last1,Last2,RotherZaehle}). We provide a brief self-contained introduction, which does not require
the assumption of isotropy. Let $\zeta$ be a particle process in $\SSd$, that is, an element of $\sN_s(\KK^d)$. We always assume that $\bE\zeta(\KK^d)\in (0,\infty)$. To define a typical particle of $\zeta$, we require a centre function. By this we mean a measurable map
$z:\KK^d\to \SSd \cup\{o\}$ such that $z(\varrho K)=\varrho z(K)$ for $K\in\KK^d$ and $\varrho\in\SO(d+1)$. An example of such a
centre function is the map which assigns to $K$ the centre of the spherical circumsphere of $K$ if $K\in\KK^d\setminus\cup_{k=0}^{d}\SS_k $ and $o$ otherwise. As in Section \ref{sec:PCF}, we fix a reference point $e\in\SSd$ and
define $\Theta_x$ and $\nu_x$, for $x\in\SSd$. In addition, we put $\Theta_o=\SO(d+1)$ and $\nu_o=\nu$. A typical particle of $\zeta$ is defined as a random spherically convex body with distribution
$$
\bQ^e_\zeta(\ \cdot \ ):=\frac{1}{\bE\zeta(\KK^d)}\,\bE\int_{\KK^d}\int_{\Theta_{z(K)}}\mathbf{1}(\varrho^{-1} K\in \ \cdot \ )\,
\nu_{z(K)}(\dint\varrho)\, \zeta(\dint K)\,.
$$
Clearly,  $\bQ^e_{\zeta}$ is concentrated on the set $\{K\in\KK^d:z(K)\in \{e,o\}\}$. Moreover,
$$
\bQ^e_{\zeta}(A)=\bQ^{\sigma e}_{\zeta}(\sigma A)\,,\qquad A\in\cB(\KK^d)\,, \sigma\in\SO(d+1)\,.
$$
The typical particle of $\zeta$ provides shape information about the particles of $\zeta$, which is obtained
after applying uniform random rotations to the individual particles so that their centres are $e$ or $o$ (depending on whether or not the particle is a great subsphere). Note that in contrast to the Euclidean framework, which admits the choice of unique translations so as to shift the particles to a prescribed common centre, a uniform selection from all possible rotations is the appropriate approach in spherical space (and, more generally, in homogeneous spaces, see \cite{RotherZaehle}). If $A\in \cB(\KK^d)$ is rotation invariant, then we simply have $\bQ^e_\zeta(A)=(\bE\zeta(\KK^d))^{-1}\, \bE\zeta(A)$.

\medskip

Now, we introduce the distribution of several random spherical polytopes as distributions of typical particles of spherical particle processes by specifying $\zeta$. As before, $Y_t$ and $\overline{Y}_t$ are a splitting tessellation and a Poisson great hypersphere tessellation, respectively, with direction distribution $\kappad$ (at time $t>0$) and $e\in\SS^d$ is a fixed reference point. For $t>0$ and  $k\in\{0,\ldots,d-1\}$ (if this applies), the distribution of
\begin{enumerate}
\item[(1)] the typical cell of $Y_t$  is defined by
$\bQ^{(d)}_t:=\bQ^e_{\bM_t}$;
\item[(2)] the typical cell of  $\overline{Y}_t$  is defined by $\overline{\bQ}^{(d)}_t:=\bQ^e_{\overline{\bM}_t}$;
\item[(3)] the typical maximal spherical $k$-face of  $Y_t$ is defined by $ {\bQ}^{(k)}_t:=\bQ^e_{{\bF}_t^{(k)}}$;
\item[(4)] the typical  spherical $k$-face of  $\overline{Y}_t$   is defined by $ \overline{\bQ}^{(k)}_t:=\bQ^e_{\overline{\bF}_t^{(k)}}$.
\end{enumerate}

Our next result yields a representation of ${\bQ}_t^{(k)}$ in terms of $(\overline{\bQ}_s^{(k)})_{s\in [0,t]}$ and is the spherical analogue of \cite[Theorem 3]{STBernoulli}. For this we denote by $N_k(t)=\sfF_t^{(k)}(\PP_k^d)$ the expected number of  maximal spherical $k$-faces of the splitting tessellation $Y_t$ and by $\overline{N}_k(t)=
\overline{\sfF}_t^{(k)}(\PP_k^d) $   the expected number of spherical $k$-faces of the Poisson great hypersphere tessellation $\overline{Y}_t$ (as described above). The computation of $N_k(t)$ or $\overline{N}_k(t)$ is in general rather involved as considerations in \cite{ArbeiterZaehleMosaics,HugSchneider2016} indicate. For this reason we shall carry out explicit computations only for the special case $k=1$ below.

\begin{theorem}\label{thm:LinkSTITPHT}
If $t>0$, then $\bQ_t^{(d)}=\overline{\bQ}_t^{(d)}$ and
$$
\bQ_t^{(k)} = \frac{(d-k)2^{d-k-1}}{ N_k(t)}\int_0^t \frac{\overline{N}_k(s)}{ s}\overline{\bQ}_s^{(k)}\,\dint s
$$
for $k\in\{0,\ldots,d-1\}$. In particular,
$$
N_k(t) = (d-k)2^{d-k-1}\int_0^t\frac{\overline{N}_k(s)}{ s}\,\dint s\,.
$$
\end{theorem}

\begin{proof}

First, we note that $N_d(t)=\sfM_t(\PP^d)=\overline{\sfM}_t(\PP^d)=\overline{N}_d(t)$ by Theorem \ref{prop:TypicalCellMeasures}. Then, by the definition of $\bQ_t^{(d)}$, Campbell's theorem and Theorem \ref{prop:TypicalCellMeasures}, we get
\begin{align*}
\bQ_t^{(d)}(\,\cdot\,)
&=\frac{1}{ N_d(t)}\int_{\PP^d}\int_{\Theta_{z(c)}}{\bf 1}(\varrho^{-1}c\in\,\cdot\,)
\,\nu_{z(c)}(\dint\varrho)\,\sfM_t(\dint c)\\
&=\frac{1}{\overline{N}_d(t)}\int_{\PP^d}\int_{\Theta_{z(c)}}{\bf 1}(\varrho^{-1}c\in\,\cdot\,)\,\nu_{z(c)}(\dint\varrho)\,\overline{\sfM}_t(\dint c)
= \overline{\bQ}_t^{(d)}(\,\cdot\,)\,.
\end{align*}
This proves the first assertion.

\medskip

Next, we fix $k\in\{0,\ldots,d-1\}$. Then, using again Campbell's theorem and this time Theorem \ref{prop:FaceIntensityMeasures}, we get
\begin{align}
\nonumber\bQ_t^{(k)}(\,\cdot\,)
\nonumber&= \frac{1}{ N_k(t)}\int_{\PP_k^d}\int_{\Theta_{z(f)}}{\bf 1}(\varrho^{-1}f\in\,\cdot\,)\,\nu_{z(f)}(\dint \varrho)\,\sfF_t^{(k)}(\dint f)\\
\nonumber&= \frac{(d-k)2^{d-k-1}}{ N_k(t)}\int_0^t \frac{1}{ s}\int_{\PP_k^d}\int_{\Theta_{z(f)}}{\bf 1}(\varrho^{-1}f\in\,\cdot\,)\,\nu_{z(f)}(\dint \varrho)\,\overline{\sfF}_s^{(k)}(\dint f)\,\dint s\\
\nonumber&= \frac{(d-k)2^{d-k-1}}{ N_k(t)}\int_0^t \frac{\overline{N}_k(s)}{ s}\,\frac{1}{\overline{N}_k(s)}\int_{\PP_k^d}\int_{\Theta_{z(f)}}{\bf 1}(\varrho^{-1}f\in\,\cdot\,)\,\nu_{z(f)}(\dint \varrho)\,\overline{\sfF}_s^{(k)}(\dint f)\,\dint s\\
&= \frac{(d-k)2^{d-k-1}}{ N_k(t)}\int_0^t \frac{\overline{N}_k(s)}{ s}\,\overline{\bQ}_s^{(k)}(\,\cdot\,)\,\dint s\,.\label{eq:QtFirstMixture}
\end{align}
The second assertion follows from the first, since $\bQ_t^{(k)}(\PP_k^d)=\overline{\bQ}_s^{(k)}(\PP_k^d)=1$ for all $t,s>0$. This completes the proof.
\end{proof}

In the following corollaries, we write $\gamma(a,x):=\int_0^xs^{a-1}e^{-s}\,\dint s$,  where $a,x>0$, for  the   lower incomplete gamma function. The first corollary provides an explicit description of the distribution of the typical maximal spherical  segment {(that is, the typical maximal spherical  $1$-face)} of a splitting tessellation on $\SS^d$ as a mixture of distributions of typical spherical edges {(that is, typical spherical $1$-faces)} in a Poisson great hypersphere tessellation. {Besides the general formulas we treat especially the case $d=2$, which is compared to the results in the Euclidean case at the end of this section.}

\begin{corollary}\label{cor:Kanten}
If $t>0$, then
$$
\bQ_t^{(1)}= \frac{d}{2t^d+d\gamma(d-1,t)}\int_0^t s^{d-2}(2s+e^{-s})\,\overline{\bQ}_s^{(1)}\,\dint s\,.
$$
Especially, if $d=2$, then
$$
\bQ_t^{(1)} = \frac{1}{ t^2+1-e^{-t}}\int_0^t (2s+e^{-s})\,\overline{\bQ}_s^{(1)}\,\dint s\,.
$$
\end{corollary}
\begin{proof}
In view of \eqref{eq:QtFirstMixture} we have to compute the mean values $\overline{N}_1(s)$ and $N_1(t)$. Let $\xi(s)$ be the random number of spherical edges of a Poisson great hypersphere tessellation process at time $s>0$. We denote by $P(s)$ the number of great hyperspheres in $\overline{Y}_s$, which is a  Poisson distributed random variable with parameter $s$. If $P(s)\in\{0,1,\ldots,d-2\}$ then $\xi(s)=0$. If $P(s)=d-1$, then $\xi(s)=1$, while for $P(s)\geq d$, we have $\xi(s)=2d\binom{P(s)}{ d}$. Indeed, any collection of $d$ great hyperspheres generating $\overline{Y}_s$ (and which are in general position with probability one by Lemma \ref{genposition}) induces a pair of antipodal vertices. From each such vertex there are precisely $2d$ emanating edges and each edge has two vertices as its endpoints (see also \cite[Equation (16)]{HugSchneider2016} with $d$ replaced by $d+1$ and with $k=2$ there). Thus,
\begin{align}\label{eq:IntensityN1bar}
\overline{N}_1(s)=\bE\xi(s) = \frac{s^{d-1}}{(d-1)!}e^{-s}+\sum_{n=d}^\infty 2d\binom{n}{ d}\frac{s^n}{ n!}e^{-s} = \frac{s^{d-1}}{(d-1)!}(2s+e^{-s})\,.
\end{align}
To compute $N_1(t)$ we apply Theorem \ref{thm:LinkSTITPHT} to deduce that
\begin{align}\label{eq:N_1(t)}
N_1(t) &=  (d-1)2^{d-2}\int_0^t \frac{1}{ s}\,\overline{N}_1(s)\,\dint s
= \frac{2^{d-2}}{(d-2)!}\int_0^t s^{d-2}(2s+e^{-s})\,\dint s \nonumber\\
&= \frac{2^{d-2}}{(d-2)!}\left(\frac{2t^d}{ d}+\gamma(d-1,t)\right)\,.
\end{align}
Plugging this into \eqref{eq:QtFirstMixture} yields
\begin{align*}
\bQ_t^{(1)}(\,\cdot\,) = \frac{d}{ 2t^d+d\gamma(d-1,t)}\int_0^t s^{d-2}(2s+e^{-s})\,\overline{\bQ}_s^{(1)}(\,\cdot\,)\,\dint s
\end{align*}
and proves the first formula. The special representation for $d=2$ and $k=1$ follows, since in this case $\gamma(1,t)=1-e^{-t}$.
\end{proof}

The previous result can be used, in particular, to determine the expected length of the typical maximal spherical  segment of a splitting tessellation $Y_t$ on $\SS^d$. Let us emphasize that in contrast to the expected length of the typical maximal segment in a Euclidean STIT-tessellation (see \cite{MNWIsegments} for the planar case, \cite{TWN} for the spatial case and \cite{NNTW,STBernoulli} for general space dimensions), the expected length of the typical maximal spherical  segment of a splitting tessellation $Y_t$ on $\SS^d$ is universal and does not depend on the underlying direction distribution $\kappa$.

\begin{corollary}\label{cor:MeanLengthTypicalSegment}
The expected length of the typical maximal spherical  segment of a splitting tessellation $Y_t$ on $\SS^d$ with $t>0$ equals
$$
\frac{d}{d-1}\frac{2\pi t^{d-1}}{ 2t^d+d\gamma(d-1,t)}\,.
$$
Especially, if $d=2$, this reduces to
$$
\frac{2\pi t}{ t^2+1-e^{-t}}\,.
$$
\end{corollary}
\begin{proof}
Let $L(\overline{Y}_s)$ be the expected length of the typical edge in a Poisson great hypersphere tessellation of $\SS^d$ at time  $s>0$, and let $\cL(\overline{Y}_s)$ be the total edge length of $\overline{Y}_s$, that is, the sum of the lengths of all edges of $\overline{Y}_s$. The total length of all edges that are located on the same great circle equals $2\pi$. Moreover, each such great circle almost surely arises as an intersection of $d-1$ great hyperspheres. Thus,
$$
\bE\cL(\overline{Y}_s) = 2\pi\frac{s^{d-1}}{(d-1)!}e^{-s}+\sum_{N=d}^\infty {2\pi} \binom{N}{ d-1}\frac{s^N}{ N!}e^{-s} = \frac{2\pi s^{d-1}}{ (d-1)!}\,.
$$
Hence, using \eqref{eq:IntensityN1bar}, the expected length of the typical edge of $\overline{Y}_s$ equals
$$
L(\overline{Y}_s) = \frac{\bE\cL(\overline{Y}_s)}{\overline{N}_1(s)} = \frac{2\pi}{ 2s+e^{-s}}\,.
$$
Combining this with Corollary \ref{cor:Kanten} yields
\begin{align*}
L(Y_t) = \frac{2\pi d}{ 2t^d+d\gamma(d-1,t)}\int_0^t s^{d-2}\,\dint s = \frac{d}{ d-1}\frac{2\pi t^{d-1}}{ 2t^d+d\gamma(d-1,t)}\,.
\end{align*}
The formula for $d=2$ follows once again from the observation that $\gamma(1,t)=1-e^{-t}$.
\end{proof}

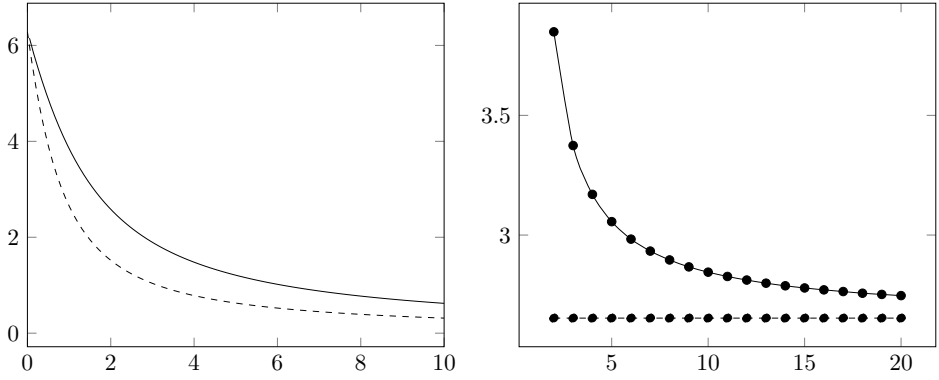
\begin{figure}[t]
\begin{center}
  \begin{tikzpicture}[scale=0.8]
    \begin{axis}[
     clip=false,
     xmin=0,xmax=10,
     xtick={0,2,4,6,8,10},
     xticklabels={$0$, $2$, $4$, $6$, $8$, $10$}
     ]
      \addplot[domain=0.05:10,samples=100,black]{2*3.1415*x/(x^2+1-exp(-x))};
      \addplot[domain=0:10,samples=100,black,dashed]{2*3.1415/(2*x+exp(-x))};
    \end{axis}
  \end{tikzpicture}
  \begin{tikzpicture}[scale=0.8]
  \begin{axis}
      \addplot[smooth,mark=*,black] plot coordinates {
          (2,3.849)
          (3,3.374)
          (4,3.170)
          (5,3.056)
          (6,2.983)
          (7,2.933)
          (8,2.896)
          (9,2.867)
          (10,2.845)
	      (11,2.827)
	      (12,2.812)
	      (13,2.799)
	      (14,2.788)
	      (15,2.779)
	      (16,2.771)
	      (17,2.764)
	      (18,2.757)
	      (19,2.752)
	      (20,2.747)
      };
      \addplot[smooth,mark=*,black,dashed] plot coordinates {
          (2,2.653)
          (3,2.653)
          (4,2.653)
          (5,2.653)
          (6,2.653)
          (7,2.653)
          (8,2.653)
          (9,2.653)
          (10,2.653)
          (11,2.653)
          (12,2.653)
          (13,2.653)
          (14,2.653)
          (15,2.653)
          (16,2.653)
          (17,2.653)
          (18,2.653)
          (19,2.653)
          (20,2.653)
            };
  \end{axis}
  \end{tikzpicture}
\end{center}
\caption{Expected length $L(Y_t)$ of the typical maximal segment of a splitting tessellation $Y_t$ on $\SS^d$ (solid curves) and expected length $L(\overline{Y}_t)$ of the typical edge of a Poisson great circle tessellation $\overline{Y}_t$ on $\SS^d$ (dashed curves). Left: Fixed dimension $d=2$ and variable time $t\in[0,10]$. Right: Fixed time $t=1$ and variable dimension $d\in\{2,3,\ldots,20\}$.}
\label{figLengthSTITvsPLT}
\end{figure}

\begin{example}{\rm Corollary \ref{cor:MeanLengthTypicalSegment} shows that the expected length of the typical maximal spherical  segment of the splitting
tessellation $Y_t$ on $\SSd$ at time $t>0$ depends on the dimension, whereas the expected length $L(\overline{Y}_t)$ of the typical edge of $\overline{Y}_t$
is independent of the dimension, that is,
$$
L(\overline{Y}_t)  = \frac{2\pi}{ 2t+e^{-t}}\,.
$$
Intuition suggest that $L(\overline{Y}_t)\le L(Y_t)$ for all $t\ge 0$ in any dimension. In fact, this follows from our formulas by basic calculus. Moreover,
by partial integration, we obtain
$$
1\le \frac{L(Y_t)}{L(\overline{Y}_t)}=\frac{d}{d-1}\frac{1+\frac{1}{2t}e^{-t}}{1+\frac{d}{d-1}\frac{1}{2t}\left(e^{-t}+\int_0^t\left(\frac{s}{t}\right)^{d-1}e^{-s}\, \dint s\right)}\le \frac{d}{d-1}\frac{1+\frac{1}{2t}e^{-t}}{1+\frac{d}{d-1}\frac{1}{2t}e^{-t}}\,,
$$
hence $L(Y_t)/L(\overline{Y}_t)\to 1$, as $d\to\infty$, for fixed $t>0$. On the other hand, we also have
$$
\frac{d}{d-1}\frac{1+\frac{1}{2t}e^{-t}}{1+\frac{d}{d-1}\frac{1}{2t}}\le\frac{L(Y_t)}{L(\overline{Y}_t)} \le
\frac{d}{d-1}\frac{1+\frac{1}{2t}e^{-t}}{1+\frac{d}{d-1}\frac{1}{2t}e^{-t}}\,,
$$
and hence
$$
\frac{1}{d}c_1\le t\left|\frac{L(Y_t)}{L(\overline{Y}_t)}-\frac{d}{d-1}\right|\le c_2\, ,\qquad t\ge 1\,,
$$
where $c_1,c_2>0$ are absolute constants. This shows that $L(Y_t)/ L(\overline{Y}_t)\to d/(d-1)$, as $t\to\infty$, for fixed dimension $d$. Moreover,
$$
\frac{d}{d-1}\frac{\pi}{t}\frac{1}{1+t^{-d}(d-1)!}\le L(Y_t)\le \frac{d}{d-1}\frac{\pi}{t}\,,\qquad t>0\,,
$$
see also Figure \ref{figLengthSTITvsPLT}. It is also easy to check from these formulas that $L(Y_t),
L(\overline{Y}_t)\to 2\pi$ as $t\to 0$.
}
\end{example}

Finally, we consider the distribution of the `birth time' of the typical maximal spherical  segment of a splitting tessellation $Y_t$ with $t>0$. To define this concept formally, we first observe that for each maximal edge $x\in \bF^{(1)}_t$ of $Y_t$ there is a unique time $s=\beta(x)$ at which $x$ is generated
by the insertion of a new maximal spherical $(d-1)$-face. This gives rise to a marked point process $\xi_t$ on $\PP^d_1$ with mark space $(0,\infty)$ given by
$$
\xi_t=\sum_{x\in \bF^{(1)}_t}\delta_{(x,\beta(x))}=\int_{\PP^d_1} \delta_{(x,\beta(x))}\, \bF^{(1)}_t(\dint x)\,.
$$
The Palm distribution of $\xi_t$ (with respect to a fixed reference point $e\in\SS^d$) is  the probability measure
$$
\bP^{e}_{\xi_t}:=\frac{1}{N_1(t)}\bE\sum_{x\in \bF^{(1)}_t}\int_{\Theta_{z(x)}}\delta_{(\varrho^{-1}x,\beta(x))}\,\nu_{z(x)}(\dint \varrho)
$$
on $\PP^d_1\times (0,\infty)$.
Then the distribution of the birth time $\beta(Y_t)$ of the typical maximal spherical segment is defined as the marginal measure of $\bP^{e}_{\xi_t}$ on the mark space $(0,\infty)$. Hence, we simply have
$$
\bP(\beta(Y_t)\in\cdot)=\frac{1}{N_1(t)}\bE \sum_{x\in\bF^{(1)}_t}
\mathbf{1}(\beta(x)\in\,\cdot\,) \,,
$$
which is independent of the chosen reference point $e\in\SS^d$.
Note that this is the same as the
marginal of the normalized intensity measure of $\xi_t$ on the mark space, that is, the mark distribution of $\xi_t$.

\begin{figure}[t]
\begin{center}
  \begin{tikzpicture}[scale=0.8]
    \begin{axis}[
     clip=false,
     xmin=0,xmax=6,
     xtick={0,2,4,6},
     xticklabels={$0$, $2$, $4$, $6$}
     ]
      \addplot[domain=0.3:6,samples=100,black,dashed]{3.1415/x};
      \addplot[domain=0.1:6,samples=100,black]{2*3.1415*x/(x^2+1-exp(-x))}; 
    \end{axis}
  \end{tikzpicture}
    \begin{tikzpicture}[scale=0.8]
      \begin{axis}[
       clip=false,
       xmin=0,xmax=3,
       xtick={0,1,2,3},
       xticklabels={$0$, $1$, $2$, $3$}
       ]
        \addplot[domain=0:3,samples=100,black,dashed]{2*x/(3*3)};
        \addplot[domain=0:3,samples=100,black]{(2*x+exp(-x))/(3*3+1-exp(-3))}; 
      \end{axis}
    \end{tikzpicture}
\end{center}
\caption{Left: The expected length of the typical maximal spherical  segment of a splitting tessellation $Y_t$ on $\SS^2$ (solid curve) and the expected length of the typical maximal segment of a STIT-tessellation in the plane (dashed curve) as a function of time $t$. Right: The density of the birth time distribution of the typical maximal spherical  segment of a splitting tessellation $Y_t$ on $\SS^2$ (solid curve) and the density of the birth time distribution of the typical maximal segment of a STIT-tessellation in the plane (dashed curve) for $t=3$ as function of $s\in (0,3)$.}
\label{figLength}
\end{figure}
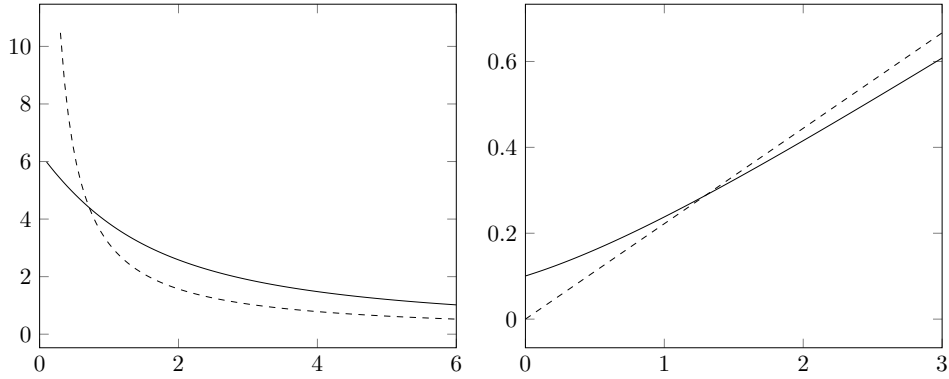

{\begin{corollary}\label{cornew2}
For each $t>0$, the random variable $\beta(Y_t)\in(0,t)$ has the density
$$
s\mapsto \frac{d s^{d-2}(2s+e^{-s})}{r 2t^d+d\gamma(d-1,t)}
$$
with respect to the Lebesgue measure on $(0,t)$. Especially, if $d=2$, this reduces to
$$
s\mapsto \frac{2s+e^{-s}}{ t^2+1-e^{-t}}\,.
$$
\end{corollary}}

{\begin{proof}
By the preceding discussion, for $s< t$ the distribution function of $\beta(Y_t)$ equals
$$
F(s)=\PP(\beta(Y_t)\le s)=\frac{1}{N_1(t)}\bE \big|\big\{x\in\bF^{(1)}_t:\beta(x)\le s\big\}\big|
= \frac{N_1(s,t)}{ N_1(t)}\,,
$$
where $N_1(t)$ is the expected number of edges of $Y_t$ and $N_1(s,t)$ is the expected number of edges whose birth time is less than or equal to $s$. By the continuous time probabilistic construction of $Y_t$ we have that $N_1(s,t)=N_1(s)$, the expected number of edges of $Y_s$. Using the formula \eqref{eq:N_1(t)} for $N_1(t)$ (and also with $t$ replaced by $s$ for $N_1(s)$) and differentiating the resulting expression with respect to $s$ completes the proof.
\end{proof}}

{The results of the Corollaries \ref{cor:MeanLengthTypicalSegment} and \ref{cornew2} might be compared with the corresponding situation for stationary and isotropic STIT-tessellations in $\RR^d$ with time parameter $t>0$. It is known from \cite[Corollary 4]{STBernoulli} that the expected length of the typical maximal segment of such a tessellation is given by
$$
\frac{d}{ d-1}\frac{d\kappa_d}{\kappa_{d-1}}\frac{1}{2t}\,.
$$
For $d=2$ this reduces to $\pi/ t$, see the left panel in Figure \ref{figLength} for a comparison with the splitting tessellation on $\SS^2$. Moreover, it is known from the discussion after \cite[Theorem 3]{STBernoulli} that for a STIT-tessellation in $\RR^d$ with time parameter $t>0$ the birth time distribution of the typical maximal segment has density on $(0,t)$ with respect to the Lebesgue measure given by
$$
s\mapsto \frac{ds^{d-1}}{ t^d}\,,
$$
which is the density of a beta distribution on $(0,t)$ with shape parameters $d$ and $1$. The right panel in Figure \ref{figLength} shows a comparison of this density with that of $\beta(Y_t)$.}

\subsection*{Acknowledgement}
This work was partially supported by the Deutsche Forschungsgemeinschaft (DFG) via RTG 2131 \textit{High-Dimensional Phenomena in Probability - Fluctuations and Discontinuity} and FOR 154  \textit{Geometry and Physics of Spatial Random Systems}.

\end{document}